\documentclass[11pt]{article}

\usepackage[margin=1in]{geometry}
\usepackage[T1]{fontenc}
\usepackage{lmodern}
\usepackage{amsmath,amssymb,amsfonts,amsthm,mathtools,mathrsfs}
\usepackage{graphicx}
\usepackage{booktabs}
\usepackage{enumitem}
\usepackage{xcolor}
\usepackage{hyperref}
\usepackage{microtype}
\usepackage{indentfirst}
\usepackage{bookmark}
\usepackage{aliascnt}
\usepackage{authblk}
\allowdisplaybreaks[2]
\hypersetup{colorlinks=true,linkcolor=blue!50!black,citecolor=blue!50!black,urlcolor=blue!50!black}

\numberwithin{equation}{section}
\newtheorem{theorem}{Theorem}[section]
\newaliascnt{lemma}{theorem}
\newtheorem{lemma}[lemma]{Lemma}
\aliascntresetthe{lemma}
\newaliascnt{proposition}{theorem}
\newtheorem{proposition}[proposition]{Proposition}
\aliascntresetthe{proposition}
\newaliascnt{corollary}{theorem}

\aliascntresetthe{corollary}
\newaliascnt{claim}{theorem}

\aliascntresetthe{claim}
\newaliascnt{definition}{theorem}
\newtheorem{definition}[definition]{Definition}
\aliascntresetthe{definition}
\newaliascnt{problem}{theorem}

\aliascntresetthe{problem}
\newaliascnt{conjecture}{theorem}
\newtheorem{conjecture}[conjecture]{Conjecture}
\aliascntresetthe{conjecture}
\theoremstyle{remark}
\newaliascnt{remark}{theorem}

\aliascntresetthe{remark}

\usepackage[nameinlink,capitalise]{cleveref}
\crefname{theorem}{Theorem}{Theorems}
\Crefname{theorem}{Theorem}{Theorems}
\crefname{lemma}{Lemma}{Lemmas}
\Crefname{lemma}{Lemma}{Lemmas}
\crefname{proposition}{Proposition}{Propositions}
\Crefname{proposition}{Proposition}{Propositions}
\crefname{corollary}{Corollary}{Corollaries}
\Crefname{corollary}{Corollary}{Corollaries}
\crefname{claim}{Claim}{Claims}
\Crefname{claim}{Claim}{Claims}
\crefname{definition}{Definition}{Definitions}
\Crefname{definition}{Definition}{Definitions}
\crefname{section}{Section}{Sections}
\Crefname{section}{Section}{Sections}
\crefname{equation}{Equation}{Equations}
\Crefname{equation}{Equation}{Equations}
\crefname{conjecture}{Conjecture}{Conjectures}
\Crefname{conjecture}{Conjecture}{Conjectures}
\crefname{problem}{Problem}{Problems}
\Crefname{problem}{Problem}{Problems}
\crefname{remark}{Remark}{Remarks}
\Crefname{remark}{Remark}{Remarks}

\DeclareMathOperator{\ex}{ex}
\DeclareMathOperator{\spex}{spex}

\title{Unbalanced spectral Tur\'an problem for color-critical graphs with prescribed large maximum degree
	\footnote{Corresponding author: Chang Liu (liuchang\_@nudt.edu.cn)}
}
\author{Chang Liu}
\affil{College of Sciences, National University of Defense Technology,\linebreak Changsha 410073, China}
\date{ }

\begin{document}
\maketitle

\begin{abstract}
Let $F$ be a connected color-critical graph with $\chi(F)=r+1\ge4$, let $S_{n,\Delta}^{(r)}=(n-\Delta)K_1\vee T(\Delta,r-1)$. We determine the graph of maximum adjacency spectral radius among all $n$-vertex $F$-free graphs with prescribed maximum degree $\Delta$. There is a constant $s_F\in[0,1)$ such that, for all sufficiently large $n$,
$\left\lceil\frac{(r-1)n}{r}\right\rceil\le \Delta\le n-\Theta(n^{s_F})$
implies that every $n$-vertex $F$-free graph $G$ with $\Delta(G)=\Delta$ satisfies $\rho(G)\le \rho\bigl(S_{n,\Delta}^{(r)}\bigr)$, with equality if and only if $G\cong S_{n,\Delta}^{(r)}$. This is the spectral counterpart of the edge theorem of [European J. Combin. 106 (2022), 103576.] and extends the clique result in \cite{paper1}. This result also provides a benchmark for unbalanced spectral Tur\'an problems arising from other extremal parameters.
\end{abstract}


\section{Introduction}

All graphs considered here are finite, simple and undirected.  Let $G=(V(G),E(G))$ be a graph with vertex set $V(G)$ and edge set $E(G)$.  Its order and size are $n(G):=|V(G)|$ and $e(G):=|E(G)|$, respectively.  For $v\in V(G)$, let $N_G(v)$ be the neighbourhood of $v$, let $d_G(v):=|N_G(v)|$ be its degree, and let $\Delta(G):=\max\{d_G(v):v\in V(G)\}$ be the maximum degree of $G$.  The adjacency matrix of an $n$-vertex graph is $A(G)=(a_{uv})_{u,v\in V(G)}$, where $a_{uv}=1$ if $uv\in E(G)$ and $a_{uv}=0$ otherwise.  We write $\rho(G)$ for the spectral radius of $G$, namely the largest eigenvalue of $A(G)$.  By the Perron--Frobenius theorem, $\rho(G)$ has a nonnegative eigenvector, called a Perron vector; if $G$ is connected, every Perron vector is positive.

For a graph $F$, a graph $G$ is called $F$-free if it contains no copy of $F$ as a subgraph.  We write $\ex(n,F)$ for the maximum number of edges in an $n$-vertex $F$-free graph and let $\operatorname{EX}(n,F)$ denote the family of graphs attaining this maximum.  Similarly, $\spex(n,F)$ denotes the maximum adjacency spectral radius among $n$-vertex $F$-free graphs, and $\operatorname{SPEX}(n,F)$ denotes the corresponding family of spectral-extremal graphs.  For vertex-disjoint graphs $G$ and $H$, their union is denoted by $G\cup H$, while their join is denoted by $G\vee H$.  Let $T(n,r)$ be the balanced complete $r$-partite graph on $n$ vertices and put $t(n,r)=e(T(n,r))$.

The classical Tur\'an-type problem asks for $\ex(n,F)$ and for a description of $\operatorname{EX}(n,F)$.  Tur\'an's theorem determines $\ex(n,K_{r+1})$, while the Erd\H{o}s--Stone--Simonovits theorem gives $\ex(n,F)=\left(1-\frac1r+o(1)\right)\frac{n^2}{2}$ for every graph $F$ with $\chi(F)=r+1\ge3$ \cite{Erdos-Stone,ErdosSimonovits,Turan}. The corresponding stability theorem asserts that near-extremal $F$-free graphs are close to $T(n,r)$ \cite{Erdos,Simonovits-1}.  Thus, in the unrestricted problem, the chromatic number determines both the extremal density and the stable structure.

The spectral Tur\'an problem replaces $e(G)$ by $\rho(G)$.  Wilf \cite{Wilf} proved that every $n$-vertex $K_{r+1}$-free graph $G$ satisfies $\rho(G)\le(1-1/r)n$, and Nikiforov \cite{Nikiforov-ESS} established a spectral Erd\H{o}s--Stone--Bollob\'as theorem, showing that the chromatic number gives the same first-order threshold in the spectral setting.  This led to the systematic study of $\operatorname{SPEX}(n,F)$ and to spectral analogues of classical Tur\'an-type theorems; see, for example, \cite{Byrne-Desai-Tait,Fang-Tait-Zhai,Nikiforov-1,Nikiforov-2}.

The edge-spectral Tur\'an problem fixes the number of edges rather than the order and asks for the largest possible spectral radius.  Brualdi and Hoffman, Stanley, and Hong, Shu and Fang obtained fundamental upper bounds for the spectral radius in terms of the number of edges \cite{Brualdi-Hoffman,Stanley,Hong-Shu-Fang}.  For $K_{r+1}$-free graphs with $m$ edges, Nikiforov proved $\rho(G)^2\le\left(1-\frac1r\right)2m$, which implies both the classical Tur\'an bound and Wilf's spectral bound \cite{Nikiforov-2002}.  Recent work has obtained edge-spectral Erd\H{o}s--Stone--Simonovits theorems and exact edge-spectral results for color-critical graphs \cite{Li-Liu-Zhang-ESS,Li-Liu-Zhang-color-critical}.  These results are particularly useful for sparse graphs, where vertex-order spectral estimates may not be sharp.

A graph $F$ is called color-critical if it contains an edge $e$ such that $\chi(F-e)<\chi(F)$; such an edge is called a critical edge.  Complete graphs and odd cycles are basic examples.  Color-critical graphs occupy a central position in all three versions of the Tur\'an problem described above.  Simonovits proved that if $\chi(F)=r+1$, then $T(n,r)$ is the unique member of $\operatorname{EX}(n,F)$ for sufficiently large $n$ \cite{Simonovits-1}.  Moon and Simonovits extended this to disjoint unions: if $F_1,\ldots,F_t$ are color-critical graphs with the same chromatic number $r+1$, then $K_{t-1}\vee T(n-t+1,r)$ is the unique extremal graph for $\bigcup_{i=1}^tF_i$ \cite{Moon,Simonovits-disjoint}.  The corresponding vertex-spectral statement was proved by Nikiforov for one color-critical graph and by Lei and Li for disjoint color-critical graphs \cite{Lei-Li,Nikiforov-2}.  Li, Liu and Zhang established the edge-spectral analogue for color-critical graphs of chromatic number at least four \cite{Li-Liu-Zhang-color-critical}; related signless Laplacian and supersaturation results appear in \cite{Fang-Li-Lin-Ma-supersaturation,Zheng-Li-Li-signless}.  These results show that color-criticality is precisely the condition under which the balanced Tur\'an construction persists across several extremal parameters.

We next turn to Tur\'an problems with a prescribed large maximum degree.  This local condition changes the extremal construction itself and leads to a non-balanced multipartite graph.  Balister, Bollob\'as, Riordan and Schelp \cite{Balister-Bollobas-Riordan-Schelp} determined the edge extremum for $C_{2k+1}$-free graphs with prescribed large maximum degree: in the range $n/2\leq\Delta(G)\leq n-k-1$, the extremal graph is complete bipartite.  Huo and Yuan \cite{Huo-Yuan} extended this result to color-critical graphs.  Their theorem is the edge counterpart of the spectral problem considered here.

\begin{theorem}[Huo and Yuan \cite{Huo-Yuan}]\label{thm:HY-result}
Let $F$ be a color-critical graph with $\chi(F)=r+1\ge4$. There is a constant $s_F\in[0,1)$ such that, for all sufficiently large $n$, every $F$-free graph $G$ with maximum degree $\Delta$ satisfying $\lceil (r-1)n/r\rceil\leq\Delta\leq n-\Theta(n^{s_F})$ satisfies $e(G)\leq\Delta (n-\Delta)+t(\Delta,r-1)$. Equality holds if and only if $G\cong (n-\Delta)K_1\vee T(\Delta,r-1)$.
\end{theorem}

The complete multipartite graph in this theorem is generally unbalanced.  The unbalancedness is forced by the exact degree condition rather than by the forbidden graph: one part has size $n-\Delta$, while the other $r-1$ parts are as equal as possible.  Liu's work on nonregular spectral extremal problems with prescribed maximum degree (see \cite{Liu-nonregular-maximum-degree}) and related results on clique counts under degree constraints further illustrate this phenomenon (see \cite{Chakraborti-Chen}).

The preceding theorem leaves open the corresponding exact spectral problem in the prescribed maximum-degree layer.  Stability alone gives only an approximate multipartite structure, while the Perron vector is sensitive to the location of the exceptional vertices and to changes that may alter the maximum degree.  The purpose of this paper is to resolve this exact spectral problem for connected color-critical graphs of chromatic number at least four.

For a fixed graph $F$ and integers $n,\Delta$, write
\begin{equation*}
	\mathfrak G_{n,\Delta}(F):=\{G: |V(G)|=n,\ G\text{ is }F\text{-free, and }\Delta(G)=\Delta\}.
\end{equation*}
The edge and spectral extremal values in this exact layer are
\begin{equation*}
	\ex_F(n,\Delta):=\max_{G\in\mathfrak G_{n,\Delta}(F)}e(G)
	\quad\text{and}\quad
	\spex_F(n,\Delta):=\max_{G\in\mathfrak G_{n,\Delta}(F)}\rho(G),
\end{equation*}
respectively.  Suppose that $\chi(F)=r+1$ and $\left\lceil\frac{(r-1)n}{r}\right\rceil\leq\Delta\leq n-1$. The lower bound is precisely the range in which the graph
\begin{equation}\label{def:root-graph}
S_{n,\Delta}^{(r)}:=(n-\Delta)K_1\vee T(\Delta,r-1)
\end{equation}
has maximum degree $\Delta$.  It is a complete $r$-partite graph with one part of size $n-\Delta$ and $r-1$ balanced remaining parts.

The previous work \cite{paper1} established the first general results for these unbalanced edge and spectral Tur\'an problems.  In the clique case it determined both extremal values in every admissible degree layer.  We record its spectral conclusion.
\begin{theorem}[\cite{paper1}]\label{thm:result-2}
	If $G$ is an $n$-vertex $K_{r+1}$-free graph with maximum degree $\Delta$ and $\left\lceil (r-1)n/r\right\rceil
	\le \Delta\le n-1$, then $\rho(G)\le\rho(S_{n,\Delta}^{(r)})$, with equality if and only if $G\cong S_{n,\Delta}^{(r)}$.
\end{theorem}

For a general $F$ with $\chi(F)=r+1$, the same paper established edge and spectral stability with respect to $S_{n,\Delta}^{(r)}$ whenever
\begin{equation*}
	a(F):=\min\{|I|: I\subseteq V(F)\text{ is independent and }\chi(F-I)\le r\}=1.
\end{equation*}
Every color-critical graph belongs to this singleton-boundary case. We use the corresponding stability theorem and spectral transfer estimate from \cite{paper1}, in the forms stated below. The boundary scale is determined by the vertex-boundary family $\partial_{c}F:=\{F-v:\ v\in V(F),\ \chi(F-v)=r\}$. We choose $s_F\in[0,1)$ so that the decomposition family of $\partial_cF$ has extremal growth $O_F(n^{1+s_F})$.

%The restriction $n-\Delta\ge C_Fn^{s_F}$ is the range in which the gain supplied by the distinguished part of $S_{n,\Delta}^{(r)}$ exceeds the possible contribution of the boundary family.

\begin{theorem}\label{thm:exact-color-critical-main}
Let $F$ be a connected color-critical graph with $\chi(F)=r+1\ge4$. There is a constant $s_F\in [0,1)$ such that, for all sufficiently large $n$, every $F$-free graph $G$ with maximum degree $\Delta$ satisfying $\left\lceil(r-1)n/r\right\rceil \le \Delta\le n-\Theta(n^{s_F})$ satisfies $\rho(G)\le\rho(S_{n,\Delta}^{(r)})$. Equality holds if and only if $G\cong S_{n,\Delta}^{(r)}$.
\end{theorem}

The principal difficulty is that a local spectral switching need not preserve the exact maximum-degree layer. We therefore maximize the spectral radius over an upper degree tail with a buffer at the upper endpoint, so that the switched graphs remain admissible. We then select a component attaining the spectral radius and root it at a vertex of maximum degree. The stability partition from \cite{paper1} is combined with Perron-vector switchings to eliminate the remaining internal edges and missing crossing edges. The resulting complete multipartite graph has maximum degree equal to the actual degree \(D\) of the upper-tail extremal graph. Strict monotonicity in \(D\) then returns the conclusion to the prescribed layer \(D=\Delta\).

\paragraph{Organization.}
\Cref{sec:def-lemma} introduces the boundary and decomposition families, recalls the standard extremal and spectral estimates, and states the stability and boundary estimates from \cite{paper1} in the notation used here. \Cref{sec:windows} defines the upper-tail problem and proves the endpoint estimates needed to obtain a degree buffer and a suitable root. In \Cref{sec:local-exactness}, the right and non-right degree ranges are treated separately; Perron-vector comparisons force the upper-tail extremal graph to be complete multipartite, after which monotonicity returns the result to the prescribed degree layer.

\paragraph{Notation and conventions.}
All graphs considered are finite, simple, and undirected. A graph $F$ is color-critical if there exists an edge $e\in E(F)$ such that $\chi(F-e)=\chi(F)-1$; such an edge is called critical. We denote by $e(G)$ the number of edges of $G$ and by $\rho(G)$ the spectral radius of its adjacency matrix. For $v\in V(G)$, let $N_G(v)$ and $d_G(v)$ denote its neighbourhood and degree, respectively, and let $\Delta(G)$ denote the maximum degree of $G$. If $G$ is connected, a Perron vector of $G$ is a positive eigenvector corresponding to $\rho(G)$. For disjoint subsets $X,Y\subseteq V(G)$, let $e_G(X,Y)$ denote the number of edges with one endpoint in $X$ and the other in $Y$. For two $n$-vertex graphs $G$ and $H$, define
\begin{equation*}
	d_{\mathrm{edit}}(G,H):=\min_{\phi:V(H)\to V(G)}
	\left|E(G)\mathbin{\triangle}\phi(E(H))\right|,
\end{equation*}
where the minimum is over all bijections $\phi:V(H)\to V(G)$ and $\phi(E(H)):=\{\phi(x)\phi(y):xy\in E(H)\}$.  For a family $\mathcal H$, set $d_{\mathrm{edit}}(G,\mathcal H):=\min_{H\in\mathcal H}d_{\mathrm{edit}}(G,H)$. We write $T(n,r)$ for the balanced complete $r$-partite graph on $n$ vertices and set $t(n,r):=e(T(n,r))$. Throughout, put $p:=r-1$ and define $S_{n,\Delta,p}:=(n-\Delta)K_1\vee T(\Delta,p)$. Thus, $S_{n,\Delta,p}$ is the same graph as $S_{n,\Delta}^{(r)}$.

\section{Preliminaries and the Upper-Tail Framework}\label{sec:exactification-prelim}

\subsection{Definitions and auxiliary results}\label{sec:def-lemma}
We record the auxiliary results needed below. Standard extremal facts are stated for reference, while the results from \cite{paper1} are given in the notation of the present paper.

\begin{theorem}[Erd\H{o}s--Stone \cite{Erdos-Stone}]\label{thm:erdos-stone}
	For all integers $p\ge2$ and $m\ge1$, and every $\varepsilon>0$, there
	exists $N$ such that every graph on $n\ge N$ vertices with at least
	$t(n,p)+\varepsilon n^2$ edges contains $T(m,p+1)$ as a subgraph.
\end{theorem}

\begin{theorem}[Erd\H{o}s--Simonovits stability theorem
	\cite{Simonovits-1,Erdos}]\label{thm:ES-stability}
	Let $\mathcal F$ be a fixed finite family of graphs with $\min_{F\in\mathcal F}\chi(F)=r+1\ge3$.  For every
	$\varepsilon>0$, there exist constants $\delta=\delta(\varepsilon,\mathcal F)>0$
	and $n_0=n_0(\varepsilon,\mathcal F)$ such that every $\mathcal F$-free graph $G$ on
	$n\ge n_0$ vertices with $e(G)\ge e(T(n,r))-\delta n^2$	differs from $T(n,r)$ in at most $\varepsilon n^2$ edges.
\end{theorem}

%\begin{lemma}[Nikiforov \cite{Nikiforov-1}]\label{lem:nikiforov-structural-alternative}
%	Let $G$ be a graph of order $n$.  Suppose that $p\ge2$,
%	$\frac1{\ln n}<c<p^{-8(p+21)(p+1)}$, and $0<\varepsilon<2^{-36}p^{-24}$. If $\rho(G)>\left(1-\frac1p-\varepsilon\right)n$, then one of the following holds:
%	\begin{itemize}
%		\item[\textnormal{(i)}]
%		$G$ contains $K_{p+1}\bigl(\lfloor c\ln n\rfloor,\dots,\lfloor c\ln n\rfloor,\lceil n^{1-\sqrt c}\rceil\bigr);$
%		\item[\textnormal{(ii)}]
%		$G$ differs from $T(n,p)$ in fewer than $\bigl(\varepsilon^{1/4}+c^{1/(8p+8)}\bigr)n^2$ edges.
%	\end{itemize}
%\end{lemma}
%
%\begin{lemma}[\cite{Desai-Kang-Li-Ni-Tait-Wang}]
%	\label{lem:spectral-stability-1}
%	Let $\mathcal F$ be a family of graphs with $\min_{F\in\mathcal F}\chi(F)=p+1$. For every $\xi>0$, there exist $\delta=\delta(\xi,\mathcal{F})>0$ and $n_1=n_1(\xi,\mathcal{F})$ such that the following holds.  If $G$ is an $\mathcal F$-free graph of order
%	$n\ge n_1$ with $\rho(G)\ge	\left(1-\frac1p-\delta\right)n$, then $G$ can be obtained from $T(n,p)$ by adding and deleting at most	$\xi n^2$ edges.
%\end{lemma}

\begin{definition}[\cite{paper1}]\label{def:independent-deletion}
	For a graph $F$ with $\chi(F)=r+1\geq 3$, define
	\begin{equation*}
		a(F):=\min\{|I|: I\subseteq V(F)\text{ is independent and }\chi(F-I)\le r\}.
	\end{equation*}
	If $a(F)=1$, the graph $F$ is called singleton-boundary.
\end{definition}

\begin{definition}[\cite{paper1}]\label{def:color-boundary-main}
	For a graph $F$ with $\chi(F)=r+1\geq 3$, define
	\begin{equation*}
		\partial_cF:=
		\{F-I:I\subseteq V(F)\text{ is independent},\ |I|=a(F),\ \chi(F-I)=r\}.
	\end{equation*}
	If $a(F)=1$, then
	\begin{equation*}
		\partial_cF=\{F-v:v\in V(F),\ \chi(F-v)=r\}.
	\end{equation*}
\end{definition}

\begin{definition}\label{def:decomposition}
	Given a family of graphs $\mathcal{F}$ with $\min_{F\in\mathcal F} \chi(F)=r+1\geq 2$ and $\mathcal{F}_r:=\{F\in\mathcal{F}: \chi(F)=r+1\}$, the decomposition family $\mathcal M(\mathcal F)$ of $\mathcal{F}$ consists of all bipartite graphs obtained from some $F\in\mathcal F_r$ by deleting $r-1$ color classes in some $(r+1)$-coloring of $V(F)$.
\end{definition}

For a graph $F$, the associated boundary decomposition family is
$\mathcal M(\partial_{c}F)$.  Following \cite{paper1,Huo-Yuan}, we measure its growth by
the following exponent.
\begin{definition}[see \cite{paper1}]\label{def:boundary-exponent}
	Let $F$ be a graph with $\chi(F)=r+1\ge3$.  Define
	\begin{equation*}
		t_F:=\inf\left\{t\in[-1,1):\,\limsup_{n\to\infty}\frac{\ex(n,\mathcal M(\partial_c F))}{n^{1+t}}<+\infty\right\}.
	\end{equation*}
	Then $t_F\in\{-1\}\cup[0,1)$, where the gap $(-1,0)$ is excluded by \cite[Theorem 2.36]{Furedi-Simonovits}.  For each fixed boundary family used below, we choose an exponent $s_F\in[0,1)$ large enough so that
	\begin{equation*}
		\max\{0,t_F\}\le s_F
		\quad\text{and}\quad
		\ex(n,\mathcal M(\partial_c F))=O_{F}(n^{1+s_F}).
	\end{equation*}
	The value of $s_F$ is not assumed to be optimal.
\end{definition}

\begin{theorem}[Erd\H{o}s \cite{Erdos}]\label{thm:erdos_decomposition}
	Let $\mathcal F$ be a family of graphs with $\chi(\mathcal F)=p+1\ge2$.	Then there exists $c=c(\mathcal F)>0$ such that
	\begin{equation*}
		\ex(n,\mathcal F)\le t(n,p)+(1+o(1))p\cdot\ex\!\left(\frac np,\mathcal M(\mathcal F)\right)+cn.
	\end{equation*}
\end{theorem}

Let $F$ be singleton-boundary with $\chi(F)=p+2\geq 3$. By \cref{def:color-boundary-main,def:boundary-exponent} and \cref{thm:erdos_decomposition}, we have, there exists a $\Lambda_F>0$, such that
\begin{equation}\label{eq:boundary-edge-cap}
\ex(n,\mathcal{M}(\partial_{c}F))\leq\Lambda_F n^{1+s_F}
\quad\text{and}\quad
\ex(n,\partial_cF)\leq t(n,p)+\Lambda_F n^{1+s_F}
\quad\text{for all sufficiently large }n.
\end{equation}
%We shall use \cref{thm:erdos_decomposition} in the following standard form.  Let $G\in\mathrm{EX}(n,\mathcal F)$, and fix a $p$-partition $V(G)=V_1\cup\cdots\cup V_p$ maximizing the number of crossing edges.  By \cref{thm:bollo-stability}, we may assume $|V_i|=\bigl(\frac1p+o(1)\bigr)n$ for all $i\in[p]$.  Moreover, after discarding $O(1)$ exceptional vertices in each class, every remaining vertex has at most $\varepsilon n$ non-neighbours in any other class. Consequently, each $G[V_i]$ is controlled by the decomposition family and
%\begin{equation}\label{eq_edros_decomp_1}
%	e(V_i)
%	\le
%	(1+o(1))
%	\ex\!\left(\frac np,\mathcal M(\mathcal F)\right)+c_i n.
%\end{equation}

\begin{theorem}[\cite{paper1}]\label{thm:unbalanced-stability-input}
Let $F$ be fixed with $\chi(F)=r+1\ge3$ and $a(F)=1$. For every $\varepsilon>0$, there are $\sigma>0$ and $n_0$ such that, for all $n\ge n_0$, all $\lceil(r-1)n/r\rceil\le\Delta\le n-1$, and every $F$-free graph $G$ with $\Delta(G)=\Delta$, the condition $\rho(G)\ge\rho(S_{n,\Delta}^{(r)})-\sigma n$ implies $d_{\mathrm{edit}}(G,S_{n,\Delta}^{(r)})\le\varepsilon n^2$.
%	Let $F$ be fixed and connected with $\chi(F)=r+1\ge3$ and $a(F)=1$.  For every $\xi>0$, there are $\sigma>0$ and $n_0$ satisfying the following.  If $n\ge n_0$, $\left\lceil\frac{r-1}{r}n\right\rceil\le\Delta\le n-1$, and $G$ is an $n$-vertex $F$-free graph with $\Delta(G)=\Delta$ and $\rho(G)\ge \rho(S_{n,\Delta}^{(r)})-\sigma n$, then $d_{\rm edit}(G,S_{n,\Delta}^{(r)})\le \xi n^2$.
\end{theorem}

The spectral transfer estimate from \cite{paper1} gives the following form.
\begin{lemma}[see \cite{paper1}]\label{lem:boundary-spectral-cap}
	Let $F$ be singleton-boundary with $\chi(F)=p+2\ge4$.  There is a constant
	$C_\partial=C_\partial(F)$ such that every $\partial_{c}F$-free graph $H$ on
	$n$ vertices satisfies
	\begin{equation*}
		\rho(H)\le \rho(T(n,p))+C_\partial n^{s_F}.
	\end{equation*}
\end{lemma}

We shall repeatedly root a graph at a maximum-degree vertex.  If $u$ has degree $\Delta$, set
\begin{equation}\label{eq:root-decomp}
        B=N_G(u),\quad A=V(G)\setminus B,\quad |B|=\Delta,\quad |A|=n-\Delta.
\end{equation}
Let $I=e_G(A)$ and $M=(n-\Delta)\Delta-e_G(A,B)$. Thus $M$ is the number of missing pairs between $A$ and $B$.

\begin{lemma}\label{lem:root-defect-main}
With the notation above, $M\ge2I$. Consequently, $e_G(A,B)+e_G(A)\le (n-\Delta)\Delta$, and $M+I\le3(M-I)$.
\end{lemma}
\begin{proof}
For each $a\in A$, the maximum-degree condition gives
\[
d_A(a)+d_B(a)\leq\Delta=|B|.
\]
Thus $|B|-d_B(a)\geq d_A(a)$. Summing over $a\in A$ yields
$M\geq2I$. Hence $M-I\geq0$, and
\[
e_G(A,B)+e_G(A)=(n-\Delta)\Delta-(M-I)\leq(n-\Delta)\Delta.
\]
Finally, $M+I=(M-I)+2I\leq3(M-I)$.
\end{proof}

Let $\pi=\{V_1,\dots,V_k\}$ be a partition of $V(G)$.  The associated
quotient matrix $B_\pi=(b_{ij})$ is defined by
\begin{equation*}
	b_{ij}:=
	\frac1{|V_i|}
	\sum_{u\in V_i}|N_{G}(u)\cap V_j|.
\end{equation*}
The partition $\pi$ is equitable if $|N_{G}(u)\cap V_j|$ is constant over all
$u\in V_i$ for every pair $i,j$.

\begin{definition}[Interlacing \cite{Brouwer-Haemers}]\label{def:interlacing}
	Let $\theta_1\ge\cdots\ge\theta_n$ and
	$\eta_1\ge\cdots\ge\eta_m$ be real sequences with $m<n$.  We say that
	$(\eta_j)_{j=1}^m$ interlaces $(\theta_i)_{i=1}^n$ if
	\begin{equation*}
		\theta_i\ge\eta_i\ge\theta_{n-m+i}
		\quad
		\text{for all } i\in[m].
	\end{equation*}
	The interlacing is tight if there exists $k\in\{0,1,\dots,m\}$ such that
	$\eta_i=\theta_i$ for $i\le k$ and
	$\eta_i=\theta_{n-m+i}$ for $i>k$.
\end{definition}

\begin{lemma}[\cite{Brouwer-Haemers}]\label{lem:quotient-interlacing}
	Let $M$ be a real symmetric nonnegative irreducible matrix, and let $B$ be
	a quotient matrix of $M$ with respect to some partition.  Then the
	eigenvalues of $B$ interlace those of $M$.  Moreover, if the partition is
	equitable, then the largest eigenvalue of $M$ equals the largest
	eigenvalue of $B$.
\end{lemma}

\begin{lemma}[\cite{Lei-Li}]\label{lem:turan-edge-estimate}
	We have $\frac{(p-1)n^2}{2p}-\frac{p}{8}	\le	t(n,p)	\le	\frac{(p-1)n^2}{2p}.$
\end{lemma}

The following estimate is standard.
\begin{lemma}\label{lem:spectral-radius-estimate}
	$0\leq \rho(T(n,p))-\frac{2t(n,p)}{n}\leq 1$. Thus, $\frac{n}{4}<\rho(T(n,p))< n$.
\end{lemma}

\begin{lemma}\label{lem:rho-join-turan-lb}
	We have
	\begin{equation*}
		\rho(S_{n,\Delta,p})
		\ge
		\frac12\left(
		\frac{p-1}{p}\Delta+
		\sqrt{
			\left(\frac{p-1}{p}\Delta\right)^2
			+4\Delta(n-\Delta)}
		\right)
		+O(\Delta^{-1}).
	\end{equation*}
\end{lemma}

\begin{proof}
	Consider the partition $V(S_{n,\Delta,p})=V\bigl((n-\Delta)K_1\bigr)\cup V\bigl(T(\Delta,p)\bigr).$
	The quotient matrix of $A(S_{n,\Delta,p})$ is
	\begin{equation*}
		B_\Delta=
		\begin{bmatrix}
			0 & \Delta \\
			n-\Delta & {2t(\Delta,p)}/{\Delta}
		\end{bmatrix}.
	\end{equation*}
	By \cref{lem:quotient-interlacing}, the Rayleigh quotient on vectors constant on the two displayed blocks gives $\rho(S_{n,\Delta,p})
	\ge \rho(B_\Delta)$, and a direct calculation gives
	\begin{equation*}
		\rho(B_\Delta)	= \frac12\left(\frac{2t(\Delta,p)}{\Delta}+\sqrt{\left(\frac{2t(\Delta,p)}{\Delta}\right)^2+4\Delta(n-\Delta)}\right).
	\end{equation*}
	The estimate follows from \cref{lem:turan-edge-estimate}.
\end{proof}

In applications, if $G$ is extremal against the candidate
$S_{n,\Delta,p}$, then \cref{lem:rho-join-turan-lb} gives
\begin{equation*}
	\rho(G)
	\ge
	\frac12
	\Bigl(\alpha+\sqrt{\alpha^2+4\beta}\Bigr)\Delta
	+O(\Delta^{-1}),
\end{equation*}
where $\alpha:=\frac{p-1}{p}$ and $\beta:=\frac{n-\Delta}{\Delta}$. If $\left\lceil\frac{p}{p+1}n\right\rceil\le\Delta\le n-C_Fn^{s_F}$ for a sufficiently large constant $C_F$, 
then $\beta\le1/p=1-\alpha$, and hence
\begin{equation*}
	\frac12
	\Bigl(\alpha+\sqrt{\alpha^2+4\beta}\Bigr)
	-(\alpha+\beta)
	=
	\frac{\beta(1-\alpha-\beta)}
	{\frac12(\alpha+\sqrt{\alpha^2+4\beta})+\alpha+\beta}
	\ge0.
\end{equation*}
Therefore
\begin{equation}\label{eq:esti-radius-lb}
	\rho(G)\ge(\alpha+\beta)\Delta+O(\Delta^{-1}).
\end{equation}

%For fixed $n-\Delta$, $\Delta$ and $p\ge1$, write
%\begin{equation*}
%	H(m_1,\ldots,m_p)
%	:=qK_1\vee K_{m_1,\ldots,m_p},
%	\quad
%	\sum_{i=1}^{p}m_i=\Delta .
%\end{equation*}

We also need elementary facts about the complete multipartite model.

\begin{lemma}\label{lem:model-spectral-comparison}
Fix $p\ge2$. The following claims hold.
\begin{enumerate}[label=\textnormal{(\roman*)}]
\item If a complete multipartite graph has positive part sizes $s_1,\ldots,s_p$, then its spectral radius is the unique positive solution of
\begin{equation}\label{eq:complete-multipartite-root}
        \sum_{i=1}^{p}\frac{s_i}{\lambda+s_i}=1 .
\end{equation}
\item For fixed $D$ and $p$, the graph $(n-D)K_1\vee K_{m_1,\ldots,m_p}$ with $\sum_{i=1}^{p} m_i= D$ has maximum spectral radius when $m_1,\ldots,m_p$ are as equal as possible, and equality occurs only for such a balanced choice, up to permutation of the parts.
\item For fixed $n$, define the function $f_{n,p}(D):=\rho(S_{n,D,p})$, where $\left\lceil\frac{p}{p+1}n\right\rceil\le D\le n-1$.
%\[
%        R_{n,p}(D):=\rho(S_{n,D,p})
%        \quad
%        \left(\left\lceil\frac{p}{p+1}n\right\rceil\le D\le n-1\right).
%\]
Then $f_{n,p}(D)$ is strictly decreasing in $D$.
\item Denote $\eta_p:=\frac{2p-1}{p(p-1)}$.
%\begin{equation}\label{eq:eta-p-def}
%        \eta_p:=\frac{2p-1}{p(p-1)}.
%\end{equation}
There are constants $\beta_p\in(0,(100p^2)^{-1}]$ and $C_p>0$ such that, uniformly for $(1-\beta_p)n\le D\le n$,
\begin{equation}\label{eq:right-quotient-expansion}
        \left|\rho(S_{n,D,p})-\rho(T(n,p))-\eta_p(n-D)\right|
        \le C_p\left(\frac{(n-D)^2}{n}+1\right).
\end{equation}
In particular,
\[
        \rho(S_{n,D,p})-\rho(T(n,p))
        =\eta_p(n-D)+O_p((n-D)^2/n)+O_p(1).
\]
\item For fixed $D$ and $p$, define $g_{D,p}(N):=\rho(S_{N,D,p})$, where $N\geq D+1$.
%\[
%        Q_{D,p}(N):=\rho(S_{N,D,p})\quad(N\ge D+1).
%\]
Then $g_{D,p}(N)$ is strictly increasing in $N$.
\end{enumerate}
\end{lemma}
\begin{proof}
\textnormal{(i)}: Let $H=K_{s_1,\ldots,s_p}$ and let $\lambda=\rho(H)$.  The Perron vector is constant on each part; write its value on the $i$-th part as $y_i$ and put
$Y:=\sum_{i=1}^{p} s_i y_i$.  The eigenvalue equations are
\begin{equation*}
	\lambda y_i=\sum_{j\in[p]\setminus\{i\}}s_jy_j=Y-s_i y_i, 
\end{equation*}
where $i\in [p]$. Thus $(\lambda+s_i)y_i=Y$.  Since $Y>0$, summing $s_iy_i=s_iY/(\lambda+s_i)$ over all $i$ gives \eqref{eq:complete-multipartite-root}.  Conversely, the positive solution of \eqref{eq:complete-multipartite-root} gives a positive eigenvector.  Uniqueness follows because $\Phi(\lambda):=\sum_{i=1}^{p}\frac{s_i}{\lambda+s_i}$ is strictly decreasing on $(0,\infty)$, with $\Phi(0)=p>1$ and
$\Phi(\lambda)\to0$ as $\lambda\to\infty$. This proves (i).

\textnormal{(ii)}: Fix $\lambda>0$.  The function $h_\lambda(x):=\frac{x}{\lambda+x}$
is strictly concave for $x>0$.  Hence
$\sum_{i=1}^{p} h_\lambda(m_i)$ is maximized, under $\sum_{i=1}^{p} m_i=D$, when
$m_1,\ldots,m_p$ differ by at most one.  Since the left side of
\eqref{eq:complete-multipartite-root} is strictly decreasing in $\lambda$, the
balanced choice gives the largest spectral radius. This proves (ii).

\textnormal{(iii)}: The part sizes of $S_{n,D,p}$ are $n-D$, $m_1(D),\ldots,m_p(D)$, $|m_i(D)-m_j(D)|\le1$, and $\sum_{i=1}^{p} m_i(D)=D$.
When $D$ increases by one, one vertex is moved from the part of size $n-D$ to one of the $p$ core parts.  In the range $D\ge \lceil pn/(p+1)\rceil$, the distinguished part has size at most every core part.  This transfer makes the vector of part sizes strictly more majorized.  By the strict concavity of $h_\lambda(x)$, the left side of \eqref{eq:complete-multipartite-root} strictly decreases for every fixed $\lambda>0$.  Hence $f_{n,p}(D):=\rho(S_{n,D,p})$ strictly decreases with respect to $D$. This proves (iii).

\textnormal{(iv)}: Consider $H_D:=(n-D)K_1\vee K_{a,\ldots,a}$, where $a=\frac{D}{p}$. Let $\lambda(D)=\rho(H_D)$.  It is determined by
\begin{equation*}
        F(\lambda,D):=
        \frac{n-D}{\lambda+n-D}
        +\frac{D}{\lambda+D/p}-1=0 .
\end{equation*}
At $D=n$, one has $\lambda(n)=\frac{p-1}{p}n=:\alpha n$.
Moreover,
\begin{equation*}
	 F_\lambda(\alpha n,n)=-\frac{1}{n},\quad F_D(\alpha n,n)=-\frac{1-\alpha^2}{\alpha n}.
\end{equation*}
The implicit-function theorem gives, uniformly for
$(1-\beta_p)n\le D\le n$ after $\beta_p>0$ is chosen sufficiently small,
\begin{equation*}
	\lambda(D)=\alpha n+\frac{1-\alpha^2}{\alpha}(n-D)+O_p((n-D)^2/n).
\end{equation*}
Since $\frac{1-\alpha^2}{\alpha}=\frac{2p-1}{p(p-1)}=\eta_p$, and replacing the equal real parts $a,\ldots,a$ by the integer Turan parts
changes the spectral radius by $O_p(1)$, \eqref{eq:right-quotient-expansion}
follows.  We fix $\beta_p\le(100p^2)^{-1}$ small enough for this uniform
implicit expansion; decreasing $\beta_p$ if necessary, the displayed bound
holds with a constant $C_p$ depending only on $p$.

\textnormal{(v)}: Finally, $S_{N+1,D,p}$ is obtained from $S_{N,D,p}$ by adding one vertex
adjacent to all vertices of the fixed $D$-vertex Turan core.  The larger graph
is connected and contains the former as a proper subgraph.  By
\Cref{lem:radius-subgraph}, $g_{D,p}(N+1)>g_{D,p}(N)$. This proves (v).
\end{proof}

\begin{lemma}[\cite{Cvetkovic-Doob}]\label{lem:radius-subgraph}
If $H$ is a subgraph of $G$, then $\rho(H)\le\rho(G)$.  If $G$ is connected and $H$ is a proper subgraph of $G$, then $\rho(H)<\rho(G)$.
\end{lemma}

\subsection{The upper-tail framework}\label{sec:windows}

The exact layer $\Delta(G)=\Delta$ is too rigid for local switching arguments, because a switching can increase the maximum degree by a bounded amount.  We therefore work first in an upward-closed degree tail and return to the exact layer only at the end.

Let $\zeta_F>0$ be a safety constant.  For a target degree $\Delta$ define
\begin{equation}\label{eq:safe-interval}
	\mathcal{I}(\Delta,\zeta_F) = [\Delta,n- \zeta_F n^{s_F}]\cap\mathbb Z
\end{equation}
and let $\mathfrak U_n(\Delta,\zeta_F;F)$ be the family of $n$-vertex $F$-free graphs $G$ such that $\Delta(G)\in\mathcal I(\Delta,\zeta_F)$.
A graph $G\in\mathfrak U_n(\Delta,\zeta_F;F)$ is called upper-tail extremal if it has maximum spectral radius in this family.  If $S_{n,\Delta}^{(r)}\in\mathfrak U_n(\Delta,\zeta_F;F)$, then every upper-tail extremal graph satisfies
\begin{equation*}
        \rho(G)\ge\rho(S_{n,\Delta}^{(r)}).
\end{equation*}

Denote $|V(F)|=f$. The constants are chosen in the following order.  First, fix
$\Lambda_F,C_\partial$ from \eqref{eq:boundary-edge-cap} and
\Cref{lem:boundary-spectral-cap}, and $\eta_p,\beta_p,C_p$ from
\Cref{lem:model-spectral-comparison}.  Choose $\varepsilon>0$ sufficiently
small so that
\begin{equation*}
	10^4p^2f^2\sqrt\varepsilon\ll1,
	\quad 10^8p^2f^2\varepsilon\ll1,
\end{equation*}
and all local inequalities in \Cref{sec:local-exactness} hold.  Set
\begin{equation}\label{eq:small-constant-choice}
	\zeta_F:=10^{10}p^3f^2\bigl(\Lambda_F+C_\partial+f+2\bigr)\varepsilon^{-2},\quad\kappa_F:=2\max\{\zeta_F,1\}.
\end{equation}
With $\kappa_F$ now fixed, the constant
$\vartheta_F=\vartheta_F(F,\kappa_F)$ is supplied by
\Cref{lem:scale-sensitive-endpoint-cap}.  We then choose
$\omega_F>\kappa_F$ sufficiently large, depending only on
$F,p,\kappa_F,\vartheta_F,C_p$, so that the endpoint comparison in
the proof of \Cref{lem:endpoint-buffer-exclusion} absorbs the fixed
$O_p(1)$ terms in \eqref{eq:right-quotient-expansion}; this also covers
$s_F=0$.  The final constant in the theorem is chosen as
\begin{equation*}
        C_F\ge \max\{\zeta_F,\vartheta_F,\kappa_F,\omega_F\}.
\end{equation*}

Let $\lceil\frac{p}{p+1}n\rceil\leq \Delta\leq n- C_F n^{s_F}$. The next estimate is the only place where the precise boundary scale enters.

\begin{lemma}\label{lem:scale-sensitive-endpoint-cap}
There is a constant $\vartheta_F=\vartheta_F(F,\kappa_F)>0$ such that the following holds.  Let $G$ be an $n$-vertex $F$-free graph with maximum degree $D$.  If $\left\lceil\frac{p}{p+1}n\right\rceil\le D\le n-1$ and $n-D\le\kappa_Fn^{s_F}$,
%\[
%\left\lceil\frac{p}{p+1}n\right\rceil\le D\le n-1
%\quad\text{and}\quad
%n-D\le\kappa_Fn^{s_F},
%\]
then
\[
\rho(G)\le \rho(S_{n,D,p})+\vartheta_Fn^{s_F}.
\]
\end{lemma}
\begin{proof}
It is enough to consider a Perron component.  If this component contains a vertex $u$ of degree $D$, use the rooted decomposition \eqref{eq:root-decomp}.  The graph $G[B]$ is $\partial_{c}F$-free, because a copy of any $F-v\in\partial_{c}F$ inside $B$ together with the root $u$ gives a copy of $F$.  Hence, by \cref{lem:boundary-spectral-cap}, we have
\begin{equation*}
	\rho(G[B])\le \rho(T(D,p))+C_\partial D^{s_F}.
\end{equation*}
Let $\boldsymbol{x}$ be an unit Perron vector of $G$, and let $\boldsymbol{x}_{A}=(x_u)_{u\in A}$ and $\boldsymbol{x}_{B}=(x_v)_{v\in B}$. Then
\begin{align*}
	\boldsymbol{x}^T A(G)\boldsymbol{x} &=
	\boldsymbol{x}_{A}^{\top} A(G[A]) \boldsymbol{x}_{A} + 2\boldsymbol{x}_{A}^{\top} A(K_{|A|,|B|})\boldsymbol{x}_{B}+\boldsymbol{x}_{B}^{\top} A(G[B]) \boldsymbol{x}_{B}\\
	&\leq (n-D-1)\|\boldsymbol x_A\|_2^2
	+2\sqrt{(n-D)D}\,\|\boldsymbol x_A\|_2\|\boldsymbol x_B\|_2
	+\rho(G[B])\|\boldsymbol x_B\|_2^2 .
\end{align*}
Therefore, $\rho(G)$ is at most the largest eigenvalue of
\begin{equation*}
	\begin{pmatrix}
		n-D-1 & \sqrt{(n-D)D}\\
		\sqrt{(n-D)D} & \rho(G[B])
	\end{pmatrix}.
\end{equation*}

Changing the first diagonal entry from $n-D-1$ to $0$ changes this largest eigenvalue by at most
\[
n-D-1\le \kappa_Fn^{s_F}.
\]
For fixed $D$, the function
\[
        h_D(x):=\frac{x+\sqrt{x^2+4(n-D)D}}{2}
\]
is increasing and $1$-Lipschitz on $[0,\infty)$.  Hence
\begin{equation*}
	\rho(G) \le	\frac{\rho(T(D,p))+\sqrt{\rho(T(D,p))^2+4(n-D)D}}{2} + O_F(n^{s_F}).
\end{equation*}
The Rayleigh quotient of $S_{n,D,p}=(n-D)K_1\vee T(D,p)$ on vectors that are constant on $(n-D)K_1$ and on $T(D,p)$ gives
\begin{equation*}
	\rho(S_{n,D,p})	\ge \frac{(2t(D,p)/D)+\sqrt{(2t(D,p)/D)^2+4(n-D)D}}{2}.
\end{equation*}
Since $\rho(T(D,p))= 2t(D,p)/D +O_p(1)$, the preceding two displays imply the claimed scale-sensitive estimate.  Rewriting the same comparison relative to $\rho(T(n,p))$ and using \eqref{eq:right-quotient-expansion} gives the equivalent Tur\'an-form estimate, with rounding terms absorbed into $\vartheta_Fn^{s_F}$.
It remains to pass from connected graphs to the Perron component.  Let $H$ be
disconnected, and let $C$ be a component with $\rho(C)=\rho(H)$.  If $C$
contains no vertex of degree $D$, then the component containing a degree-$D$
vertex has at least $D+1$ vertices, so $|C|\le n-D-1$.

Thus $\rho(C)\le |C|-1\le n-D-2=O_{F,\kappa_F}(n^{s_F})$. Otherwise $C$ contains a vertex of degree $D$.  Since
\[
        |C|-D\le n-D\le\kappa_Fn^{s_F}
\]
and $|C|\ge D+1\ge \frac{p}{p+1}n$, every resulting error is still
$O_{F,\kappa_F}(n^{s_F})$.  Applying the connected estimate to $C$ and using
\Cref{lem:model-spectral-comparison}\textnormal{(v)} gives the stated bound after
increasing $\vartheta_F$ if necessary.
\end{proof}

\begin{lemma}\label{lem:endpoint-buffer-exclusion}
There is a constant $\omega_F>\kappa_F$ such that the following holds.  Suppose $n-\Delta\ge \omega_Fn^{s_F}$, and let $G^{\star}$ be upper-tail extremal in $\mathfrak U_n(\Delta,\zeta_F;F)$.  If $D=\Delta(G^{\star})$, then $n-D\ge \kappa_F n^{s_F}$ for all sufficiently large $n$.
\end{lemma}
\begin{proof}
Assume instead that $n-D<\kappa_F n^{s_F}$. Recall that $\eta_p=\frac{2p-1}{p(p-1)}$.  By  \eqref{eq:right-quotient-expansion} (\cref{lem:model-spectral-comparison} (iv)) and \cref{lem:scale-sensitive-endpoint-cap},
\begin{equation*}
	\rho(G^{\star})\le \rho(S_{n,D,p})+\vartheta_F n^{s_F}\leq \rho(T(n,p))+(\eta_p\kappa_F+\vartheta_F+o(1))n^{s_F}.
\end{equation*}
Put $q_0:=n-\Delta$ and $q_D:=n-D$. Choose $\delta>0$ so small that the
error term in \eqref{eq:right-quotient-expansion} is at most
$\eta_pq_0/4$ whenever $q_0\leq\delta n$. If $q_0>\delta n$, then the quotient description of the multipartite model and the strict separation from the endpoint give
\[
\rho(S_{n,\Delta,p})\geq \rho(T(n,p))+c_{\delta,p}n
\]
for some $c_{\delta,p}>0$. This contradicts the preceding upper bound, since $s_F<1$. If $q_0\leq\delta n$, then \eqref{eq:right-quotient-expansion}, applied at $D$ and at $\Delta$, gives
\[
\rho(S_{n,\Delta,p})-\rho(S_{n,D,p})
\geq \eta_p(q_0-q_D)-\frac{\eta_p}{2}q_0-O_p(1).
\]
As $q_0\geq\omega_F n^{s_F}$ and $q_D<\kappa_Fn^{s_F}$, choosing $\omega_F$ sufficiently large yields
\[
\rho(S_{n,\Delta,p})>\rho(T(n,p))+(\eta_p\kappa_F+\vartheta_F+o(1))n^{s_F}
\geq \rho(G^\star),
\]
again contradicting upper-tail extremality.
\end{proof}

\begin{lemma}\label{lem:perron-component-root}
Let $G^\star\in\mathfrak U_n(\Delta,\zeta_F;F)$ be an upper-tail extremal graph, and put
$D=\Delta(G^\star)$.  Let $C^\star$ be a component of $G^\star$ such that
$\rho(C^\star)=\rho(G^\star)$.  Then $C^\star$ contains a vertex of degree $D$.
\end{lemma}
\begin{proof}
Suppose that $C^\star$ contains no vertex of degree $D$.  A different component then contains a vertex of degree $D$ and consequently has at least $D+1$ vertices.  Hence
$|C^\star|\le n-D-1$.  Since $\rho(C^\star)\le |C^\star|-1$, we obtain
$\rho(G^\star)\le n-D-2$.  On the other hand, the model $S_{n,\Delta,p}$ contains
$K_{n-\Delta,\Delta}$, and $D\ge\Delta\ge pn/(p+1)$ gives
\[
 \rho(G^\star)\ge\rho(S_{n,\Delta,p})\ge\sqrt{(n-\Delta)\Delta}
 \ge\sqrt p\,(n-\Delta)>n-D-2,
\]
where the last inequality follows from $D\ge\Delta$ and $p\ge2$.  This contradiction proves the lemma.
\end{proof}

\section{\texorpdfstring{Proof of \cref{thm:exact-color-critical-main}}{Proof of the main theorem}}\label{sec:local-exactness}

Let $F$ be a color-critical graph with $\chi(F)=p+2\geq4$, and put $f:=|V(F)|$.  Fix an arbitrary upper-tail extremal graph
$G^{\star}\in\mathfrak U_n(\Delta,\zeta_F;F)$, and write $D:=\Delta(G^{\star})$.
Because $C_F\ge\omega_F$, \cref{lem:endpoint-buffer-exclusion} yields
\begin{equation}\label{eq:actual-degree-buffer}
 n-D\ge\kappa_Fn^{s_F}\ge2\zeta_Fn^{s_F}.
\end{equation}
In the right range, the switchings increase the maximum degree by at most a constant, so \eqref{eq:actual-degree-buffer} keeps the resulting graph in $\mathfrak U_n(\Delta,\zeta_F;F)$.  In the non-right range, the linear gap $n-D>\varepsilon^2n/200$ absorbs the possible increase $O(\varepsilon^3D)$.  These are the only admissibility checks required for the local switchings.

We analyze the local structure according to the actual degree $D$: 
\begin{itemize}
	\item[\textnormal{(i)}] the right range:  $D\ge\left(1-\frac{\varepsilon^2}{200}\right)n$.
	\item[\textnormal{(ii)}] the non-right range:  $\left\lceil\frac{p}{p+1}n\right\rceil\le D<
	\left(1-\frac{\varepsilon^2}{200}\right)n$.
\end{itemize}

Both arguments begin with the same rooted decomposition.  The right range uses the endpoint expansion, whereas the non-right range is completed by a bilinear comparison with a complete multipartite envelope.
In both ranges we prove the same structural conclusion: every upper-tail
extremal graph satisfying the buffer is isomorphic to $S_{n,D,p}$.

Fix a component $C^{\star}$ with $\rho(C^{\star})=\rho(G^{\star})$. By \cref{lem:perron-component-root}, choose
$u^*\in V(C^{\star})$ with $d_{G^{\star}}(u^*)=D$ and put
\begin{equation}\label{eq:extremal-uptail-decomposition}
	B=N_{G^{\star}}(u^*),\quad A=V(G^{\star})\setminus B, \quad\text{and}\quad |A|=n-D.
\end{equation}
Define
\begin{equation*}
	A^+:=A\cap V(C^{\star}),\quad A^0:=A\setminus V(C^{\star}).
\end{equation*}
We shall use the following support facts throughout both ranges:
\begin{equation*}
B\subseteq V(C^{\star}),\quad
x_v>0\ \text{for }v\in B\cup A^+,\quad
x_u=0\ \text{for }u\in A^0,
\end{equation*}
with
\begin{equation*}
e_{G^{\star}}(A^0,B)=e_{G^{\star}}(A^0,A^+)=0.
\end{equation*}
Let $B=V_1\cup\cdots\cup V_p$ be a partition maximizing $\sum_{1\le i<j\le p}e_{G^{\star}}(V_i,V_j)$.  Then the following lemma holds.

\begin{lemma}\label{lem:B-partition}
Let $F$ be a color-critical graph with $\chi(F)=p+2\geq 4$.  With the cleaning tolerance $\varepsilon>0$ fixed in \Cref{sec:windows}, for sufficient large $n$,
\begin{align*}
&e(G^{\star}[A])+(n-D)D-e_{G^{\star}}(A,B)
  +\sum_{i=1}^{p} e(G^{\star}[V_i])
  \le 5\varepsilon^6D^2<5\varepsilon^2D^2,\\
&e(G^{\star}[B])\ge t(D,p)-\varepsilon^6D^2
  >t(D,p)-\varepsilon^2D^2,\\
&\sum_{i=1}^{p} e(G^{\star}[V_i])
  \le \varepsilon^6D^2<\varepsilon^2D^2,\\
&\max_{1\le i\le p}\left||V_i|-\frac Dp\right|
  \le 2\varepsilon^3D<2\varepsilon D.
\end{align*}
\end{lemma}
\begin{proof}
We first separate the auxiliary stability accuracy from the fixed cleaning tolerance $\varepsilon$. Apply \Cref{thm:ES-stability} to the fixed family $\partial_cF$ with edit
tolerance $\varepsilon^6$, and let
$\delta_{\mathrm{ES}}=\delta_{\mathrm{ES}}(\varepsilon,F)>0$ be the resulting
edge-deficit tolerance.  Choose an auxiliary number $\xi >0$ so small that
\begin{equation*}
	\left(\frac{p+1}{p}\right)^2\xi \le\min\left\{\frac{\varepsilon^6}{4}, \frac{\delta_{\mathrm{ES}}}{2}\right\}.
\end{equation*}
We use \Cref{thm:unbalanced-stability-input} with this fixed $\xi $.
Since $S_{n,D,p}\in\mathfrak U_n(\Delta,\zeta_F;F)$, upper-tail extremality gives
$\rho(G^{\star})\ge\rho(S_{n,D,p})$, so the spectral hypothesis of that theorem
is automatic.  After increasing $n_0$ if necessary, it follows that
$d_{\rm edit}(G^{\star},S_{n,D,p})\le\xi n^2$, and hence
\begin{equation}\label{eq:rooted-global-edge-lower1}
 e(G^{\star})\ge e(S_{n,D,p})-\xi n^2
 =(n-D)D+t(D,p)-\xi n^2.
\end{equation}

Put $I=e(G^{\star}[A])$ and
$M=(n-D)D-e_{G^{\star}}(A,B)$.  By \Cref{lem:root-defect-main},
$M\ge2I$ and $M+I\le3(M-I)$.  Since $G^{\star}$ is $F$-free,
$G^{\star}[B]$ is $\partial_cF$-free.  Hence \eqref{eq:boundary-edge-cap} gives
\begin{equation}\label{eq:rooted-local-B-upper1}
 e(G^{\star}[B])\le t(D,p)+\Lambda_FD^{1+s_F}.
\end{equation}
Since
$e(G^{\star})=e(G^{\star}[B])+(n-D)D-M+I$, equations
\eqref{eq:rooted-global-edge-lower1} and
\eqref{eq:rooted-local-B-upper1} imply
\[
 M-I\le \Lambda_FD^{1+s_F}+\xi n^2.
\]
Because $D\ge \frac{p}{p+1}n$ and $s_F<1$, we may increase $n_0$ so that
$\Lambda_FD^{1+s_F}\le \frac{\varepsilon^6}{4}D^2$.  Moreover,
the choice of $\xi $ gives
$\xi n^2\le \frac{\varepsilon^6}{4}D^2$.  Therefore
\begin{equation}\label{eq:esitimate-Miss-int-edge}
 M+I\le3(M-I)
 \le3\left(\Lambda_FD^{1+s_F}+\xi n^2\right)
 \le\frac32\varepsilon^6D^2
 <4\varepsilon^6D^2.
\end{equation}

On the other hand, $M\ge I$ gives
$e(G^{\star})\le e(G^{\star}[B])+(n-D)D$.  Together with
\eqref{eq:rooted-global-edge-lower1},
\[
 e(G^{\star}[B])\ge t(D,p)-\xi n^2
 \ge t(D,p)-\varepsilon^6D^2.
\]
The same choice also gives
$\xi n^2\le \frac{\delta_{\mathrm{ES}}}{2}D^2$, so in particular
$e(G^{\star}[B])\ge t(D,p)-\delta_{\mathrm{ES}}D^2$.  After increasing $n_0$ once more, $D\ge p n/(p+1)$ is also above the
order threshold in \Cref{thm:ES-stability}.  Applying that theorem to the
$\partial_cF$-free graph $G^{\star}[B]$ therefore yields a balanced partition
$B=U_1\cup\cdots\cup U_p$, with
$\lfloor D/p\rfloor\le |U_i|\le\lceil D/p\rceil$, such that
\[
 \sum_{i=1}^p e(G^{\star}[U_i])\le\varepsilon^6D^2.
\]
Indeed, after relabelling the comparison copy of $T(D,p)$, every edge inside
one of its $p$ parts is among the at most $\varepsilon^6D^2$ edits.

Since $V_1\cup\cdots\cup V_p$ maximizes the number of crossing edges, it
minimizes the total number of internal edges.  Hence
\[
 \sum_{i=1}^p e(G^{\star}[V_i])
 \le\sum_{i=1}^p e(G^{\star}[U_i])
 \le\varepsilon^6D^2.
\]
Together with \eqref{eq:esitimate-Miss-int-edge}, this gives
\[
 e(G^{\star}[A])+(n-D)D-e_{G^{\star}}(A,B)
 +\sum_{i=1}^p e(G^{\star}[V_i])
 =I+M+\sum_{i=1}^p e(G^{\star}[V_i])
 <5\varepsilon^6D^2.
\]

Let
$\ell:=\max_{1\le i\le p}\left||V_i|-\frac Dp\right|$ and, without loss of
generality, assume
$\left||V_1|-\frac Dp\right|=\ell$.  Then
\begin{align*}
 e(G^{\star}[B])
 &\le \sum_{1\le i<j\le p}|V_i||V_j|+\varepsilon^6D^2\nonumber\\
 &=|V_1|(D-|V_1|)
   +\frac12\left((D-|V_1|)^2-\sum_{i=2}^p|V_i|^2\right)
   +\varepsilon^6D^2\nonumber\\
 &\le |V_1|(D-|V_1|)
   +\frac{p-2}{2(p-1)}(D-|V_1|)^2
   +\varepsilon^6D^2\nonumber\\
 &=\frac{p-1}{2p}D^2
   -\frac{p}{2(p-1)}\left(|V_1|-\frac Dp\right)^2
   +\varepsilon^6D^2\nonumber\\
 &=\frac{p-1}{2p}D^2-\frac{p}{2(p-1)}\ell^2
   +\varepsilon^6D^2.
\end{align*}
Here the penultimate inequality is Cauchy--Schwarz on
$|V_2|,\ldots,|V_p|$.  By \Cref{lem:turan-edge-estimate} and the preceding lower bound on $e(G^{\star}[B])$, $e(G^{\star}[B])
\ge \frac{p-1}{2p}D^2-\frac p8-\varepsilon^6D^2$.
Combining the last two displays gives $\ell\le
\sqrt{\frac{4(p-1)}p\varepsilon^6D^2+\frac{p-1}{4}}
\le2\varepsilon^3D$.
for sufficiently large $n$.  Since the globally fixed $\varepsilon$ is
smaller than $1$, the displayed strict comparisons with the coarser
$\varepsilon^2$ and $\varepsilon$ bounds follow as well.
\end{proof}

\subsection{The right-endpoint range}\label{subsec:right}
In what follows, $G^{\star}$ is an arbitrary upper-tail extremal graph, with $\Delta(G^{\star})=D$, and the root and the sets $A,B,A^+,A^0$ are those fixed above.  In this subsection we assume that the actual degree lies in the right range,
\begin{equation*}
	\left(1-\frac{\varepsilon^2}{200}\right)n\le D\le n-\kappa_Fn^{s_F}.
\end{equation*}
We use the coarser error estimate provided by \Cref{lem:B-partition}.  Set $\alpha:=\frac{p-1}{p}$ and $\beta:=\frac{n-D}{D}$.  Then
\begin{equation}\label{eq:right-beta-bound}
 \beta=\frac{n-D}{D}
 \le \frac{\varepsilon^2/200}{1-\varepsilon^2/200}
 <\frac{\varepsilon^2}{100}<\varepsilon.
\end{equation}

\begin{lemma}\label{lem:low-degree-right}
	Let	$	L:=\left\{v\in B:\ d_B(v)\leq\Bigl(1-\frac1p-12\varepsilon\Bigr) D\right\}.	$
	Then $	|L|\leq \varepsilon D.	$
\end{lemma}
\begin{proof}
	Suppose, to the contrary, that $|L|>\varepsilon D$. Then there exists a set $S\subseteq L$ with
	$
	|S|=\lfloor \varepsilon D\rfloor.
	$
	We have
	\begin{align*}
		e(G^{\star}[B\setminus S])
		&\geq e(G^{\star}[B])-\sum_{v\in S}d_B(v)\\
		&\geq t(D,p)-\varepsilon^2 D^2-\lfloor \varepsilon D\rfloor\Bigl(1-\frac{1}{p}-12\varepsilon\Bigr) D\\
		&\geq \frac{p-1}{2p} D^2-\frac{p}{8}-\varepsilon^2 D^2-\frac{p-1}{p}\varepsilon D^2+12\varepsilon^2 D^2\\
		&>\frac{p-1}{2p}(D-\lfloor \varepsilon D \rfloor)^2+\Bigl(11-\frac{p-1}{2p}+o(1)\Bigr)\varepsilon^2 D^2\\
		&\geq t(D-\lfloor \varepsilon D\rfloor,p)+10\varepsilon^2 D^2.
	\end{align*}
	By \cref{thm:erdos-stone}, the graph $G^{\star}[B\setminus S]$ contains $T(N,p+1)$ for some sufficiently large fixed $N$. Hence there exists a graph $F-v \in \partial_{c}F$ such that $F-v \subseteq T(N,p+1)$.	Consequently, $u^*\vee (F-v)$ contains a copy of $F$, a contradiction. Therefore $|L|\leq \varepsilon D$.
\end{proof}

\begin{lemma}\label{lem:K-small-right}
	For each $i\in[p]$, define $K_i:=\{v\in V_i:\ d_{V_i}(v)\geq 2\varepsilon D\}$, and $K:=\bigcup_{i=1}^p K_i$.
	Then $|K|\leq \varepsilon D$.
\end{lemma}
\begin{proof}
	By the definition of $K$,
	$
	\sum_{i=1}^p e(V_i)=\frac12\sum_{i=1}^p \sum_{v\in V_i}d_{V_i}(v)
	\geq \frac12\sum_{i=1}^p\sum_{v\in K_i}d_{V_i}(v)
	\geq |K|\varepsilon D.
	$
	Since \cref{lem:B-partition} gives
	$\sum_{i=1}^p e(G^{\star}[V_i])\leq \varepsilon^2 D^2$, it follows that $|K|\leq \varepsilon D$.
\end{proof}

\begin{lemma}\label{lem:Vprime-degrees-right}
	Let	$ V_i':=V_i\setminus (L\cup K) $. Then, for every $i\in[p]$, every $j\in[p]\setminus\{i\}$, and every $v\in V_i'$, we have
	\begin{equation*}
		d_{V_j'}(v)\geq
		\begin{cases}
			\bigl(\frac1p-16\varepsilon\bigr)D, & \text{if } p=2,\\
			\bigl(\frac1p-20\varepsilon\bigr)D, & \text{if } p\geq 3.
		\end{cases}
	\end{equation*}
\end{lemma}
\begin{proof}
	For every $v\in V_i'$, by the definitions of $L$ and $K$, we have $d_B(v)>\Bigl(1-\frac1p-12\varepsilon\Bigr)D$, and $d_{V_i}(v)<2\varepsilon D$. Therefore,
	\begin{equation}\label{eq_sumdegree_lb_new}
		\sum_{j\in[p]\setminus\{i\}}d_{V_j'}(v)
		\geq d_B(v)-|L|-|K|-d_{V_i}(v)	> \Bigl(1-\frac1p-16\varepsilon\Bigr)D.
	\end{equation}
	
	If $p=2$, then the only possible $j$ satisfies $j\neq i$, and thus
	$
	d_{V_j'}(v)\ge d_B(v)-|L|-|K|-d_{V_i}(v)>\Bigl(1-\frac12-16\varepsilon\Bigr) D.
	$
	
	Now assume $p\geq 3$. Then
	\begin{align*}
		d_{V_j'}(v)
		&\geq d_B(v)-(D-|V_i|-|V_j|)-d_{V_i}(v)-|L|-|K|\\
		&\geq \Bigl(1-\frac{1}{p}-12\varepsilon\Bigr)D-\Bigl(1-2\Bigl(\frac{1}{p}-2\varepsilon\Bigr)\Bigr)D-4\varepsilon D\\
		&=\Bigl(\frac1p-20\varepsilon\Bigr)D.
	\end{align*}
	This proves the lemma.
\end{proof}

Set
\begin{align*}
	A_1:&=\{u\in A:\ d_B(u)\geq (1-\varepsilon) D\}, \quad A_2:=A\setminus A_1,\\
	B_1:&=\{v\in B:\ d_A(v)\geq (1-\varepsilon)|A|\}, \quad B_2:=B\setminus B_1.
\end{align*}
For $i\in\{1,2\}$ and $j\in[p]$, write $B_{i,j}:=B_i\cap V_j$.

Let $\boldsymbol x=(x_v)_{v\in V(G)}$ be a Perron vector of $G$, normalized so that $ x_{\hat u}:=\max_{v\in V(G)}x_v=1 $.
Let $ x_{\hat v}:=\max_{v\in B\setminus K}x_v $. If $\hat v\in V_i$ for some $i\in[p]$, choose $\tilde v\in V_i'$ such that
$ x_{\tilde v}=\min_{u\in V_i'}x_u $.  Also choose $\bar v\in B\setminus (L\cup K)$ such that $ x_{\bar v}=\min_{u\in B\setminus (L\cup K)}x_u $.

\begin{lemma}\label{lem:perron-entry-right}
	With the notation above, the following statements hold for sufficiently large $n$.
	\begin{itemize}
		\item[\textnormal{(i)}]
		$
		x_{\hat v}>1-\frac1p-2\varepsilon.
		$
		In particular,
		$
		d_B(\hat v)>\Bigl(1-\frac1p-12\varepsilon\Bigr) D,
		$
		and hence $\hat v\notin L$.
		
		\item[\textnormal{(ii)}]
		$
		x_{\tilde v}>1-\frac{1}{p}-30\varepsilon>\frac{2}{5}.
		$
		
		\item[\textnormal{(iii)}]
		$
		x_{\bar v}>\frac25.
		$
	\end{itemize}
\end{lemma}
\begin{proof}
	By \cref{lem:rho-join-turan-lb} and the extremality of $G^{\star}$, $\rho(G^{\star})\geq\frac12\Bigl(\alpha+\sqrt{\alpha^2+4\beta}\Bigr) D+O( D^{-1})$.
	For sufficiently large $n$, by \eqref{eq:esti-radius-lb},
	\begin{equation}\label{eq:rho_lower_xi2}
		\rho(G^{\star})\geq (\alpha+\beta-\varepsilon^2) D.
	\end{equation}
	
	\medskip
	\noindent\textbf{Proof of \textnormal{(i)}.}
	Since $x_{\hat u}=1$, the eigenequation at $\hat u$ yields
	$
	\rho(G^{\star})=\sum_{w\in N_{G^{\star}}(\hat u)}x_w.
	$
	Among the neighbours of $\hat u$, at most $|A|+|K|$ lie in $A\cup K$, and every other neighbour lies in $B\setminus K$, where the Perron entry is at most $x_{\hat v}$. Thus
	$
	\rho(G^{\star})\leq |A|+|K|+\bigl(D-|A|-|K|\bigr)x_{\hat v}.
	$
	Hence
	\begin{equation*}
		x_{\hat v}\geq
		\frac{\rho(G^{\star})-|A|-|K|}{D-|A|-|K|}
		\geq
		\frac{(\alpha+\beta-\varepsilon^2)D-\beta D-\varepsilon D}{(1-\beta-\varepsilon)D}
		=
		\frac{\alpha-\varepsilon-\varepsilon^2}{1-\beta-\varepsilon}.
	\end{equation*}
	Since $1-\beta-\varepsilon<1$, we obtain
	$
	x_{\hat v}\geq \alpha-\varepsilon-\varepsilon^2>\alpha-2\varepsilon=1-\frac1p-2\varepsilon
	$
	for sufficiently small $\varepsilon$.
	
	Next, the eigenequation at $\hat v$ gives
	$
	\rho(G^{\star})x_{\hat v}\leq |A|+|K|+d_B(\hat v)x_{\hat v},
	$
	and therefore
	$
	d_B(\hat v)\geq \rho(G^{\star})-\frac{|A|+|K|}{x_{\hat v}}
	\geq
	(\alpha+\beta-\varepsilon^2)D-\frac{(\beta+\varepsilon)D}{\alpha-2\varepsilon}.
	$
	Since $\alpha\geq \frac12$, for sufficiently small $\varepsilon$,
	$
	\frac{1}{\alpha-2\varepsilon}\leq \frac{1}{1/2-2\varepsilon}\leq 2+12\varepsilon.
	$
	Thus
	\begin{align*}
		d_B(\hat v)
		&\geq
		\Bigl(\alpha+\beta-\varepsilon^2-(2+12\varepsilon)(\beta+\varepsilon)\Bigr) D>
		\Bigl(\alpha-3\varepsilon-25\varepsilon^2\Bigr)D\\
		&>\Bigl(\alpha-12\varepsilon\Bigr)D	= \Bigl(1-\frac1p-12\varepsilon\Bigr)D.
	\end{align*}
	Hence $\hat v\notin L$.
	
	\medskip
	\noindent\textbf{Proof of \textnormal{(ii)}.}
	Without loss of generality, assume that $\hat v\in V_1'$, and let
	$
	x_{\tilde v}=\min_{u\in V_1'}x_u.
	$
	Set
	$
	L_0:=L\setminus K.
	$
	Since $L_0\subseteq B\setminus K$, every vertex in $L_0$ has Perron entry at most $x_{\hat v}$. Therefore the eigenequation at $\hat v$ yields
	\begin{equation}\label{eq:hatv_upper_refined_new}
		\rho(G^{\star})x_{\hat v}
		\leq
		|A|+|K|
		+\bigl(|L_0|+d_{V_1'}(\hat v)\bigr)x_{\hat v}
		+\sum_{i=2}^p\sum_{u\in V_i'}x_u.
	\end{equation}
	On the other hand, the eigenequation at $\tilde v$ gives
	\begin{equation}\label{eq:tildev_lower_refined_new}
		\rho(G^{\star})x_{\tilde v}
		\geq
		\sum_{i=2}^p\sum_{u\in N_{V_i'}(\tilde v)}x_u.
	\end{equation}
	Subtracting~\eqref{eq:tildev_lower_refined_new} from~\eqref{eq:hatv_upper_refined_new}, we obtain
	\begin{equation*}
		\rho(G^{\star})(x_{\hat v}-x_{\tilde v})
		\leq
		|A|+|K|
		+
		\Bigl(
		|L_0|+d_{V_1'}(\hat v)+\sum_{i=2}^p\bigl(|V_i'|-d_{V_i'}(\tilde v)\bigr)
		\Bigr)x_{\hat v}.
	\end{equation*}
	Now $|L_0|\leq |L|\leq \varepsilon D$, and $d_{V_1'}(\hat v)\leq d_{V_1}(\hat v)<2\varepsilon D$, because $\hat v\notin K$. Moreover, by~\eqref{eq_sumdegree_lb_new},
	$
	\sum_{i=2}^p d_{V_i'}(\tilde v)>
	\Bigl(1-\frac1p-16\varepsilon\Bigr) D.
	$
	On the other hand,
	\begin{equation*}
		\sum_{i=2}^p |V_i'|
		\leq \sum_{i=2}^p |V_i|
		=
		D-|V_1|
		\leq \Bigl(1-\frac1p+2\varepsilon\Bigr)D.
	\end{equation*}
	Hence,
	$
	\sum_{i=2}^p\bigl(|V_i'|-d_{V_i'}(\tilde v)\bigr)\leq 18\varepsilon D.
	$
	Substituting these estimates, we obtain
	\begin{equation*}
		\rho(G^{\star})(x_{\hat v}-x_{\tilde v})
		\leq
		(\beta+\varepsilon) D+21\varepsilon D x_{\hat v}.
	\end{equation*}
	Using~\eqref{eq:rho_lower_xi2}, this implies
	$
	x_{\tilde v}
	\geq
	\Bigl(1-\frac{21\varepsilon}{\alpha+\beta-\varepsilon^2}\Bigr)x_{\hat v}
	-\frac{\beta+\varepsilon}{\alpha+\beta-\varepsilon^2}.
	$
	By part~\textnormal{(i)},
	$
	x_{\hat v}>\alpha-2\varepsilon.
	$
	Therefore,
	\begin{equation*}
		x_{\tilde v}>
		\Bigl(1-\frac{21\varepsilon}{\alpha+\beta-\varepsilon^2}\Bigr)(\alpha-2\varepsilon)
		-\frac{2\varepsilon}{\alpha+\beta-\varepsilon^2}>\alpha-2\varepsilon-\frac{21\alpha\varepsilon+2\varepsilon}{\alpha-\varepsilon^2}.		
	\end{equation*}
	Since $\alpha\geq \frac12$, for sufficiently small $\varepsilon$,
	$\frac{\alpha}{\alpha-\varepsilon^2}<\frac{22}{21}$, and $\frac{1}{\alpha-\varepsilon^2}<\frac52$.
	Hence
	$
	\frac{21\alpha\varepsilon+2\varepsilon}{\alpha-\varepsilon^2}<22\varepsilon+5\varepsilon=27\varepsilon,
	$
	and thus
	\begin{equation*}
		x_{\tilde v}>\alpha-29\varepsilon>1-\frac1p-30\varepsilon.
	\end{equation*}
	In particular,
	$
	x_{\tilde v}>1-\frac1p-30\varepsilon>\frac25
	$
	for sufficiently small $\varepsilon$.
	
	\medskip
	\noindent\textbf{Proof of \textnormal{(iii)}.}
	We continue to assume that $\hat v\in V_1'$.
	
	\smallskip
	\noindent\emph{Case 1: $p\geq 3$.}
	By relabelling if necessary, we may assume $\bar v\in V_2'$. The eigenequation at $\bar v$ gives
	\begin{equation*}
		\rho(G^{\star})x_{\bar v}
		\geq
		\sum_{u\in N_{V_1'}(\bar v)}x_u
		+
		\sum_{i=3}^p\sum_{u\in N_{V_i'}(\bar v)}x_u.
	\end{equation*}
	Since every vertex in $V_1'$ has Perron entry at least $x_{\tilde v}$,
	\begin{equation*}
		\rho(G^{\star})x_{\bar v}
		\geq
		d_{V_1'}(\bar v)x_{\tilde v}
		+
		\sum_{i=3}^p\sum_{u\in N_{V_i'}(\bar v)}x_u.
	\end{equation*}
	Subtracting this inequality from~\eqref{eq:hatv_upper_refined_new}, we obtain
	\begin{equation*}
		\rho(G^{\star})(x_{\hat v}-x_{\bar v}) \leq (\beta+\varepsilon)D +	\Bigl(|L_0|+d_{V_1'}(\hat v)+|V_2'|+\sum_{i=3}^p(|V_i'|-d_{V_i'}(\bar v))\Bigr)x_{\hat v}	-d_{V_1'}(\bar v)x_{\tilde v}.
	\end{equation*}
	Now $|L_0|\leq \varepsilon D$, $d_{V_1'}(\hat v)<2\varepsilon D$,  and $|V_2'|\leq |V_2|\leq \Bigl(\frac1p+2\varepsilon\Bigr) D$. Moreover, by \cref{lem:Vprime-degrees-right}, $d_{V_i'}(\bar v)\geq \Bigl(\frac1p-20\varepsilon\Bigr) D$ with $ i\in\{1,3,\dots,p\}$.
	Hence $d_{V_1'}(\bar v)\geq \Bigl(\frac1p-20\varepsilon\Bigr) D$, and $\sum_{i=3}^p d_{V_i'}(\bar v)\geq (p-2)\Bigl(\frac1p-20\varepsilon\Bigr) D$.
	Since
	\begin{equation*}
		\sum_{i=3}^p |V_i'|	\leq \sum_{i=3}^p |V_i| = D-|V_1|-|V_2|\leq \Bigl(1-\frac2p+4\varepsilon\Bigr)D,
	\end{equation*}
	we obtain $\sum_{i=3}^p\bigl(|V_i'|-d_{V_i'}(\bar v)\bigr)\leq (20p-36)\varepsilon D$.
	Therefore,
	\begin{equation*}
		\rho(G^{\star})(x_{\hat v}-x_{\bar v}) \leq	(\beta+\varepsilon)D +\Bigl(\frac1p+(20p-31)\varepsilon\Bigr)D x_{\hat v} -\Bigl(\frac1p-20\varepsilon\Bigr)D x_{\tilde v}.
	\end{equation*}
	Equivalently,
	\begin{equation*}
		x_{\bar v} \geq x_{\hat v}	-
		\frac{
			\frac1p(x_{\hat v}-x_{\tilde v})
			+(20p-31)\varepsilon x_{\hat v}
			+20\varepsilon x_{\tilde v}
			+(\beta+\varepsilon)
		}{\rho(G^{\star})/ D}.
	\end{equation*}
	From the proof of part~\textnormal{(ii)},
	$
	\rho(G^{\star})(x_{\hat v}-x_{\tilde v})
	\leq
	(\beta+\varepsilon)D+21\varepsilon D x_{\hat v}
	\leq (\beta+22\varepsilon)D.
	$
	Using~\eqref{eq:rho_lower_xi2}, and observing that $\alpha\geq \frac23$ when $p\geq 3$, we get
	\begin{equation*}
		x_{\hat v}-x_{\tilde v}
		\leq \frac{\beta+22\varepsilon}{\alpha+\beta-\varepsilon^2}
		<
		\frac{23\varepsilon}{3/5}
		<39\varepsilon.
	\end{equation*}
	for sufficiently small $\varepsilon$. 
	
	Hence
	$
	\frac1p(x_{\hat v}-x_{\tilde v})\leq \frac{39}{p}\varepsilon\leq 13\varepsilon.
	$
	Since $x_{\hat v}\leq 1$, $x_{\tilde v}\leq 1$, and $\beta\leq \varepsilon$, we infer
	$
	\frac1p(x_{\hat v}-x_{\tilde v})+(20p-31)\varepsilon x_{\hat v}+20\varepsilon x_{\tilde v}+(\beta+\varepsilon)
	\leq (20p+4)\varepsilon.
	$
	Again using $\alpha+\beta-\varepsilon^2>\frac35$ for sufficiently small $\varepsilon$, we obtain
	\begin{equation*}
		x_{\bar v}>
		x_{\hat v}-\frac{(20p+4)\varepsilon}{3/5}
		=
		x_{\hat v}-\frac53(20p+4)\varepsilon.
	\end{equation*}
	By part~\textnormal{(i)}, $x_{\hat v}>\alpha-2\varepsilon$,	and hence
	\begin{equation*}
		x_{\bar v}>
		\alpha-2\varepsilon-\frac53(20p+4)\varepsilon
		=
		\alpha-\frac{100p+26}{3}\varepsilon
		>
		\alpha-(34p+9)\varepsilon.
	\end{equation*}
	Since $\alpha\geq \frac23$, we conclude that $x_{\bar v}>\frac{2}{3}-(34p+9)\varepsilon>\frac{2}{5}$
	for sufficiently small $\varepsilon$.
	
	\smallskip
	\noindent\emph{Case 2: $p=2$.}
	In this case $\alpha=\frac12$. By relabelling, we may assume that $\hat v\in V_1'$ and $\bar v\in V_2'$. The eigenequation at $\bar v$ gives $\rho(G^{\star})x_{\bar v}\geq d_{V_1'}(\bar v)x_{\tilde v}$.
	By \cref{lem:Vprime-degrees-right},
	$
	d_{V_1'}(\bar v)\geq \Bigl(\frac12-16\varepsilon\Bigr)D.
	$
	Subtracting this from~\eqref{eq:hatv_upper_refined_new}, we obtain
	\begin{equation*}
		\rho(G^{\star})(x_{\hat v}-x_{\bar v})
		\leq
		(\beta+\varepsilon) D
		+
		\Bigl(\frac12+5\varepsilon\Bigr) D x_{\hat v}
		-
		\Bigl(\frac12-16\varepsilon\Bigr) D x_{\tilde v}.
	\end{equation*}
	Thus
	\begin{equation*}
		x_{\bar v}
		\geq
		x_{\hat v}
		-
		\frac{
			\frac12(x_{\hat v}-x_{\tilde v})
			+5\varepsilon x_{\hat v}
			+16\varepsilon x_{\tilde v}
			+(\beta+\varepsilon)
		}{\rho(G^{\star})/ D}.
	\end{equation*}
	As above,
	\begin{equation*}
		x_{\hat v}-x_{\tilde v}
		\leq
		\frac{\beta+22\varepsilon}{\frac12+\beta-\varepsilon^2}
		<
		\frac{23\varepsilon}{2/5}
		<58\varepsilon.
	\end{equation*}
	for sufficiently small $\varepsilon$, and therefore
	$
	\frac12(x_{\hat v}-x_{\tilde v})<29\varepsilon.
	$
	
	Using $x_{\hat v}\leq 1$, $x_{\tilde v}\leq 1$, $\beta\leq \varepsilon$, and
	$
	\frac12+\beta-\varepsilon^2>\frac25,
	$
	we deduce that
	\begin{equation*}
		x_{\bar v}>
		x_{\hat v}-\frac{29\varepsilon+5\varepsilon+16\varepsilon+2\varepsilon}{2/5}
		=
		x_{\hat v}-130\varepsilon.
	\end{equation*}
	By part~\textnormal{(i)},
	$
	x_{\hat v}>\frac12-2\varepsilon,
	$
	whence
	$
	x_{\bar v}>\frac12-132\varepsilon>\frac25
	$
	for sufficiently small $\varepsilon$. This completes the proof.
\end{proof}

\begin{lemma}\label{lem:B2-small-right}
	We have	$ |B_2\setminus (L\cup K)|\leq \varepsilon D $.
\end{lemma}
\begin{proof}
	Suppose, to the contrary, that
	$|B_2\setminus (L\cup K)|>\varepsilon D$.  By definition of $B_2$,
	\begin{equation*}
		(n-D) D-e_{G^{\star}}(A,B)=|A||B|-e_{G^{\star}}(A,B) \ge \varepsilon |A||B_2\setminus(L\cup K)|	>\varepsilon^2 (n-D)D .
	\end{equation*}
	Thus
	\begin{equation}\label{eq:normA-B}
		\|A(|A|K_1\cup |B|K_1\cup E_{{G^{\star}}}(A,B))\|_2
		\leq \sqrt{(1-\varepsilon^2)(n-D)D}.
	\end{equation}
	On the other hand, $G^{\star}[B]$ is $\partial_{c}F$-free.  By \Cref{lem:boundary-spectral-cap},
	\begin{equation}\label{eq:ubGB-radius}
		\rho(G^{\star}[B])\leq \rho(T(D,p))+C_{\partial} D^{s_F}.
	\end{equation}
	Similar to the proof of \Cref{lem:scale-sensitive-endpoint-cap}, for every
	unit Perron vector $\boldsymbol{x}$, let
	$\boldsymbol{x}_A=(x_u)_{u\in A}$ and
	$\boldsymbol{x}_B=(x_v)_{v\in B}$.  Then
	\begin{equation}\label{eq:ubG-radius}
		\boldsymbol{x}^TA(G^{\star})\boldsymbol{x}
		\le (n-D-1)\|\boldsymbol{x}_A\|_2^2
		+2\|A(|A|K_1\cup |B|K_1\cup E_{G^{\star}}(A,B))\|_2
		\|\boldsymbol{x}_A\|_2\|\boldsymbol{x}_B\|_2
		+\rho(G^{\star}[B])\|\boldsymbol{x}_B\|_2^2 .
	\end{equation}
	
	Therefore, $\rho(G^{\star})$ is at most the largest eigenvalue of
	\begin{equation*}
		\begin{pmatrix}
			n-D-1 & \sqrt{(1-\varepsilon^2)(n-D)D}\\
			\sqrt{(1-\varepsilon^2)(n-D)D} & \rho(G^{\star}[B])
		\end{pmatrix}.
	\end{equation*}
	
	Since $p\ge2$, for all sufficiently large $D$, by \cref{lem:spectral-radius-estimate}, we have
	\begin{equation}\label{eq:B2-T-radius-crude}
		\frac{D}{4}\le\rho(T(D,p))\le D.
	\end{equation}
	
%	The upper bound follows from the maximum-degree bound.  The lower bound
%	follows from the Rayleigh quotient on the all-one vector and
%	\Cref{lem:turan-edge-estimate}:
%	\[
%	\rho(T(\Delta,p))\ge \frac{2t(\Delta,p)}{\Delta}
%	\ge \left(1-\frac1p\right)\Delta-\frac{p}{4\Delta}
%	\ge \frac{\Delta}{4}
%	\]
%	for $p\ge2$ and all large $\Delta$.
	
	Combining \eqref{eq:normA-B}, \eqref{eq:ubGB-radius} and \eqref{eq:ubG-radius}, we have
	\begin{align}
		\rho(G^{\star})&\leq \lambda_{\max}
		\begin{pmatrix}
			n-D-1 & \sqrt{(1-\varepsilon^2)(n-D)D}\\
			\sqrt{(1-\varepsilon^2)(n-D)D} & \rho(T(D,p))
		\end{pmatrix}
		+ C_{\partial} D^{s_F}\nonumber\\
		&<
		\lambda_{\max}
		\begin{pmatrix}
			n-D & \sqrt{(1-\varepsilon^2)(n-D)D}\\
			\sqrt{(1-\varepsilon^2)(n-D)D} & \rho(T(D,p))
		\end{pmatrix}
		+ C_{\partial} D^{s_F}.\label{eq:B2-upper-explicit}
	\end{align}
	Here the last step uses that the largest eigenvalue is monotone in a
	diagonal entry and increases by at most the amount added to that entry.
	
	For the model graph $S_{n,D,p}=(n-D)K_1\vee T(D,p)$, the Rayleigh
	quotient on vectors constant on the two displayed blocks gives the same
	matrix with bottom-right entry $2t(D,p)/D$.  By
	\cref{lem:spectral-radius-estimate}, replacing this entry by
	$\rho(T(D,p))$ changes the largest eigenvalue by at most
	$1$.  Hence
\begin{equation}\label{eq:B2-model-lower-explicit}
		\rho(S_{n,D,p})
		\ge
		\lambda_{\max}
		\begin{pmatrix}
			0 & \sqrt{(n-D) D}\\
			\sqrt{(n-D) D} & \rho(T(D,p))
		\end{pmatrix}
		-1.
	\end{equation}
	
	First, we calculate that
	\begin{align*}
		&\lambda_{\max}
		\begin{pmatrix}
			0 & \sqrt{(n-D)D}\\
			\sqrt{(n-D)D} & \rho(T(D,p))
		\end{pmatrix}
		-
		\lambda_{\max}
		\begin{pmatrix}
			0 & \sqrt{(1-\varepsilon^2)(n-D)D}\\
			\sqrt{(1-\varepsilon^2)(n-D)D} & \rho(T(D,p))
		\end{pmatrix}
		\\
		&\quad =
		\frac{
			\sqrt{\rho(T(D,p))^2+4(n-D) D}
			-
			\sqrt{\rho(T(D,p))^2+4(1-\varepsilon^2)(n-D)D}
		}{2}
		\\
		&\quad =
		\frac{
			2\varepsilon^2(n-D)D
		}{
			\sqrt{\rho(T(D,p))^2+4(n-D)D}
			+
			\sqrt{\rho(T(D,p))^2+4(1-\varepsilon^2)(n-D)D}
		}.
	\end{align*}
	Using \eqref{eq:B2-T-radius-crude} and $n-D\le D$, the denominator is at most $2\sqrt{5} D$.  Therefore
	\begin{equation}\label{eq:B2-offdiag-gain}
	\begin{gathered}
		\lambda_{\max}
		\begin{pmatrix}
			0 & \sqrt{(n-D)D}\\
			\sqrt{(n-D)D} & \rho(T(D,p))
		\end{pmatrix}
		\\
		-\lambda_{\max}
		\begin{pmatrix}
			0 & \sqrt{(1-\varepsilon^2)(n-D)D}\\
			\sqrt{(1-\varepsilon^2)(n-D)D} & \rho(T(D,p))
		\end{pmatrix}
		\\
		\ge \frac1{\sqrt{5}}\varepsilon^2(n-D) .
	\end{gathered}
	\end{equation}
	
	It remains to bound the cost of raising the upper-left diagonal entry from	$0$ to $n-D$.  For $0\le x\le n-D$, the derivative of
	\begin{equation*}
		\lambda_{\max}
		\begin{pmatrix}
			x & \sqrt{(1-\varepsilon^2)(n-D)D}\\
			\sqrt{(1-\varepsilon^2)(n-D)D} & \rho(T(D,p))
		\end{pmatrix}
	\end{equation*}
	with respect to $x$ is
	\begin{equation*}
		\frac{1}{2}
		\left(1-\frac{\rho(T(D,p))-x}{\sqrt{(\rho(T(D,p))-x)^2+4(1-\varepsilon^2)(n-D)D}}
		\right).
	\end{equation*}
	By \eqref{eq:right-beta-bound}, $n-D=\beta D \le \varepsilon D$, and so \eqref{eq:B2-T-radius-crude} gives
	$\rho(T(D,p))-x\geq \frac{D}{4}-(n-D) \geq D/5$.  The displayed derivative is at most
	\begin{equation*}
		\frac{(1-\varepsilon^2)(n-D) D}{(\rho(T(D,p))-x)^2}
		\le
		\frac{(n-D)D}{(D/5)^2}=25\beta .
	\end{equation*}
	Integrating from $0$ to $(n-D)$ yields
	\begin{align}
		&\quad\lambda_{\max}
		\begin{pmatrix}
			n-D & \sqrt{(1-\varepsilon^2)(n-D)D}\\
			\sqrt{(1-\varepsilon^2)(n-D)D} & \rho(T(D,p))
		\end{pmatrix}
		\nonumber\\
		&\quad-\lambda_{\max}
		\begin{pmatrix}
			0 & \sqrt{(1-\varepsilon^2)(n-D)D}\\
			\sqrt{(1-\varepsilon^2)(n-D)D} & \rho(T(D,p))
		\end{pmatrix}\nonumber\\
		&\le 25\beta (n-D) .\label{eq:B2-leftdiag-loss}
	\end{align}
	
	Combining \eqref{eq:B2-upper-explicit},
	\eqref{eq:B2-model-lower-explicit}, \eqref{eq:B2-offdiag-gain}, and
	\eqref{eq:B2-leftdiag-loss}, we obtain
	\begin{align}
		\rho(S_{n,D,p})-\rho(G^{\star})
		&\ge
		\left(\frac1{\sqrt5}\varepsilon^2-25\beta\right)(n-D)
		-C_{\partial} D^{s_F}-1\nonumber\\
		&\ge
		\left(\frac{1}{\sqrt5}-\frac{1}{4}\right)\varepsilon^2(n-D)
		-C_{\partial} D^{s_F}-1\nonumber\\
		&\ge
		\frac{1}{10}\varepsilon^2 (n-D)
		-C_{\partial} D^{s_F}-1.
		\label{eq:B2-final-positive-gap}
	\end{align}
	Here \eqref{eq:right-beta-bound} gives $25\beta\le\varepsilon^2/4$.
	Finally, by $G^{\star}\in\mathfrak{U}_{n}(\Delta,\zeta_F;F)$ and \eqref{eq:safe-interval}, the right-endpoint safety range gives $n-D\ge\zeta_F n^{s_F}\ge\zeta_F D^{s_F}$. By \eqref{eq:small-constant-choice} and \eqref{eq:B2-final-positive-gap},
	\begin{equation*}
		\rho(S_{n,D,p})-\rho(G^{\star}) > (2C_{\partial}+2)D^{s_F}-C_{\partial}D^{s_F}-1>0.
	\end{equation*}
	Thus the final line of \eqref{eq:B2-final-positive-gap} is positive for all	large $n$.  This gives $\rho(S_{n,D,p})>\rho(G^{\star})$, contradicting the local upper-tail extremality assumption $\rho(G^{\star})\ge \rho(S_{n,D,p})$.
	Therefore $|B_2\setminus(L\cup K)|\le\varepsilon D$.
\end{proof}

Define $B_{1,i}':=B_{1,i}\setminus (L\cup K)$. If $v\in B_{1,i}'$, then for every $j\in[p]\setminus\{i\}$,
\begin{align}
	d_{B_{1,j}'}(v)
	&=d_{V_j}(v)-d_{B_{2,j}\setminus (L\cup K)}(v)-d_{L\cap V_j}(v)-d_{K_j}(v)\nonumber\\
	&\geq d_B(v)-d_{V_i}(v)-\sum_{k\in[p]\setminus\{i,j\}}|V_k|-|B_2\setminus (L\cup K)|-|L|-|K|\nonumber\\
	&\geq \Bigl(1-\frac{1}{p}-12\varepsilon\Bigr)D-2\varepsilon D-\Bigl(1-2\Bigl(\frac{1}{p}-2\varepsilon\Bigr)\Bigr)D-\varepsilon D-\varepsilon D-\varepsilon D\nonumber\\
	&\geq \Bigl(\frac{1}{p}-21\varepsilon\Bigr) D.\label{eq_b1i_degree}
\end{align}

\begin{lemma}\label{lem:B1i-independent-right}
	For every $i\in[p]$, we have
	$	E\bigl(G^{\star}[B_{1,i}']\bigr)=\emptyset.	$
\end{lemma}
\begin{proof}
	Suppose, to the contrary, that $ E\bigl(G^{\star}[B_{1,1}']\bigr)\neq\emptyset $.
	Choose an edge 	$ v_{1,1}v_{1,2}\in E\bigl(G^{\star}[B_{1,1}']\bigr) $.	By \cref{lem:B2-small-right} and the estimates $|L|,|K|\leq \varepsilon D$, we have, for every $i\in[p]$,
	$
	|B_{1,i}'|
	\geq |V_i|-|B_2\setminus (L\cup K)|-|L|-|K|
	\geq \Bigl(\frac1p-5\varepsilon\Bigr) D.
	$
	In particular, $ |B_{1,i}'|>f $	for sufficiently large $D$.
	
	Choose distinct vertices $v_{1,3},\dots,v_{1,f}\in B_{1,1}'$, and let
	$ S_1:=\{v_{1,1},\dots,v_{1,f}\}\subseteq B_{1,1}'	$.	Then $G^{\star}[S_1]$ contains an edge.
	
	By~\eqref{eq_b1i_degree}, every vertex of $B_{1,1}'$ has at least
	$
	\Bigl(\frac1p-21\varepsilon\Bigr) D
	$
	neighbours in $B_{1,2}'$. Hence
	\begin{align*}
		\left|\bigcap_{j=1}^{f}N_{B_{1,2}'}(v_{1,j})\right|
		&\geq \sum_{j=1}^f |N_{B_{1,2}'}(v_{1,j})|-(f-1)\left|\bigcup_{j=1}^f N_{B_{1,2}'}(v_{1,j})\right|\\
		&\geq \sum_{j=1}^f d_{B_{1,2}'}(v_{1,j})-(f-1)|V_2|\\
		&\geq f\Bigl(\frac1p-21\varepsilon\Bigr)D-(f-1)\Bigl(\frac1p+2\varepsilon\Bigr)D\\
		&=\Bigl(\frac1p-23f\varepsilon+2\varepsilon\Bigr)D
		>f.
	\end{align*}
	Choose
	$
	S_2:=\{v_{2,1},\dots,v_{2,f}\}\subseteq \bigcap_{j=1}^f N_{B_{1,2}'}(v_{1,j}).
	$
	Then every vertex of $S_2$ is adjacent to every vertex of $S_1$.
	
	Proceed inductively. Suppose that for some $2\leq \ell\leq p-1$ we have already chosen pairwise disjoint sets $S_t=\{v_{t,1},\dots,v_{t,f}\}\subseteq B_{1,t}'$ $(1\leq t\leq \ell)$
	such that
	$
	G^{\star}[S_1\cup\cdots\cup S_\ell]\supseteq T(\ell f,\ell)+e,
	$
	where the additional edge lies inside $S_1$. Again by~\eqref{eq_b1i_degree},
	\begin{align*}
		\left|\bigcap_{t=1}^{\ell}\bigcap_{v\in S_t}N_{B_{1,\ell+1}'}(v)\right|
		&\geq \sum_{t=1}^{\ell}\sum_{v\in S_t}|N_{B_{1,\ell+1}'}(v)|-(\ell f-1)\left|\bigcup_{t=1}^{\ell}\bigcup_{v\in S_t}N_{B_{1,\ell+1}'}(v)\right|\\
		&\geq \ell f\Bigl(\frac1p-21\varepsilon\Bigr)D-(\ell f-1)\Bigl(\frac1p+2\varepsilon\Bigr)D\\
		&=\Bigl(\frac1p-23\ell f\varepsilon+2\varepsilon\Bigr)D
		>f.
	\end{align*}
	Hence, we may choose
	$
	S_{\ell+1}\subseteq \bigcap_{t=1}^{\ell}\bigcap_{v\in S_t}N_{B_{1,\ell+1}'}(v)
	$
	with $|S_{\ell+1}|=f$.
	
	At the end, we obtain pairwise disjoint sets
	$S_i\subseteq B_{1,i}'$ $(i\in[p])$
	of size $f$ such that $ G^{\star}[S_1\cup\cdots\cup S_p]\supseteq T(fp,p)+e $. 
	Finally, since every vertex of $B_{1,i}'\subseteq B_{1,i}$ has at most $\varepsilon|A|$ non-neighbours in $A$, we have
	\begin{align*}
		\left|\bigcap_{i=1}^p\bigcap_{v\in S_i}N_A(v)\right|
		&\geq |A|-\sum_{i=1}^p\sum_{v\in S_i}|A\setminus N_A(v)|\\
		&>(1-rf\varepsilon)|A|>f
	\end{align*}
	for sufficiently large $|A|$. Choose a set
	$
	U\subseteq \bigcap_{i=1}^p\bigcap_{v\in S_i}N_A(v)
	$
	with $|U|=f$. Then $U$ is joined to every $S_i$, and the sets $S_1,\ldots,S_p$ are mutually complete with one additional edge inside $S_1$.  In other words,
	\begin{equation*}
		G^{\star}[U\cup S_1\cup\cdots\cup S_p]
		\supseteq fK_1\vee\bigl(T(fp,p)+e\bigr),
	\end{equation*}
%	where the extra edge lies in the first Tur\'an class.  This graph contains $F$.  To see this, choose a critical edge $ab$ of $F$ and a proper $(p+1)$-coloring of $F-ab$ in which $a$ and $b$ have the same color.  Embed $a$ and $b$ on the two endpoints of the extra edge in $S_1$, embed the remaining vertices of their color class into $D_1$, embed one of the other color classes into $U$, and embed the remaining $p-1$ color classes into $S_2,\ldots,S_p$.  All required cross-color edges are present, and the edge $ab$ is supplied by the internal edge in $S_1$.  
	This gives a copy of $F$ in $G^{\star}$, a contradiction.
\end{proof}

\begin{lemma}\label{lem:low-degree-right_empty}
	We have $ L=\emptyset. $
\end{lemma}
\begin{proof}
	Assume, to the contrary, that $L\neq \emptyset$, and choose $v_0\in L$. Recall that
	$
	x_{\hat v}=\max_{v\in B\setminus K}x_v.
	$
	By relabelling the parts if necessary, we may assume that
	$
	\hat v\in V_1\setminus K.
	$
	By \cref{lem:perron-entry-right}, we know that $\hat v\notin L$.
	
	Using the eigenequation at $\hat v$, and covering its neighbourhood by $A$, $K$, $L$, $V_1\setminus (L\cup K)$, $B_2\setminus (L\cup K)$, $\bigcup_{i=2}^p B_{1,i}'$, we obtain the upper estimate
	\begin{align*}
		\rho(G^{\star})x_{\hat v}
		&\le\sum_{u\in N_A(\hat v)}x_u+\sum_{u\in N_K(\hat v)}x_u+\sum_{u\in N_L(\hat v)}x_u
		+\sum_{u\in N_{V_1\setminus (L\cup K)}(\hat v)}x_u\\
		&\quad
		+\sum_{u\in N_{B_2\setminus (L\cup K)}(\hat v)}x_u
		+\sum_{i=2}^p \sum_{u\in N_{B_{1,i}'}(\hat v)}x_u.
	\end{align*}
	Since every vertex of $L\cup (B\setminus K)$ has Perron entry at most $x_{\hat v}$, we have
	\begin{align*}
		\rho(G^{\star})x_{\hat v}
		&\leq |A|+|K|
		+|L|x_{\hat v}
		+d_{V_1\setminus (L\cup K)}(\hat v)x_{\hat v}
		+\sum_{u\in B_2\setminus (L\cup K)}x_u
		+\sum_{i=2}^p \sum_{u\in B_{1,i}'}x_u\\
		&\leq |A|+|K|
		+\bigl(|L|+d_{V_1}(\hat v)+|B_2\setminus (L\cup K)|\bigr)x_{\hat v}
		+\sum_{i=2}^p \sum_{u\in B_{1,i}'}x_u.
	\end{align*}
	Now, by \Cref{lem:low-degree-right},
	$|L|\leq \varepsilon D$; moreover, $d_{V_1}(\hat v)<2\varepsilon D$ and $|B_2\setminus (L\cup K)|\leq \varepsilon D$,
	and therefore
	\begin{equation}\label{eq_sum_b1prime_lowdeg}
		\sum_{i=2}^p \sum_{u\in B_{1,i}'}x_u
		>
		\rho(G^{\star})x_{\hat v}-(|A|+|K|)-4\varepsilon D x_{\hat v}.
	\end{equation}
	
	 Since $v_0\in B=N_{G^{\star}}(u^*)$, the edge $u^*v_0$ is present in $G^{\star}$.  Define $G'$ by deleting all edges incident with $v_0$ except $u^*v_0$, and then joining $v_0$ to every vertex of $\bigcup_{i=2}^p B_{1,i}'$:
	\begin{equation*}
		G'
		=
		G^{\star}-\{v_0u:\  u\in N_{G^{\star}}(v_0)\setminus\{u^*\}\}
		+\{v_0w:\  w\in \bigcup_{i=2}^p B_{1,i}'\}.
	\end{equation*}
	Then
	\begin{equation}\label{eq_diff_sr_1}
		(\rho(G')-\rho(G^{\star}))\|\boldsymbol x\|_2^2
		\geq \boldsymbol x^\top \left( A(G')-A(G^{\star}) \right)\boldsymbol x
		=2x_{v_0}\biggl(\sum_{i=2}^p \sum_{w\in B_{1,i}'}x_w-\sum_{u\in N_{G^{\star}}(v_0)\setminus\{u^*\}}x_u\biggr).
	\end{equation}
	The equality holds because the root edge $u^*v_0$ is kept.  Using~\eqref{eq_sum_b1prime_lowdeg} and
	$
	\sum_{u\in N_{G^{\star}}(v_0)\setminus\{u^*\}}x_u
	\leq |A|+|K|+d_B(v_0)x_{\hat v}$, and
	$d_B(v_0)\leq \Bigl(1-\frac{1}{p}-12\varepsilon\Bigr) D$,
	we get
	\begin{align*}
		\sum_{i=2}^p \sum_{w\in B_{1,i}'}x_w-\sum_{u\in N_{G^{\star}}(v_0)\setminus\{u^*\}}x_u
		&>
		\bigl(\rho(G^{\star})-4\varepsilon D-d_B(v_0)\bigr)x_{\hat v}-2(|A|+|K|)\\
		&>
		\bigl(\beta+8\varepsilon+o(1)\bigr)D\,x_{\hat v}-2(\beta+\varepsilon)D.
	\end{align*}
	By \cref{lem:perron-entry-right},
	$x_{\hat v}>1-\frac1p-2\varepsilon=:\alpha-2\varepsilon$, $\alpha=\frac{p-1}{p}\geq \frac{1}{2}$.
	Since $\beta\leq \varepsilon$, we obtain
	\begin{align*}
		\bigl(\beta+8\varepsilon+o(1)\bigr)D x_{\hat v}-2(\beta+\varepsilon)D
		&>
		\Bigl((\beta+8\varepsilon)(\alpha-2\varepsilon)-2(\beta+\varepsilon)+o(1)\Bigr)D\\
		&\geq
		\Bigl((\beta+8\varepsilon)\Bigl(\frac12-2\varepsilon\Bigr)-2(\beta+\varepsilon)+o(1)\Bigr)D\\
		&\geq
		\Bigl(\frac12\varepsilon-18\varepsilon^2+o(1)\Bigr)D>0
	\end{align*}
	for sufficiently large $D$. Hence, by~\eqref{eq_diff_sr_1},
	$ \rho(G')\ge\rho(G^{\star}) $. By \Cref{lem:perron-component-root}, the root $u^*$ is chosen in the component $C^\star$ attaining the spectral radius. Since $v_0\in B=N_{G^{\star}}(u^*)$, we have $x_{v_0}>0$, and the inequality is strict.
	
	It remains to verify that $G'$ is still $F$-free. Suppose otherwise that $G'$ contains a copy $F'$. Since $G^{\star}$ and $G'$ differ only in the edges incident with $v_0$, we must have
	$	v_0\in V(F')	$.	Write $|N_{F'}(v_0)|=f_0<f$.
	By construction,
	$
	 N_{F'}(v_0) \subseteq \bigcup_{i=2}^p B_{1,i}' \cup\{u^*\}.
	$
	Then $N_{F'}(v_0)\setminus\{u^*\} \subseteq \bigcup_{i=2}^p B_{1,i}'$.  Every vertex in $N_{F'}(v_0)\setminus\{u^*\}$ satisfies
	$
	d_{B_{1,1}'}(v)\geq \Bigl(\frac{1}{p}-21\varepsilon\Bigr) D
	$
	by~\eqref{eq_b1i_degree}. Therefore
	\begin{align*}
		\left|\bigcap_{v\in N_{F'}(v_0)\setminus\{u^*\}}N_{B_{1,1}'}(v)\right|
		&\geq |B_{1,1}'|-|N_{F'}(v_0)\setminus\{u^*\}|\Bigl(\frac1p+2\varepsilon-\frac1p+21\varepsilon\Bigr) D\\
		&\geq \Bigl(\frac1p-2\varepsilon-23f\varepsilon\Bigr) D>f
	\end{align*}
	for sufficiently small $\varepsilon$. Hence, we may choose
	$
	v_1\in \bigcap_{v\in N_{F'}(v_0)\setminus\{u^*\}}N_{B_{1,1}'}(v)\setminus V(F').
	$
	The vertex $v_1$ is also adjacent to $u^*$, because $v_1\in B=N_{G^{\star}}(u^*)$.  Thus every neighbour of $v_0$ in $F'$ is adjacent to $v_1$ in $G^{\star}$.  Consequently
	$
	G^{\star}[(V(F')\setminus\{v_0\})\cup\{v_1\}]
	$
	contains a copy of $F$, a contradiction.  Hence $G'$ is $F$-free.
	
	The edge $u^*v_0$ was kept, so $d_{G'}(u^*)= D$.  Therefore $\Delta(G')\ge D$.  If $\Delta(G')=D'> D$, then only vertices in $\bigcup_{i=2}^p B_{1,i}'$ can gain one new incident edge.  Moreover, the bounds on $\bigcup_{i=2}^p B_{1,i}'$, the vertex $v_0$ has degree at most $|\bigcup_{i=2}^p B_{1,i}'|+1\leq D+1$.  
	
	Hence $D'\le D+1$.  Recall that $n- D\geq \kappa_{F} n^{s_F}$ and $\kappa_{F}\geq 2\zeta_F$ (see \eqref{eq:small-constant-choice}). Then $\Delta \leq D'\leq n-\zeta_F n^{s_F}$ for all sufficiently large $n$. Thus $G'\in\mathfrak{U}_n(\Delta,\zeta_F;F)$ and $\rho(G')>\rho(G^{\star})$, this contradicts the upper-tail extremality of $G^{\star}$.
	
	%Here the local margin $\kappa_F\geq 2\zeta_F$ is used.  
	
%	The endpoint buffer, \Cref{lem:endpoint-buffer-exclusion}, has already been applied before the right-endpoint cleaning is invoked.  Thus the actual degree satisfies $n-\Delta\ge\kappa_Fn^{s_F}=2\zeta_Fn^{s_F}$.  Hence, if the switched graph has maximum degree $D'\le \Delta+1$, then
	
%	  Thus the switched graph stays in the safe upper-tail class from \Cref{rem:endpoint-bookkeeping}.  Since it is $F$-free and has strictly larger Rayleigh quotient, 
\end{proof}

\begin{lemma}\label{lem_nob1k}
	$B_{1,i}\cap K_i=\emptyset$.
\end{lemma}
\begin{proof}
	Suppose to the contrary that there exists a vertex $v\in B_{1,i}\cap K_i$ for some $i\in[p]$.
	Since $v\in K_i$, by the definition of $K_i$ we have $d_{V_i}(v)\geq 2\varepsilon D$.
	By \cref{lem:low-degree-right_empty}, we already know that $L=\emptyset$, so $B'_{1,i}=B_{1,i}\setminus K_i$.
	Hence
	\begin{equation*}
		d_{B'_{1,i}}(v)=d_{V_i}(v)-d_{B_{2,i}\setminus K_i}(v)-d_{K_i}(v)\geq 2\varepsilon D-|B_{2,i}\setminus K_i|-(|K_i|-1).
	\end{equation*}
	Using
	\begin{equation*}
		|B_{2,i}\setminus K_i|\leq |B_2\setminus K|\leq \varepsilon D,
		\quad
		|K_i|\leq |K|\leq \varepsilon D,
	\end{equation*}
	we obtain
	$
	d_{B'_{1,i}}(v)\geq 2\varepsilon D-\varepsilon D -(\varepsilon D-1)\geq 1.
	$
	Thus $v$ has a neighbour in $B'_{1,i}$, so
	$
	E(B'_{1,i})\neq \emptyset,
	$
	contradicting \cref{lem:B1i-independent-right}. Therefore
	$
	B_{1,i}\cap K_i=\emptyset.
	$
\end{proof}

Since $V_i=B_{1,i}\cup B_{2,i}$, it follows that $K_i\subseteq B_{2,i}$. Moreover, because $L=\emptyset$, we also have $B'_{1,i}=B_{1,i}$.

\begin{lemma}\label{lem_b2minusK_otherdegree}
	For every $i\in[p]$, every $j\in[p]\setminus\{i\}$, and every $v\in K_i$, we have
	$
	d_{B_{1,j}}(v)\geq \Bigl(\frac{1}{2p}-10\varepsilon\Bigr) D.
	$
\end{lemma}
\begin{proof}
	Fix $i\in[p]$, $j\in[p]\setminus\{i\}$, and $v\in K_i$. Since the partition
	$
	B=V_1\cup\cdots\cup V_p
	$
	maximizes
	$
	\sum_{1\leq s<t\leq p}e_{G^{\star}}(V_s,V_t),
	$
	moving $v$ from $V_i$ to $V_j$ cannot increase the number of crossing edges. Hence
	$
	d_{V_j}(v)\geq d_{V_i}(v).
	$
	Since $L=\emptyset$ by \cref{lem:low-degree-right_empty},
	$
	d_B(v)>\Bigl(1-\frac1p-12\varepsilon\Bigr) D.
	$
	Moreover,
	\begin{equation*}
		d_{V_j}(v)+d_{V_i}(v)
		=
		d_B(v)-\sum_{k\in[p]\setminus\{i,j\}}d_{V_k}(v)
		\geq d_B(v)-\sum_{k\in[p]\setminus\{i,j\}}|V_k|.
	\end{equation*}
	Since
	$
	|V_i|,|V_j|\geq \Bigl(\frac1p-2\varepsilon\Bigr) D,
	$
	we have
	\begin{equation*}
		\sum_{k\in[p]\setminus\{i,j\}}|V_k|
		=
		D-|V_i|-|V_j|
		\leq \Bigl(1-\frac2p+4\varepsilon\Bigr) D.
	\end{equation*}
	Therefore,
	$
	d_{V_j}(v)+d_{V_i}(v)
	\geq \Bigl(\frac1p-16\varepsilon\Bigr) D.
	$
	Since $d_{V_j}(v)\geq d_{V_i}(v)$, we get
	$
	d_{V_j}(v)\geq \frac12\Bigl(\frac{1}{p}-16\varepsilon\Bigr)D
	=
	\Bigl(\frac1{2p}-8\varepsilon\Bigr)D.
	$
	
	By \cref{lem_nob1k}, we have $B_{1,j}\cap K_j=\emptyset$. Hence, 
	$
	d_{B_{1,j}}(v)
	=
	d_{V_j}(v)-d_{B_{2,j}\setminus K_j}(v)-d_{K_j}(v)
	\geq d_{V_j}(v)-|B_2\setminus K|-|K|.
	$
	Using $|B_2\setminus K|\leq \varepsilon D$ and
	$|K|\leq \varepsilon D$, we obtain
	$
	d_{B_{1,j}}(v)\geq \Bigl(\frac1{2p}-10\varepsilon\Bigr) D.
	$
\end{proof}

\begin{lemma}\label{lem_b2_otherdegree}
	For every $i\in[p]$, every $j\in[p]\setminus\{i\}$, and every $v\in B_{2,i}$, we have
	$
	d_{B_{1,j}}(v)\geq \Bigl(\frac{1}{2p}-10\varepsilon\Bigr) D.
	$
\end{lemma}
\begin{proof}
	Fix $i\in[p]$, $j\in[p]\setminus\{i\}$, and $v\in B_{2,i}$.	If $v\in K_i$, then the conclusion follows from \cref{lem_b2minusK_otherdegree}. So assume that $v\in B_{2,i}\setminus K_i$. Then
	$
	d_{V_i}(v)<2\varepsilon D
	$
	by the definition of $K_i$. Since $L=\emptyset$,
	$
	d_B(v)>\Bigl(1-\frac1p-12\varepsilon\Bigr) D.
	$
	Therefore,
	\begin{equation*}
		d_{V_j}(v)=
		d_B(v)-d_{V_i}(v)-\sum_{k\in[p]\setminus\{i,j\}}d_{V_k}(v)\geq d_B(v)-d_{V_i}(v)-\sum_{k\in[p]\setminus\{i,j\}}|V_k|.
	\end{equation*}
	Since  $ d_{V_i}(v)<2\varepsilon D $ and
	\begin{equation*}
		\sum_{k\in[p]\setminus\{i,j\}}|V_k|
		\leq D-2\Bigl(\frac{1}{p}-2\varepsilon\Bigr) D
		=
		\Bigl(1-\frac{2}{p}+4\varepsilon\Bigr) D,
	\end{equation*}
	we get $ d_{V_j}(v) \geq \Bigl(1-\frac{1}{p}-12\varepsilon\Bigr)D-2\varepsilon D-\Bigl(1-\frac{2}{p}+4\varepsilon\Bigr) D = \Bigl(\frac{1}{p}-18\varepsilon\Bigr) D $.
	Again by \cref{lem_nob1k},
	$ d_{B_{1,j}}(v) = d_{V_j}(v)-d_{B_{2,j}\setminus K_j}(v)-d_{K_j}(v)
	\geq d_{V_j}(v)-|B_2\setminus K|-|K| $.
	Using
	$ |B_2\setminus K|\leq \varepsilon D$ and  $ |K|\leq \varepsilon D $,
	we deduce that	$ d_{B_{1,j}}(v)\geq \Bigl(\frac{1}{p}-20\varepsilon\Bigr) D$, which is stronger than the required bound.
\end{proof}

\begin{lemma}\label{lem:B_perron_lower}
	After \cref{lem:low-degree-right_empty} has given $L=\emptyset$, every vertex $v\in B$ satisfies $x_v\geq \frac{1}{6p}$.
\end{lemma}
\begin{proof}
	If $v\notin K$, then $v\in B\setminus(L\cup K)$, and \cref{lem:perron-entry-right}\textnormal{(iii)} gives
	$x_v\ge x_{\bar v}>2/5\ge \frac{1}{6p}$.
	Now let $v\in K_i$. Choose $j\ne i$. By \cref{lem_b2minusK_otherdegree},
	\begin{equation*}
		d_{B_{1,j}}(v)\ge \left(\frac1{2p}-10\varepsilon\right) D .
	\end{equation*}
	\cref{lem_nob1k} and $L=\emptyset$ imply $B_{1,j}\subseteq B\setminus(L\cup K)$, so every vertex of
	$B_{1,j}$ has Perron coordinate at least $x_{\bar v}>2/5$. The eigenequation at $v$, together with
	$\rho(G^{\star})\le D$, therefore gives
	\begin{equation*}
		x_v\ge
		\frac{(2/5)(1/(2p)-10\varepsilon) D}{\rho(G^{\star})}
		\ge \frac25\left(\frac{1}{2p}-10\varepsilon\right)
		\ge \frac{1}{6p}.
	\end{equation*}
\end{proof}

\begin{lemma}\label{lem_ind1_degree}
	For every $i\in[p]$ and every $v\in B_{2,i}$, we have
	$ d_{B_{1,i}}(v)\leq f	$.
\end{lemma}
\begin{proof}
	By symmetry, it suffices to consider the case $i=1$. Suppose, to the contrary, that there exists a vertex $v\in B_{2,1}$ such that
	$ d_{B_{1,1}}(v)>f $. Choose a set
	$S_1\subseteq N_{B_{1,1}}(v)$ and $|S_1|=f$.
	
	We now construct sets
	$S_j\subseteq B_{1,j}$, $(j=2,\dots,p)$
	each of size $f$, such that every vertex of $S_j$ is adjacent to every vertex in
	$
	\{v\}\cup S_1\cup\cdots\cup S_{j-1}.
	$
	For $2\leq j\leq p$, set
	$
	T_{j-1}:=\{v\}\cup S_1\cup\cdots\cup S_{j-1}.
	$
	Then
	$
	|T_{j-1}|=1+(j-1)f\leq 1+(p-1)f.
	$
	By \cref{lem_b2_otherdegree}, the vertex $v$ has at least
	$
	\Bigl(\frac{1}{2p}-10\varepsilon\Bigr)D
	$
	neighbours in $B_{1,j}$. On the other hand, every vertex of
	$ S_1\cup\cdots\cup S_{j-1}\subseteq B_1 $	has at least $
	\Bigl(\frac{1}{p}-21\varepsilon\Bigr)D	$
	neighbours in $B_{1,j}$ by~\eqref{eq_b1i_degree}. Since
	$
	|B_{1,j}|\leq |V_j|\leq \Bigl(\frac1p+2\varepsilon\Bigr) D,
	$ it follows that
	\begin{align*}
		\left|\bigcap_{x\in T_{j-1}}N_{B_{1,j}}(x)\right|
		&\geq \sum_{x\in T_{j-1}}|N_{B_{1,j}}(x)|-(|T_{j-1}|-1)\left|\bigcup_{x\in T_{j-1}}N_{B_{1,j}}(x)\right|\\
		&\geq \Bigl(\frac{1}{2p}-10\varepsilon\Bigr) D
		+(|T_{j-1}|-1)\Bigl(\frac{1}{p}-21\varepsilon\Bigr) D-(|T_{j-1}|-1)\Bigl(\frac{1}{p}+2\varepsilon\Bigr)D\\
		&\geq
		\Bigl(\frac{1}{2p}-23|T_{j-1}|\varepsilon\Bigr)D.
	\end{align*}
	Since $|T_{j-1}|\leq 1+(p-1)f$ and $p,f$ are fixed, this quantity is greater than $f$ for sufficiently small $\varepsilon$. 
	
	Hence, we may choose $ S_j\subseteq \bigcap_{x\in T_{j-1}}N_{B_{1,j}}(x)$, and $|S_j|=f$. Proceeding inductively, we obtain pairwise disjoint sets
	$S_1\subseteq B_{1,1},\ S_2\subseteq B_{1,2},\ \dots,\ S_p\subseteq B_{1,p}$, and $|S_j|=f$, such that every edge between distinct sets $S_1,\dots,S_p$ is present, and $v$ is adjacent to every vertex of $S_1\cup\cdots\cup S_p$.
	
	Since every vertex of $S_1\cup\cdots\cup S_p$ lies in $B_1$, each has at most $\varepsilon|A|$ non-neighbours in $A$.  Thus
	\begin{equation*}
		\left|\bigcap_{j=1}^p \bigcap_{x\in S_j}N_A(x)\right|
		\geq (1-rf\varepsilon)|A|>f
	\end{equation*}
	for sufficiently small $\varepsilon$.
	
	Choose $S_0\subseteq \bigcap_{j=1}^p \bigcap_{x\in S_j}N_A(x)$, and $ |S_0|=f $. Then $S_0,S_1,\ldots,S_p$ form the required multipartite reservoir.  The fixed root $u^*$ is adjacent to $v$ and to all vertices of $S_1\cup\cdots\cup S_p$, while $v$ is adjacent to every vertex of $S_1\cup\cdots\cup S_p$.  Choose a critical edge $ab$ of $F$ and a proper $(p+1)$-coloring of $F-ab$ in which $a$ and $b$ have the same color.  Embed $a,b$ on $u^*,v$, the remaining vertices of their color class into $S_0$, and the other $p$ color classes into $S_1,\ldots,S_p$.  Thus all edges of $F$ are realised in $G^{\star}$, a contradiction.
\end{proof}

We next strengthen the cross-degree estimate. Since $L=\emptyset$, \cref{lem:B2-small-right,lem:K-small-right}
gives $|B_2|\leq 2\varepsilon D$. Thus, for every $v\in B_{2,i}$,
\cref{lem_ind1_degree} yields
\[
d_{V_i}(v)\leq d_{B_{1,i}}(v)+|B_{2,i}|-1\leq f+2\varepsilon D.
\]
Consequently, by \cref{lem_nob1k,lem:B-partition}, for $j\ne i$,
\begin{align}
d_{B_{1,j}}(v)
&\geq d_B(v)-d_{V_i}(v)-\sum_{k\in[p]\setminus\{i,j\}}|V_k|-|B_2\setminus K|-|K|\nonumber\\
&\geq \Bigl(1-\frac1p-12\varepsilon\Bigr)D-(f+2\varepsilon D)
-\Bigl(1-\frac2p+4\varepsilon\Bigr)D-2\varepsilon D\nonumber\\
&\geq \Bigl(\frac1p-21\varepsilon\Bigr)D
\label{eq:b2-strong-otherdegree}
\end{align}
for all sufficiently large $n$.

\begin{lemma}\label{lem_b2_MFfree}
	For every $i\in[p]$, the graph $G^{\star}[B_{2,i}]$ is $\mathcal M(\partial_{c}F)$-free.
\end{lemma}
\begin{proof}
	By symmetry, it suffices to consider the case $i=1$. Suppose, to the contrary, that $G^{\star}[B_{2,1}]$ contains a copy of some graph
	$ F^*\in \mathcal{M}(\partial_{c}F) $. Let $ f_1:=|V(F^*)|<f $.
	
	For every $v\in B_{2,1}$ and $j\in[p]\setminus\{1\}$, \eqref{eq:b2-strong-otherdegree} gives
	\[
	 d_{B_{1,j}}(v)\geq \Bigl(\frac{1}{p}-21\varepsilon\Bigr)D.
	\]
	
	We now construct sets
	$S_j\subseteq B_{1,j}$ $(j=2,\dots,p)$
	each of size $f$, such that every vertex of $S_j$ is adjacent to every vertex in
	$
	V(F^*)\cup S_2\cup\cdots\cup S_{j-1}.
	$
	
	For $j=2$, every vertex of $V(F^*)$ has at least
	$
	\Bigl(\frac{1}{p}-21\varepsilon\Bigr) D
	$
	neighbours in $B_{1,2}$. Since
	$
	|B_{1,2}|\leq |V_2|\leq \Bigl(\frac1p+2\varepsilon\Bigr) D,
	$
	we obtain
	\begin{align*}
		\left|\bigcap_{v\in V(F^*)}N_{B_{1,2}}(v)\right|
		&\geq \sum_{v\in V(F^*)}|N_{B_{1,2}}(v)|-(f_1-1)\left|\bigcup_{v\in V(F^*)}N_{B_{1,2}}(v)\right|\\
		&\geq f_1\Bigl(\frac1p-21\varepsilon\Bigr) D-(f_1-1)\Bigl(\frac1p+2\varepsilon\Bigr) D\\
		&=\Bigl(\frac1p-(23f_1-2)\varepsilon\Bigr) D.
	\end{align*}
	Since $f_1<f$ is fixed, this quantity is greater than $f$ for sufficiently small $\varepsilon$. Hence, we may choose
	$ S_2\subseteq \bigcap_{v\in V(F^*)}N_{B_{1,2}}(v)	$, and $|S_2|=f$.
	
	Now suppose that $S_2,\dots,S_{j-1}$ have been chosen for some $3\leq j\leq p$, and let
	$
	T_{j-1}:=V(F^*)\cup S_2\cup\cdots\cup S_{j-1}.
	$
	Then
	$
	|T_{j-1}|\leq f_1+(j-2)f\leq (p-1)f-1.
	$
	Every vertex of $T_{j-1}$ has at least
	$
	\Bigl(\frac{1}{p}-21\varepsilon\Bigr)D
	$
	neighbours in $B_{1,j}$, and hence
	\begin{align*}
		\left|\bigcap_{v\in T_{j-1}}N_{B_{1,j}}(v)\right|
		&\geq \sum_{v\in T_{j-1}}|N_{B_{1,j}}(v)|-(|T_{j-1}|-1)\left|\bigcup_{v\in T_{j-1}}N_{B_{1,j}}(v)\right|\\
		&\geq |T_{j-1}|\Bigl(\frac1p-21\varepsilon\Bigr)D-(|T_{j-1}|-1)\Bigl(\frac1p+2\varepsilon\Bigr) D\\
		&=\Bigl(\frac1p-(23|T_{j-1}|-2)\varepsilon\Bigr) D.
	\end{align*}
	Again, since $|T_{j-1}|\leq (p-1)f-1$ and $p,f$ are fixed, this quantity is greater than $f$ for sufficiently small $\varepsilon$. Hence we may choose
	$ S_j\subseteq \bigcap_{v\in T_{j-1}}N_{B_{1,j}}(v) $ and $|S_j|=f$.
	
	Proceeding inductively, we obtain sets
	$ S_2,\dots,S_p$, $ S_j\subseteq B_{1,j}$, and $ |S_j|=f $,
	such that every edge between distinct sets $S_2,\dots,S_p$ is present, and every vertex of $F^*$ is adjacent to every vertex in $S_2\cup\cdots\cup S_p$. Thus $	G^{\star}[S_2\cup\cdots\cup S_p]\supseteq T((p-1)f,p-1) $, and by the definition of $\mathcal M(\partial_{c}F)$, the graph $ G^{\star}[V(F^*)\cup S_2\cup\cdots\cup S_p] $ contains a copy of some graph in $\partial_{c}F$.
	
	The copy of a member of $\partial_{c}F$ just obtained lies entirely inside $B=N_{G^{\star}}(u^*)$, where $u^*$ is the fixed maximum-degree root chosen at the beginning of the endpoint proof.  Hence $u^*$ is adjacent to every vertex of that copy.  Adding $u^*$ as the deleted vertex of the corresponding member of $\partial_{c}F$ extends the copy to a copy of $F$ in $G^{\star}$, a contradiction.
\end{proof}

By \eqref{eq:small-constant-choice}, one has in particular $\zeta_F\ge 10^3p(\Lambda_F+C_\partial+2)\varepsilon^{-2}$. By \cref{lem:B1i-independent-right}, \cref{lem_ind1_degree} and~\cref{lem_b2_MFfree},
\begin{equation*}
	e(G^{\star}[V_i])
	=e(G^{\star}[B_{2,i}])+e_{G^{\star}}(B_{1,i},B_{2,i})
	\leq \ex(|B_{2,i}|,\mathcal M(\partial_{c}F))+f|B_{2,i}|.
\end{equation*} 
Together with \cref{lem:K-small-right}, \cref{lem:B2-small-right}, \cref{lem:low-degree-right_empty} and \cref{lem_nob1k}, we have $L=\emptyset$ and $|B_2|\le |B_2\setminus K|+|K|\le 2\varepsilon D\le 2\varepsilon n$.

Then by \eqref{eq:boundary-edge-cap}, we have, there exists a constant $\Lambda_F>0$ such that
\begin{equation}\label{eq:sum_internal_edges}
	\sum_{i=1}^p e(G^{\star}[V_i])
	\leq
	\sum_{i=1}^p \Bigl(\ex(|B_{2,i}|,\mathcal M(\partial_{c}F))+f|B_{2,i}|\Bigr)
	\leq \Lambda_F\, D^{s_F}|B_2|
	\leq 2\Lambda_F\varepsilon n^{1+s_F}.
\end{equation}

\begin{lemma}\label{lem:A1_structure}
	We have	$ E(G^{\star}[A_1])=\emptyset $.
\end{lemma}
\begin{proof}
	Suppose that $uu'\in E(G^{\star}[A_1])$.  We first record the precise reservoirs needed for the critical-edge embedding.  Since $u,u'\in A_1$, each of them has at most $\varepsilon D$ non-neighbours in $B$.  Also $L=\emptyset$, $B_{1,j}\cap K_j=\emptyset$ by \cref{lem_nob1k}, and \cref{lem:B2-small-right} gives, for $j\in[p]$,
	\begin{equation*}
		|B_{1,j}|\ge |V_j|-|B_2\setminus K|-|K|
		\ge \Bigl(\frac1p-4\varepsilon\Bigr)D.
	\end{equation*}
	Hence
	\begin{equation*}
		|B_{1,j}\cap N_{G^{\star}}(u)\cap N_{G^{\star}}(u')|
		\ge \Bigl(\frac1p-6\varepsilon\Bigr)D>f
	\end{equation*}
	for every $j$, once $n$ is large and $\varepsilon$ is fixed sufficiently small.
	
	Choose $S_1\subseteq B_{1,1}\cap N_{G^{\star}}(u)\cap N_{G^{\star}}(u')$ with $|S_1|=f$.  Suppose that $S_1,\ldots,S_{j-1}$ have been chosen, where $2\le j\le p$, each $S_i\subseteq B_{1,i}$ has size $f$, all cross-edges between previously chosen sets are present, and every vertex in the previous sets is adjacent to both $u$ and $u'$.  For a vertex $z\in B_{1,i}$ and $j\ne i$, the estimate \eqref{eq_b1i_degree} gives $	d_{B_{1,j}}(z)\ge \Bigl(\frac1p-21\varepsilon\Bigr)D $, while $|B_{1,j}|\le |V_j|\le(1/p+2\varepsilon) D$.  Thus, each previously chosen vertex misses at most $23\varepsilon D$ vertices of $B_{1,j}$.  Therefore, the common candidate set
	\begin{equation*}
		T_j:=B_{1,j}\cap N_{G^{\star}}(u)\cap N_{G^{\star}}(u')\cap
		\bigcap_{i<j}\bigcap_{z\in S_i}N_{G^{\star}}(z)
	\end{equation*}
	satisfies
	\begin{equation*}
		|T_j|\ge \Bigl(\frac{1}{p}-6\varepsilon-23(j-1)f\varepsilon\Bigr)D>f.
	\end{equation*}
	Choose $S_j\subseteq T_j$ with $|S_j|=f$.  This induction gives sets $S_j\subseteq B_{1,j}$, $j\in[p]$, such that $u$ and $u'$ are adjacent to every vertex of $S_1\cup\cdots\cup S_p$ and the sets $S_1,\ldots,S_p$ are mutually complete.
	
	Since every vertex of $B_1$ has at most $\varepsilon |A|$ non-neighbours in $A$, we have
	\begin{equation*}
		\left|A\cap\bigcap_{j=1}^p\bigcap_{w\in S_j}N_{G^{\star}}(w)\right|
		\ge (1-pf\varepsilon)|A|>f.
	\end{equation*}
	Choose $S_0\subseteq A\cap\bigcap_{j=1}^p\bigcap_{w\in S_j}N_{G^{\star}}(w)$, and $|S_0|=f$. Now choose a critical edge $ab$ of $F$ and a proper $(p+1)$-coloring of $F-ab$ in which $a$ and $b$ have the same color.  Embed $a$ and $b$ on $u$ and $u'$, embed the remaining vertices of their color class into $S_0$, and embed the other $p$ color classes into $S_1,\ldots,S_p$.  The edge $ab$ is realised by $uu'$, every required edge between distinct color classes is present by construction, and extra edges are harmless because we seek a non-induced copy.  This gives $F\subseteq G^{\star}$, a contradiction.
\end{proof}

\begin{lemma}\label{lem:right-envelope-comparison}
	Let $\widehat G:=|A|K_1\vee K_{|V_1|,\ldots,|V_p|}$. Then
	$\widehat G$ is $F$-free, $\Delta(\widehat G)= D$, and
	$\rho(\widehat G)\le\rho(G^{\star})$.  If $\boldsymbol y$ is the Perron vector of
	$\widehat G$ normalized by $\max_{v\in V(\widehat{G})} y_v=1$, then $y_v\geq \frac{1}{4}$ for every
	vertex $v$.
\end{lemma}
\begin{proof}
	The graph $\widehat G$ is complete $(p+1)$-partite and is therefore $F$-free. By \cref{lem:B-partition},
	$|V_i|\ge (1/p-2\varepsilon) D$ for every $i\in[p]$. On the other hand,
	\eqref{eq:right-beta-bound} gives
	\begin{equation*}
		|A|=n-D=\beta D
		<\frac{\varepsilon^2}{100}D
		<(1/p-2\varepsilon)D
		\le |V_i|.
	\end{equation*}
	Thus $A$ is a smallest part of $\widehat G$, and hence $\Delta(\widehat G)=n-|A|=D$. Therefore, $\widehat G\in\mathfrak{U}_{n}(\Delta,\zeta_F;F)$. Then the upper-tail extremality gives $\rho(\widehat G)\le \rho(G^{\star})$.
	
%	It remains to record a uniform lower bound for the Perron entries of
%	$\widehat G$. Set $s_0:=|A|$, and $s_i:=|V_i|\,\ (i\in[p])$. 

	Since 	$\widehat G[B]=K_{|V_1|,\ldots,|V_p|}$ and the parts $V_i$ are within
	$2\varepsilon D$ of $D/p$, the Rayleigh quotient on $B$ gives 
	\begin{equation*}
		\rho(\widehat G)>\rho(\widehat{G}[B])\geq\frac{2e(\widehat{G}[B])}{D}=\frac{2\sum_{1\leq i<j\leq p}|V_i||V_j|}{D}\geq (\frac{p-1}{p}-2\varepsilon(p-1)+4p(p-1)\varepsilon^2) D>\frac{9(p-1)}{10p} D.
	\end{equation*}
	for all sufficiently large $D$ and small $\varepsilon$. 
	
	%Also each part of $\widehat G$ has size at most $D$.

	%By \eqref{eq:esti-radius-lb}, 
	
	By symmetry, the Perron vector $\boldsymbol y$ is constant on each part of
	$\widehat G$. Let $y_0$ be its common value on $A$, and let $y_i$ be
	its common value on $V_i$ for each $i\in[p]$. Then the eigenvalue equations for the complete multipartite graph
	$\widehat G$ are 
	\begin{equation*}
		(\rho(\widehat G)+s_0)y_0=s_0 y_0+\sum_{i=1}^p s_i y_i, \quad (\rho(\widehat G)+s_i)y_i=s_0 y_0+\sum_{i=1}^p s_i y_i \quad\text{for}\quad i\in[p].
	\end{equation*}
%	and for $$, 
%	\begin{equation*}
%		
%	\end{equation*}
	Hence, for every $i\in[p]$, $\frac{y_i}{y_0}=\frac{\rho(\widehat G)+s_0}{\rho(\widehat G)+s_i}$.
	Since $s_0=|A|\le s_i=|V_i|$ for every $i$, we have $y_0\ge y_i$ for
	all $i$. Because $\boldsymbol y$ is normalized by $\max_{v\in V(\widehat{G})} y_v=1$, it follows
	that $y_0=1$. Therefore, for every $i\in[p]$,
	\begin{equation*}
		y_i
		=\frac{\rho(\widehat G)+s_0}{\rho(\widehat G)+s_i}
		\ge \frac{\rho(\widehat G)}{\rho(\widehat G)+s_i}
		\ge \frac{\rho(\widehat G)}{\rho(\widehat G)+D}
		\ge \frac{9(p-1)/(10p)}{9(p-1)/(10p)+1}
		= \frac{9(p-1)}{19p-9}\geq \frac{1}{4}.
	\end{equation*}
	Vertices of $A$ have entry $y_0=1$, while vertices of each $V_i$ have entry
	$y_i\ge \frac{1}{4}$. Thus $y_v\ge \frac{1}{4}$ for every vertex
	$v\in V(\widehat G)$.
\end{proof}

\begin{lemma}\label{lem:A2_B2_empty}
	We have $ A_2\cup B_2=\emptyset $.
\end{lemma}
\begin{proof}
	Let $\widehat G:=|A|K_1\vee K_{|V_1|,\ldots,|V_p|}$, and let
	$\boldsymbol y$ be its Perron vector normalized by $\max_{v\in V(\widehat{G})} y_v=1$.  By
	\Cref{lem:right-envelope-comparison}, $\widehat{G}\in\mathfrak{U}_{n}(\Delta,\zeta_F;F)$, $\rho(\widehat G)\le\rho(G^{\star})$, and $y_v\ge \frac{1}{4}$ for every
	vertex $v$.  Hence
	\begin{equation}\label{eq:right-envelope-bilinear-nonpositive}
		\boldsymbol y^\top (A(\widehat{G})-A(G^{\star}))\boldsymbol x
		=(\rho(\widehat{G})-\rho(G^{\star}))\boldsymbol y^\top\boldsymbol x\le 0 .
	\end{equation}
	By~\eqref{eq:sum_internal_edges},
	$
	\sum_{i=1}^p e(G^{\star}[V_i])\leq \Lambda_F D^{s_F}|B_2|.
	$
	On the other hand, \Cref{lem:B_perron_lower} gives $x_v\geq \frac{1}{6p}$ for every
	$v\in B$, while \Cref{lem:right-envelope-comparison} gives $y_u\ge \frac{1}{4}$
	for every $u\in A$.  Therefore
	\begin{align*}
		\boldsymbol y^\top (A(\widehat G)-A(G^{\star}))\boldsymbol x
		&\geq
		\frac12\sum_{u\in A_2}\sum_{v\in B\setminus N_{G^{\star}}(u)}(y_u x_v+x_u y_v)\\
		&\quad
		+\frac12\sum_{v\in B_2}\sum_{u\in A\setminus N_{G^{\star}}(v)}(y_u x_v+x_u y_v)
		-2e(G^{\star}[A])-2\sum_{i=1}^p e(G^{\star}[V_i])\\
		&\geq
		\frac{1}{48p}\varepsilon D |A_2|
		+\frac{1}{48p}\varepsilon|A||B_2|
		-2e(G^{\star}[A])-2\sum_{i=1}^p e(G^{\star}[V_i]).
	\end{align*}
	Since $E(G^{\star}[A_1])=\emptyset$, every edge in $G^{\star}[A]$ is incident with $A_2$.
	Therefore $2e(G^{\star}[A])\leq 2|A||A_2|=2\beta D |A_2| \leq \frac{\varepsilon^2}{50}D|A_2|$ by \eqref{eq:right-beta-bound}. Moreover, by \eqref{eq:sum_internal_edges}, $\sum_{i=1}^p e(G^{\star}[V_i])\leq \Lambda_F|B_2|D^{s_F}$. Then by \eqref{eq:small-constant-choice}, we obtain
	\begin{equation*}
		\boldsymbol y^\top (A(\widehat G)-A(G^{\star}))\boldsymbol x \geq \varepsilon\left(\frac{1}{48p}-\frac{\varepsilon}{50}\right)D|A_2|+\left(\frac{\zeta_F\varepsilon}{48p}n^{s_F}-2\Lambda_F D^{s_F}\right)|B_2|>0.
	\end{equation*}
	Thus, if $A_2\cup B_2\neq\emptyset$, then
	$\boldsymbol y^\top (A(\widehat G)-A(G^{\star}))\boldsymbol x>0$, contradicting
	\eqref{eq:right-envelope-bilinear-nonpositive}.  Therefore
	$A_2\cup B_2=\emptyset$.
\end{proof}

\Cref{lem:A2_B2_empty} implies that $A=A_1$ and $B=B_1$.
Combining this with \Cref{lem:B1i-independent-right,lem_nob1k,lem:A1_structure},
we conclude that $A$ and each $V_i$ are independent in $G^{\star}$.  Hence $G^{\star}$ is a spanning subgraph of
\[
        \widehat G=|A|K_1\vee K_{|V_1|,\ldots,|V_p|}.
\]
By \Cref{lem:radius-subgraph} and \Cref{lem:model-spectral-comparison}(ii), upper-tail extremality forces
$G^{\star}\cong S_{n,D,p}$.

\begin{proposition}\label{prop:right-range-exact} Let $G^{\star}\in\mathfrak U_n(\Delta,\zeta_F;F)$ be an upper-tail extremal graph and put $D=\Delta(G^{\star})$.  If $\left(1-\frac{\varepsilon^2}{200}\right)n\le D$ and $n-D\ge\kappa_Fn^{s_F}$, then $G^{\star}\cong S_{n,D,p}$ for all sufficiently large $n$.
\end{proposition}
\begin{proof}
This is exactly the conclusion of the right-end cleaning argument above; all estimates there are expressed in terms of the actual degree $D$ and the buffer $n-D\ge\kappa_Fn^{s_F}$.
\end{proof}

\subsection{The non-right range}\label{subsec:non-right}
In this subsection we keep the notation fixed in
\eqref{eq:extremal-uptail-decomposition} and assume that the actual degree lies
in the non-right range,
\[
\left\lceil\frac{p}{p+1}n\right\rceil
\le D<\left(1-\frac{\varepsilon^2}{200}\right)n.
\]

%Thus
%$B=N_{G^{\star}}(u^*)$, $A=V(G^{\star})\setminus B$, $|B|=D$, and
%$|A|=n-D$.  We assume
%\[
%        \left\lceil\frac{p}{p+1}n\right\rceil\le D<
%        \left(1-\frac{\varepsilon^2}{200}\right)n .
%\]
Put, as before, $\alpha:=\frac{p-1}{p}$, and $\beta:=\frac{n-D}{D}=\frac{|A|}{D}$.
Then
\begin{equation}\label{eq:estimate-beta-non-right}
        \frac{\varepsilon^2}{200}<\beta\le\frac1p .
\end{equation}
Since $S_{n,D,p}\in\mathfrak U_n(\Delta,\zeta_F;F)$, upper-tail extremality and
\eqref{eq:esti-radius-lb} give
\begin{equation}\label{eq:middle-esti-radius}
        \rho(G^{\star})\ge \rho(S_{n,D,p})
        >(\alpha+\beta-\varepsilon^3)D
\end{equation}
for all sufficiently large $n$.

The definitions of the exceptional sets are adjusted to the scale that is actually
available in the non-right range.  We use the following notations:
\begin{align}
 A_1&:=\{u\in A:d_B(u)\ge(1-4\varepsilon^3)D\},
 & A_2&:=A\setminus A_1,\label{eq:nonright-A12}\\
 B_1&:=\{v\in B:d_A(v)\ge(1-(10^4p^2f^2)^{-1})|A|\},
 & B_2&:=B\setminus B_1,\label{eq:nonright-B12}\\
 L&:=\{v\in B:d_B(v)\le(1-\frac{1}{p}-40(10^4p^2f^2)\varepsilon^3)D\},
 &K_A&:=\{u\in A:d_A(u)\ge4\varepsilon^3D\},\label{eq:nonright-LKA}\\
 K_i&:=\{v\in V_i:d_{V_i}(v)\ge(10^4p^2f^2)\varepsilon^3D\},
 &K&:=\bigcup_{i=1}^pK_i.\label{eq:nonright-K}
\end{align}
As in \cref{subsec:right}, put
\begin{equation*}
        V_i':=V_i\setminus(L\cup K),\quad
        B'_{s,i}:=B_s\cap V_i'\quad (s\in[2],\ i\in[p]).
\end{equation*}
Recall from \eqref{eq:esitimate-Miss-int-edge} and \Cref{lem:B-partition} that,
with $I=e(G^{\star}[A])$ and
$M=(n-D)D-e_{G^{\star}}(A,B)$,
\begin{equation}\label{eq:nonright-basic-defect}
        M+I<4\varepsilon^6D^2,
        \quad
        \sum_{i=1}^p e(G^{\star}[V_i])\le\varepsilon^6D^2,
        \quad
        \left||V_i|-\frac Dp\right|\le2\varepsilon^3D .
\end{equation}
We shall repeatedly use the rooted inequality
\begin{equation}\label{eq:nonright-root-defect}
        M\ge2I
\end{equation}
from \Cref{lem:root-defect-main}.

\begin{lemma}\label{lem:non-right-structural}
The following statements hold for all sufficiently large $n$.
\begin{enumerate}[label=\textnormal{(\roman*)}]
\item $|A_2|<\varepsilon^3D$, $|B_2|<\varepsilon^3D$, $|L|<\varepsilon^3D$, $|K_A|<\varepsilon^3D$, and $|K|<\varepsilon^3D$.
\item If $v\in V_i'$ and $j\ne i$, then
\begin{equation}\label{eq:nonright-good-cross-degree}
 d_{V_j'}(v)\ge\left(\frac1p-50(10^4p^2f^2)\varepsilon^3\right)D,
 \quad
 d_{B'_{1,j}}(v)\ge\left(\frac1p-51(10^4p^2f^2)\varepsilon^3\right)D.
\end{equation}
\item $ E(G^{\star}[A_1])=\emptyset$, and $E(G^{\star}[B'_{1,i}])=\emptyset\quad(i\in[p])$.
%\begin{equation}\label{eq:nonright-good-independent}
% E(G^{\star}[A_1])=\emptyset,
% \quad
% E(G^{\star}[B'_{1,i}])=\emptyset\quad(i\in[p]).
%\end{equation}
\item The total number of missing crossing pairs between the $V_i$ satisfies
\begin{equation}\label{eq:nonright-cross-missing-total}
 \sum_{1\le i<j\le p}
 \bigl(|V_i||V_j|-e_{G^{\star}}(V_i,V_j)\bigr)
 \le2\varepsilon^6D^2.
\end{equation}
\end{enumerate}
\end{lemma}
\begin{proof}
\textnormal{(i)}: By \eqref{eq:nonright-A12}, for $u\in A_2$, $D-d_{B}(u)> 4\varepsilon^3 D$.  Hence $4\varepsilon^3D|A_2|<M<4\varepsilon^6D^2$, which gives $|A_2|<\varepsilon^3D$.
For $B_2$, \eqref{eq:nonright-B12} and the identity
$M=\sum_{v\in B}(|A|-d_A(v))$ give $(10^4p^2f^2)^{-1}|A||B_2|<M<4\varepsilon^6D^2.$
By \eqref{eq:estimate-beta-non-right}, $|A|=\beta D>\varepsilon^2D/200$, and therefore
%\label{eq:nonright-B2-size-calc}
\begin{equation*}
        |B_2|<\frac{4(10^4p^2f^2)\varepsilon^6D^2}{|A|}
        <800(10^4p^2f^2)\varepsilon^4D<\varepsilon^3D,
\end{equation*}
where the last inequality is part of the choice of $\varepsilon$.
Moreover, $2I=\sum_{u\in A}d_A(u)
\ge4\varepsilon^3D|K_A|$, so $|K_A|<\varepsilon^3D$.  Likewise,
\begin{equation*}
	2\sum_{i=1}^p e(G^{\star}[V_i])
	\ge\sum_{i=1}^p\sum_{v\in K_i}d_{V_i}(v)
	\ge(10^4p^2f^2)\varepsilon^3D|K|,
\end{equation*}
hence
%\label{eq:nonright-K-size-calc}
\begin{equation*}
        |K|\le\frac{2\varepsilon^6D^2}
        {(10^4p^2f^2)\varepsilon^3D}
        =\frac{2}{(10^4p^2f^2)}\varepsilon^3D<\varepsilon^3D.
\end{equation*}

It remains to bound $L$.  Suppose $|L|\ge\varepsilon^3D$, and choose
$S\subseteq L$ with $|S|=\lfloor\varepsilon^3D\rfloor$.  By
\Cref{lem:B-partition},
\begin{align*}
 e(G^{\star}[B\setminus S])\ge e(G^{\star}[B])-\sum_{v\in S}d_B(v)\ge t(D,p)-\varepsilon^6D^2-|S|(1-\frac{1}{p}-40(10^4p^2f^2)\varepsilon^3)D.
\end{align*}
Using
$t(D,p)\ge \frac{\alpha}{2}D^2-\frac p8$ and
$t(D-|S|,p)\le\frac{\alpha}{2}(D-|S|)^2$, we obtain
\begin{align*}
 e(G^{\star}[B\setminus S])-t(D-|S|,p)
 &\ge 40(10^4p^2f^2)\varepsilon^3|S|D-\varepsilon^6D^2
      -\frac{\alpha}{2}|S|^2-\frac p8\\
 &\ge (40(10^4p^2f^2)-2)\varepsilon^6D^2>0
\end{align*}
for sufficiently large $n$.  In fact, the last quantity is a fixed positive multiple of $(D-|S|)^2$, since $\varepsilon$ is fixed.  By the Erd\H{o}s--Stone theorem (\cref{thm:erdos-stone}), $G^{\star}[B\setminus S]$ contains a complete $(p+1)$-partite graph with all classes larger than $f$.  It therefore contains a member of $\partial_cF$, and adjoining the root $u^*$ produces a copy of $F$, a contradiction.  Hence $|L|<\varepsilon^3D$.  This proves (i).

\textnormal{(ii)}: Let $v\in V_i'$ and $j\ne i$.  By \eqref{eq:nonright-LKA} and \eqref{eq:nonright-K},
\begin{equation*}
	d_B(v)>(1-\frac{1}{p}-40(10^4p^2f^2)\varepsilon^3)D,
	\quad d_{V_i}(v)<(10^4p^2f^2)\varepsilon^3D.
\end{equation*}
Also, by \eqref{eq:nonright-basic-defect}, 
\begin{equation*}
	 \sum_{k\in[p]\setminus\{i,j\}}|V_k|=|B|-|V_i|-|V_j|\leq\left(1-\frac2p+4\varepsilon^3\right)D.
\end{equation*}
Consequently, using (i),
\begin{align*}
 d_{V_j'}(v)
 &\ge d_B(v)-d_{V_i}(v)
      -\sum_{k\in[p]\setminus\{i,j\}}|V_k|-|L|-|K|\\
 &>\left(\frac1p-(41(10^4p^2f^2)+6)\varepsilon^3\right)D\ge\left(\frac1p-50(10^4p^2f^2)\varepsilon^3\right)D.
\end{align*}
Subtracting $|B_2|<\varepsilon^3D$ gives the second inequality in
\eqref{eq:nonright-good-cross-degree}.

\textnormal{(iii)}:  If $a\in A_1$, then
for every $j$, %$|B'_{1,j}\setminus N_{G^{\star}}(a)|\le D-d_B(a)\le4\varepsilon^3D$.
\begin{equation}\label{eq:nonright-A1-to-Bgood}
 |B'_{1,j}\setminus N(a)|\le D-d_B(a)\le4\varepsilon^3D.
\end{equation}
If $v\in B'_{1,i}$ and $j\ne i$, then 
%$|B'_{1,j}\setminus N_{G^{\star}}(v)|\le |V_j|-d_{B'_{1,j}}(v)\le 53(10^4p^2f^2)\varepsilon^3D.$
\begin{equation}\label{eq:nonright-Bgood-to-Bgood}
 |B'_{1,j}\setminus N_{G^{\star}}(v)|
 \le |V_j|-d_{B'_{1,j}}(v)
 \le 53(10^4p^2f^2)\varepsilon^3D.
\end{equation}
Finally, for $v\in B_1$, 
%$ |A_1\setminus N(v)|\le |A|-d_A(v)\le(10^4p^2f^2)^{-1}|A| $.
\begin{equation*}
 |A_1\setminus N(v)|
 \le |A|-d_A(v)\le(10^4p^2f^2)^{-1}|A|.
\end{equation*}
Similar to the proof of \cref{lem:B1i-independent-right} and \cref{lem:A1_structure}, for a set $T$ of already embedded vertices to which the next vertex is required to be adjacent, we estimate the corresponding common neighbourhood.Thus, if at most $f$ already chosen vertices have to be joined to a new
reservoir $B'_{1,j}$, then we get
\begin{align}
 \left|B'_{1,j}\cap\bigcap_{w\in T}N_{G^{\star}}(w)\right|
 &\ge |B'_{1,j}|-53f(10^4p^2f^2)\varepsilon^3D-4f\varepsilon^3D\nonumber\\
 &\ge \frac{D}{p}-6\varepsilon^3D
      -57f(10^4p^2f^2)\varepsilon^3D>f.\label{eq:nonright-B-reservoir}
\end{align}
Similarly, for at most $pf$ vertices of $B_1$,
\begin{align}
 \left|A_1\cap\bigcap_{w\in T}N_{G^{\star}}(w)\right|
 &\ge |A|-|A_2|-pf(10^4p^2f^2)^{-1}|A|\nonumber\\
 &\ge \left(1-\frac{pf}{(10^4p^2f^2)}
      -200\varepsilon\right)|A|>f,
 \label{eq:nonright-A-reservoir}
\end{align}
where we used $|A|>\varepsilon^2D/200$ and $|A_2|<\varepsilon^3D$.

Now let $xy$ be a color-critical edge of $F$, so that $F-xy$ has a proper
$(p+1)$-coloring in which $x$ and $y$ have the same color.  If $A_1$ contained
an edge $ab$, map $x,y$ to $a,b$.  The other $p$ color classes are chosen,
one at a time, inside $B'_{1,1},\ldots,B'_{1,p}$ using
\eqref{eq:nonright-B-reservoir}; any remaining vertices in the color class of
$x,y$ are then chosen in $A_1$ using \eqref{eq:nonright-A-reservoir}.  This
embeds $F$, a contradiction.  If $B'_{1,i}$ contained an edge, use it for the
critical pair, use $A_1$ for one of the other color classes and the remaining
$B'_{1,j}$ for the other $p-1$ classes.  The same two reservoir estimates give
an embedding of $F$.  Hence (iii) holds.

\textnormal{(iv)}: Simply, note that
\begin{align*}
 \sum_{1\le i<j\le p}
 \bigl(|V_i||V_j|-e_{G^{\star}}(V_i,V_j)\bigr)&=\sum_{1\le i<j\le p}|V_i||V_j|
   -e(G^{\star}[B])+\sum_{i=1}^pe(G^{\star}[V_i])\\
 &\le t(D,p)-\bigl(t(D,p)-\varepsilon^6D^2\bigr)
       +\varepsilon^6D^2
 \le2\varepsilon^6D^2.
\end{align*}
This completes the proof.
\end{proof}

The next lemma removes $K$ completely. This is where the global
crossing-defect estimate in \eqref{eq:nonright-cross-missing-total} is used; it
avoids any attempt to deduce $K\subseteq B_2$ from a Perron estimate.

\begin{lemma}\label{lem:nonright-K-empty}
We have $K=\emptyset$.
\end{lemma}
\begin{proof}
Suppose $v\in K_i$.  The maximum-cut choice of $V_1,\ldots,V_p$ gives
$d_{V_i}(v)\le d_{V_j}(v)$ for every $j\ne i$.  Hence $d_{V_j}(v)\ge(10^4p^2f^2)\varepsilon^3D$ for $j\in[p]$. By \Cref{lem:non-right-structural}(i), for every $j\in [p]$,
\begin{equation*}
 |N_{G^{\star}}(v)\cap B'_{1,j}|\ge d_{V_j}(v)-|B_2|-|L|-|K|\ge((10^4p^2f^2)-3)\varepsilon^3D
 >\frac{(10^4p^2f^2)}{2}\varepsilon^3D.
\end{equation*}
Choose $W_j\subseteq N_{G^{\star}}(v)\cap B'_{1,j}$ with $|W_j|=m:=\left\lfloor\frac{(10^4p^2f^2)}{2}\varepsilon^3D\right\rfloor$, where $j\in[p]$.
There are $\binom{m}{f}^p$ choices of $f$-subsets
$S_j\subseteq W_j$ for all $j\in[p]$.  
%A fixed missing edge between two different $W_i,W_j$ destroys at most $\binom{m-1}{f-1}^2\binom mf^{p-2}$ such choices. 
Let $X:=\prod_{i=1}^p \binom{W_i}{f}$. We call a choice $(S_1,\ldots,S_p)\in X$ bad if there exist $i\ne j$ and vertices $x\in S_i$, $y\in S_j$ such that $xy\notin E(G^\star)$. By \eqref{eq:nonright-cross-missing-total}, the number of missing crossing edges is at most $2\varepsilon^6D^2$. For each fixed missing crossing edge, the number of choices containing its two endpoints is $\binom{m-1}{f-1}^2\binom mf^{p-2}$. Hence, by the union bound, the proportion of bad choices in $X$ is at most
\begin{equation*}
	\frac{2\varepsilon^6D^2\binom{m-1}{f-1}^2\binom{m}{f}^{p-2}}{\binom mf^p}=2\varepsilon^6D^2\left(\frac fm\right)^2\leq\frac{9f^2}{(10^4p^2f^2)^2}<1.
\end{equation*}
Therefore at least one choice in $X$ is not bad, that is, there exist $f$-sets $S_i\subseteq W_i$ for $i\in[p]$ such that all crossing edges between distinct $S_i$ and $S_j$ are present in $G^\star$.

%By \eqref{eq:nonright-cross-missing-total}, the proportion of
%choices destroyed by at least one missing crossing edge is at most
%\begin{align*}
% 2\varepsilon^6D^2\frac{\binom{m-1}{f-1}^2\binom mf^{p-2}}{\binom{m}{f}^p}=2\varepsilon^6D^2\left(\frac{f}{m}\right)^2\le \frac{9f^2}{(10^4p^2f^2)^2}<1.
%\end{align*}
%Thus we may choose $S_1,\ldots,S_p$, each of size $f$, such that all crossing edges between distinct $S_j$ are present and every vertex of every $S_j$ is adjacent to $v$.

Every vertex of $S_1\cup\cdots\cup S_p$ lies in $B_1$.  Therefore
\begin{align*}
 \left|\bigcap_{w\in S_1\cup\cdots\cup S_p}N_A(w)\right|\ge |A|-pf(10^4p^2f^2)^{-1}|A|=\left(1-\frac{pf}{(10^4p^2f^2)}\right)|A|>f.
\end{align*}
Choose a set $S_0$ of $f$ vertices in this common neighbourhood, avoiding $u^*$.  Map a
critical edge of $F$ to $u^*v$.  In a proper $(p+1)$-coloring of the graph
obtained by deleting that critical edge, map the other $p$ color classes into
$S_1,\ldots,S_p$ and the remaining vertices in the color class of the critical
pair into $S_0$.  The root $u^*$ is adjacent to all of $B$, while $v$ is
adjacent to all $S_j$, and $S_0$ is adjacent to all $S_j$.  Hence this is a copy
of $F$ in $G^{\star}$, a contradiction.  Thus $K=\emptyset$.
\end{proof}

\begin{lemma}\label{lem:nonright-switch-certificate}
	The following hold.
	\begin{enumerate}[label=\textnormal{(\roman*)}]
		\item For $u\in A^+\setminus\{u^*\}$, we have $\rho(G^{\star})x_{u}\geq \sum_{j=1}^p\sum_{w\in B'_{1,j}}x_w$.
		\item For $v\in V_i$, where $i\in [p]$, we have $\rho(G^{\star})x_{v}\geq \sum_{u\in A_1}x_u+\sum_{j\in [p]\setminus\{i\}}\sum_{w\in B'_{1,j}}x_w$.
	\end{enumerate}
\end{lemma}
\begin{proof}
\textnormal{(i)}: Let $u\in A^+\setminus\{u^*\}$ and suppose that $ \rho(G^{\star})x_u<\sum_{j=1}^p\sum_{w\in B'_{1,j}}x_w$. Define $$G_u:=G^{\star}-\{uw:w\in N_{G^{\star}}(u)\}+\{uw:w\in\bigcup_{j=1}^{p}B'_{1,j}\}.$$
We first verify that $G_u$ is $F$-free.  If $G_u$ contained a new copy $F'$ of $F$, then
$u\in V(F')$ and $N_{F'}(u)\subseteq\bigcup_{j=1}^{p}B'_{1,j}$. By \eqref{eq:nonright-A-reservoir}, there is a vertex
\[
a\in A_1\cap\bigcap_{z\in N_{F'}(u)}N_{G^{\star}}(z)
   \setminus V(F').
\]
Replacing $u$ by $a$ gives a copy of $F$ in $G^{\star}$, a contradiction.
Thus $G_u$ is $F$-free.

Moreover, $d_{G_u}(u)=\sum_{j=1}^{p}|B'_{1,j}|\le D$, and every other vertex gains at most one incident edge.  The root $u^*$ is
unchanged and still has degree $D$.  Since $D<(1-\varepsilon^2/200)n$ and $s_F<1$, for all sufficiently large $n$ we have
$D+1\le n-\zeta_Fn^{s_F}$.  Hence
$G_u\in\mathfrak U_n(\Delta,\zeta_F;F)$.
Finally $u\in A^+$ gives $x_u>0$, and
\[
\boldsymbol x^\top(A(G_u)-A(G^{\star}))\boldsymbol x
=2x_u\left(
\sum_{j=1}^{p}\sum_{w\in B'_{1,j}}x_w
-\rho(G^{\star})x_u\right)>0,
\]
contradicting upper-tail extremality.  This proves (i).

\textnormal{(ii)}: Let $v\in V_i$ and suppose that $\rho(G^{\star})x_v<\sum_{u\in A_1}x_u+\sum_{j\in[p]\setminus\{i\}}\sum_{w\in B'_{1,j}}x_w$.
Define
\[
G_v:=G^{\star}-\{vw:w\in N_{G^{\star}}(v)\}
       +\{vw:w\in A_1\cup
       \bigcup_{j\in[p]\setminus\{i\}}B'_{1,j}\}.
\]
If $G_v$ contained a new copy $F'$ of $F$, then $v\in V(F')$ and all new
neighbours of $v$ lie in
$A_1\cup\bigcup_{j\ne i}B'_{1,j}$.  By
\eqref{eq:nonright-A1-to-Bgood} and
\eqref{eq:nonright-Bgood-to-Bgood},
\[
\left|B'_{1,i}\cap
\bigcap_{z\in N_{F'}(v)}N_{G^{\star}}(z)\right|>f.
\]
Replacing $v$ by an unused vertex in this intersection gives a copy of $F$ in
$G^{\star}$, a contradiction.  Thus $G_v$ is $F$-free.

Also
\begin{equation*}
	d_{G_v}(v)\le |A_1|+\sum_{j\ne i}|B'_{1,j}|\le |A|+D-|V_i|\le D+2\varepsilon^3D,
\end{equation*}
and every other vertex gains at most one edge.  Since $u^*\in A_1$, the edge
$u^*v$ is present after the switch, so $d_{G_v}(u^*)=D$.  By the non-right gap, $n-(D+2\varepsilon^3D+1)\ge \frac{\varepsilon^2}{400}D$.
for all sufficiently large $n$, and this is larger than
$\zeta_Fn^{s_F}$.  Hence
$G_v\in\mathfrak U_n(\Delta,\zeta_F;F)$.
Finally $v\in B\subseteq V(C^{\star})$ gives $x_v>0$, and
\[
\boldsymbol x^\top(A(G_v)-A(G^{\star}))\boldsymbol x
=2x_v\left(
\sum_{u\in A_1}x_u+
\sum_{j\in[p]\setminus\{i\}}\sum_{w\in B'_{1,j}}x_w
-\rho(G^{\star})x_v\right)>0,
\]
again contradicting upper-tail extremality.  This proves (ii).
\end{proof}

The preceding switching certificate is now used only to obtain a uniform
Perron lower bound.  In particular, we do not try to prove that a Perron-maximal
vertex of $B$ lies outside $L$.

\begin{lemma}\label{lem:nonright-perron-lower}
Let $\boldsymbol{x}=(x_v)_{v\in V(G^{\star})}$ be a Perron vector of $G^{\star}$ normalized by
$\max_{v\in V(G^{\star})}x_v=1$.  Then, for every $v\in B$ and every $u\in A^+$,
\begin{align}
 x_v&\ge\frac{\alpha+\beta}{1+\beta}-12\varepsilon^3,
 \label{eq:nonright-xB-uniform}\\
 x_u&\ge\frac{\alpha+\beta}{\alpha+2\beta}-12\varepsilon^3.
 \label{eq:nonright-xA-uniform}
\end{align}
\end{lemma}
\begin{proof}
By \Cref{lem:nonright-K-empty}, $K=\emptyset$.  
%Put
%\begin{equation*}
%	E_0:=|A_2|+|B_2|+|L|<3\varepsilon^3D,
%	\quad
%	S_A:=\sum_{u\in A_1}x_u,
%	\quad
%	S_i:=\sum_{v\in B'_{1,i}}x_v,
%	\quad
%	S:=S_A+\sum_{i=1}^pS_i.
%\end{equation*}
Let $\hat{u}\in V(G^{\star})$ satisfy $x_{\hat{u}}=1$.  Every vertex outside
$A_1\cup\bigcup_{i=1}^{p}B'_{1,i}$ belongs to $A_2\cup B_2\cup L$, and by \cref{lem:non-right-structural} (i), hence
\begin{equation}\label{eq:nonright-total-good-mass}
        \sum_{u\in A_1}x_{u}+\sum_{i=1}^{p}\sum_{w\in B'_{1,i}}x_{w}\ge \rho(G^{\star})-(|A_2|+|B_2|+|L|)>\rho(G^{\star})-3\varepsilon^3D.
\end{equation}
Fix $i$.  Since $B'_{1,i}$ is independent, every $v\in B'_{1,i}$ satisfies
\begin{equation*}
	\rho(G^{\star})x_v\leq \sum_{u\in A_1}x_{u}+\sum_{j\in [p]\setminus\{i\}}\sum_{w\in B'_{1,j}}x_{w}+ |A_2|+|B_2|+|L|
\end{equation*}
Summing over $v\in B'_{1,i}$ gives
\begin{equation*}
	\rho(G^{\star})\sum_{v\in B'_{1,i}}x_{v} \leq |B'_{1,i}|(\sum_{u\in A_1}x_{u}+\sum_{j\in [p]\setminus\{i\}}\sum_{w\in B'_{1,j}}x_{w}+ |A_2|+|B_2|+|L|),
\end{equation*}
and therefore
\begin{align*}
&\quad (\rho(G^{\star})+|B'_{1,i}|)
\left(\sum_{u\in A_1}x_{u}
+\sum_{j\in [p]\setminus\{i\}}\sum_{w\in B'_{1,j}}x_{w}\right)\\
&\geq \rho(G^{\star})
\left(\sum_{u\in A_1}x_{u}
+\sum_{i=1}^{p}\sum_{w\in B'_{1,i}}x_{w}\right)-|B'_{1,i}|(|A_2|+|B_2|+|L|)\\
&>\rho(G^{\star})
\left(\sum_{u\in A_1}x_{u}
+\sum_{i=1}^{p}\sum_{w\in B'_{1,i}}x_{w}\right)
-3\varepsilon^3 D|B'_{1,i}|.
\end{align*}
By \eqref{eq:nonright-total-good-mass},
\begin{align*}
 \sum_{u\in A_1}x_{u}+\sum_{j\in [p]\setminus\{i\}}\sum_{w\in B'_{1,j}}x_{w}
 &\ge\frac{\rho(G^{\star})(\sum_{u\in A_1}x_{u}+\sum_{i=1}^{p}\sum_{w\in B'_{1,i}}x_{w})-3\varepsilon^3 D|B'_{1,i}|}
              {\rho(G^{\star})+|B'_{1,i}|}\nonumber\\
 &\ge\frac{\rho(G^{\star})^2}
              {\rho(G^{\star})+|B'_{1,i}|}-3\varepsilon^3 D.
\end{align*}
By \cref{lem:nonright-switch-certificate} (ii), every $v\in V_i$ satisfies
\begin{equation*}
 x_v\ge
 \frac{\rho(G^{\star})}{\rho(G^{\star})+|B'_{1,i}|}
 -\frac{|A_2|+|B_2|+|L|}{\rho(G^{\star})}.
\end{equation*}
Now use $|B'_{1,i}|\le |V_i|\le(D/p+2\varepsilon^3D)$,
\eqref{eq:middle-esti-radius}, and \cref{lem:non-right-structural} (i) to get
\begin{align*}
 x_v\ge\frac{\alpha+\beta-\varepsilon^3}{1+\beta+\varepsilon^3}-8\varepsilon^3\ge\frac{\alpha+\beta}{1+\beta}-12\varepsilon^3.
\end{align*}
This proves \eqref{eq:nonright-xB-uniform}.

Since $A_1$ is independent, the same argument gives
\begin{equation*}
	\rho(G^{\star})\sum_{u\in A_1}x_u\le |A_1|(\sum_{i=1}^{p}\sum_{w\in B'_{1,i}}x_{w}+|A_2|+|B_2|+|L|),
\end{equation*}
so, using \eqref{eq:nonright-total-good-mass},
\begin{equation}\label{eq:nonright-target-A-mass}
 \sum_{i=1}^{p}\sum_{w\in B'_{1,i}}x_{w}\geq\frac{\rho(G^{\star})^2}{\rho(G^{\star})+|A_1|}-(|A_2|+|B_2|+|L|).
\end{equation}
If $u'\in A^+\setminus\{u^*\}$, \cref{lem:nonright-switch-certificate} (i) and
\eqref{eq:nonright-target-A-mass} imply
\begin{equation*}
	 x_{u'}\ge
	\frac{\rho(G^{\star})}{\rho(G^{\star})+|A_1|}
	-\frac{|A_2|+|B_2|+|L|}{\rho(G^{\star})}>\frac{\rho(G^{\star})}{\rho(G^{\star})+|A_1|}-\frac{3\varepsilon^3 D}{\rho(G^{\star})}.
\end{equation*}
For the root, $N(u^*)=B$ gives
$\rho(G^{\star})x_{u^*}=\sum_{w\in B}x_w\ge  \sum_{i=1}^{p}\sum_{w\in B'_{1,i}}x_{w}$, so the same lower
bound holds.  Since $|A_1|\le|A|=\beta D$, we obtain
\begin{equation*}
 x_{u'}\ge\frac{\alpha+\beta-\varepsilon^3}{\alpha+2\beta-\varepsilon^3}-8\varepsilon^3\ge\frac{\alpha+\beta}{\alpha+2\beta}-12\varepsilon^3.
\end{equation*}
This proves \eqref{eq:nonright-xA-uniform}.
\end{proof}

We next isolate the part of the right-end decomposition-family argument that is
still valid away from the endpoint.  The low-$B$ set $L$ is not cleared here;
instead we apply the boundary argument only to vertices outside $L$.

\begin{lemma}\label{lem:nonright-B2-internal}
For every $i\in[p]$ the following hold.
\begin{enumerate}[label=\textnormal{(\roman*)}]
\item If $v\in (B_2\setminus L)\cap V_i$, then $ d_{B'_{1,i}}(v)\le f$.
\item The graph $G^{\star}[(B_2\setminus L)\cap V_i]$ is
$\mathcal M(\partial_cF)$-free.
\item
\begin{equation*}
 \sum_{i=1}^p e\bigl(G^{\star}[V_i\setminus L]\bigr)
 \le (\Lambda_FD^{s_F}+f)|B_2\setminus L|.
\end{equation*}
\end{enumerate}
\end{lemma}
\begin{proof}
\textnormal{(i)}: By \Cref{lem:nonright-K-empty}, $K=\emptyset$.  Let $v\in(B_2\setminus L)\cap V_i$ and $j\ne i$.  Since $d_{V_i}(v)<(10^4p^2f^2)\varepsilon^3D$ and $d_B(v)>(1-\frac{1}{p}-40(10^4p^2f^2)\varepsilon^3)D$, the same calculation as in \eqref{eq:nonright-good-cross-degree}, 
%now subtracting only $B_2\cup L$ in the target part, gives
\begin{equation}\label{eq:nonright-B2-cross-good}
        d_{B'_{1,j}}(v)
        \ge\left(\frac1p-50(10^4p^2f^2)\varepsilon^3\right)D.
\end{equation}

Suppose first that $d_{B'_{1,i}}(v)>f$, and choose
$S_i\subseteq N(v)\cap B'_{1,i}$ with $|S_i|=f$.  For each $j\ne i$, construct
$S_j\subseteq B'_{1,j}$ of size $f$ successively.  At any step, the vertex $v$
has at least the quantity in \eqref{eq:nonright-B2-cross-good} neighbours in
the new reservoir, while every previously chosen vertex of $B'_1$ misses at
most $53(10^4p^2f^2)\varepsilon^3D$ vertices there.  Hence the number of choices is
at least
\begin{equation*}
	\left(\frac1p-50(10^4p^2f^2)\varepsilon^3\right)D-pf\,53(10^4p^2f^2)\varepsilon^3D>f.
\end{equation*}
Thus the $S_j$ can be chosen so that they are pairwise complete and $v$ is
adjacent to all of them.  Since every selected vertex lies in $B_1$,
\begin{equation*}
	\left|\bigcap_{w\in\bigcup_{j=1}^{p}S_j}N_A(w)\right|
	\ge\left(1-\frac{pf}{(10^4p^2f^2)}\right)|A|>f.
\end{equation*}
Choose $f$ common neighbours in $A$, avoiding $u^*$.  As in the proof of
\Cref{lem:nonright-K-empty}, map a critical edge of $F$ to $u^*v$, the other
$p$ color classes into $S_1,\ldots,S_p$, and the remaining vertices in the
critical color class into the common $A$-reservoir.  This gives a copy of $F$,
a contradiction.  Hence (i) holds.

\textnormal{(ii)}: Suppose that
$G^{\star}[(B_2\setminus L)\cap V_i]$ contains a copy of some
$F^{-}\in\mathcal M(\partial_cF)$.  We first record the common-reservoir estimate
used in the greedy choice.  If $w\in B'_{1,k}$ and $k\ne j$, then
\eqref{eq:nonright-good-cross-degree} and
$|B'_{1,j}|\le |V_j|\le D/p+2\varepsilon^3D$ give
\begin{equation*}
	|B'_{1,j}\setminus N_{G^{\star}}(w)|\le |V_j|-d_{B'_{1,j}}(w)\le 53(10^4p^2f^2)\varepsilon^3D.
\end{equation*}
If instead $w\in V(F^{-})\subseteq(B_2\setminus L)\cap V_i$, then
\eqref{eq:nonright-B2-cross-good} gives the same bound.  Thus every vertex that is
prescribed before choosing a set in a new reservoir $B'_{1,j}$ excludes at
most $53(10^4p^2f^2)\varepsilon^3D$ vertices of that reservoir.

Greedily, for each $j\ne i$, choose a set $S_j\subseteq B'_{1,j}$ of size $f$ that is complete to $V(F^{-})$ and to all previously chosen sets.  Fix one vertex $w_0\in V(F^{-})$.  At a given step let $P$ be the union of the sets chosen at earlier steps.  Besides $w_0$, the set $(V(F^{-})\setminus\{w_0\})\cup P$ contains at most $pf$ prescribed vertices.
Hence
\begin{align*}
 \left|B'_{1,j}\cap
   \bigcap_{w\in V(F^{-})\cup P}N_{G^{\star}}(w)\right|&\ge d_{B'_{1,j}}(w_0)-pf\,53(10^4p^2f^2)\varepsilon^3D\\
 &\ge
 \left(\frac1p-50(10^4p^2f^2)\varepsilon^3\right)D
 -pf\,53(10^4p^2f^2)\varepsilon^3D>f.
\end{align*}
Thus the greedy construction succeeds at every step.  By the definition of
$\mathcal M(\partial_cF)$,
$G^{\star}[V(F^{-})\cup\bigcup_{j\ne i}S_j]$ contains a member of
$\partial_cF$.  All its vertices lie in $B=N_{G^{\star}}(u^*)$, so adjoining $u^*$ yields
a copy of $F$, a contradiction.  This proves (ii).

\textnormal{(iii)}: Finally, because $K=\emptyset$ and $B'_{1,i}$ is independent,
$V_i\setminus L$ is the disjoint union of $B'_{1,i}$ and
$(B_2\setminus L)\cap V_i$.  Therefore (i), (ii), and
\eqref{eq:boundary-edge-cap} give
\begin{equation*}
 e(G^{\star}[V_i\setminus L])\le \ex\bigl(|(B_2\setminus L)\cap V_i|,\mathcal M(\partial_cF)\bigr)+f|(B_2\setminus L)\cap V_i|\le (\Lambda_FD^{s_F}+f)|(B_2\setminus L)\cap V_i|.
\end{equation*}
Summing over $i$ proves \textnormal{(iii)}.
\end{proof}

We now carry out the bilinear envelope comparison.  Unlike the invalid
middle-$L$ switching, the $A$-side Perron mass is not counted as a local gain.
Instead, the global inequality $M\ge2I$ pays for all edges inside $A$, while
missing crossing edges incident with $L$ pay for all internal $B$-edges that
touch $L$.

Let $\widehat G:=|A|K_1\vee K_{|V_1|,\ldots,|V_p|}$. Let $\boldsymbol y$ be a Perron vector of $\widehat G$, constant on every part, normalized so that its value on $|A|K_1$ is $y_0=1$, and write $y_i$ for its value on $V_i$.
\begin{lemma}\label{lem:nonright-envelope}
  We have $\widehat G\in\mathfrak U_n(\Delta,\zeta_F;F)$. For $i\in[p]$, $\frac{\alpha+2\beta}{1+\beta}-10\varepsilon^3\le y_i\le1+5\varepsilon^3$.
%\begin{align}
% \frac{\alpha+2\beta}{1+\beta}-10\varepsilon^3
% \le y_i\le1+5\varepsilon^3.\label{eq:nonright-y-bounds}
%\end{align}
%Moreover $\widehat G\in\mathfrak U_n(\Delta,\zeta_F;F)$ and
%\begin{equation}\label{eq:nonright-envelope-bilinear}
% \boldsymbol y^\top(A(\widehat G)-A(G^{\star}))\boldsymbol x\le0.
%\end{equation}
\end{lemma}
\begin{proof}
The graph $\widehat G$ is $(p+1)$-partite and hence $F$-free. Recall that $\beta=\frac{n-D}{D}$.  A vertex of $|A|K_1$ has degree $D$, while a vertex of $V_i$ has degree $n-|V_i|=D+\beta D-|V_i|\le D+2\varepsilon^3D$, because $\beta D\le D/p$ and $|V_i|\ge D/p-2\varepsilon^3D$ by \eqref{eq:estimate-beta-non-right} and \eqref{eq:nonright-basic-defect}.  Therefore
\begin{equation*}
	n-\Delta(\widehat G)\ge \beta D-2\varepsilon^3D\ge\left(\frac{\varepsilon^2}{200}-2\varepsilon^3\right)D\ge\frac{\varepsilon^2}{400}D.
\end{equation*}
For sufficiently large $n$ this exceeds $\zeta_Fn^{s_F}$, so $\widehat G \in \mathfrak{U}_{n}(\Delta,\zeta_F;F)$.  Hence
$\rho(\widehat G)\le\rho(G^{\star})$,  and
\begin{equation}\label{eq:nonright-envelope-bilinear}
	 \boldsymbol y^\top(A(\widehat G)-A(G^{\star}))\boldsymbol x
	=(\rho(\widehat G)-\rho(G^{\star}))\boldsymbol y^\top\boldsymbol x \leq 0.
\end{equation}

%Set $\widehat\rho:=\rho(\widehat G)$.  
Since
\begin{equation*}
	\sum_{1\leq i<j\leq p}|V_i||V_j|
	\ge e(G^{\star}[B])-\sum_{i=1}^{p} e(G^{\star}[V_i])
	\ge t(D,p)-2\varepsilon^6D^2,
\end{equation*}
we have, using the average-degree lower bound for spectral radius,
\begin{align*}
\rho(\widehat G)\ge\frac{2e(\widehat G)}{n}\ge\frac{2|A|D+2t(D,p)-4\varepsilon^6D^2}{n}\ge\frac{\alpha+2\beta-5\varepsilon^6}{1+\beta}D\ge(\alpha+\beta-5\varepsilon^6)D.
\end{align*}
%The last step uses
%\[
% \frac{\alpha+2\beta}{1+\beta}-(\alpha+\beta)
% =\frac{\beta(1/p-\beta)}{1+\beta}\ge0.
%\]
For a complete multipartite graph, the quotient equations give
\begin{equation}\label{eq:nonright-y-quotient}
        y_i=\frac{\rho(\widehat G)+|A|}{\rho(\widehat G)+|V_i|}.
\end{equation}
If $|A|>|V_i|$, then $y_i>1$, and the lower bound on $y_i$ follows because
$(\alpha+2\beta)/(1+\beta)\le1$.  If $|A|\le|V_i|$, the right side of
\eqref{eq:nonright-y-quotient} is increasing in $\rho(\widehat G)$, and hence
\begin{equation*}
 y_i\ge\frac{(\alpha+\beta-5\varepsilon^6)D+\beta D}{(\alpha+\beta-5\varepsilon^6)D+(\frac{1}{p}+2\varepsilon^3)D}= \frac{\alpha+2\beta-5\varepsilon^6}
 {1+\beta+2\varepsilon^3-5\varepsilon^6}\ge\frac{\alpha+2\beta}{1+\beta}-10\varepsilon^3.
\end{equation*}
For the upper bound, if $|A|\le|V_i|$ then $y_i\le1$; otherwise
$0<|A|-|V_i|\le2\varepsilon^3D$, and
\begin{equation*}
	y_i=1+\frac{|A|-|V_i|}{\rho(\widehat G)+|V_i|}\leq 1+5\varepsilon^3.
\end{equation*}
This completes the proof.
\end{proof}

Recall that $\boldsymbol{x}=(x_v)_{v\in V(G^{\star})}$ be a Perron vector of $G^{\star}$. 
\begin{lemma}\label{lem:esitimate-x-y}
	For every $u\in A^+$, $v\in V_i$, and $w\in V_j$ with $i\ne j$,
	\begin{enumerate}[label=\textnormal{(\roman*)}]
		\item $x_{v}+y_i x_{u}\geq 1+\frac{\varepsilon^2}{1000}$,
		\item $y_i x_{w}+y_jx_{v} \geq\frac{2}{5}$,
		\item $y_i(x_v+x_w)\leq 3$.
	\end{enumerate}
%	\begin{align}
%		x_u+y_i x_a&\ge1+\frac{\varepsilon^2}{1000},
%		\label{eq:nonright-AB-weight}\\
%		y_i x_v+y_jx_u&\ge\frac25,
%		\label{eq:nonright-BB-weight}\\
%		y_i(x_u+x_v)&\le3.
%		\label{eq:nonright-internal-weight}
%	\end{align}
\end{lemma}
\begin{proof}
	\textnormal{(i)}: Combining \Cref{lem:nonright-perron-lower} and \cref{lem:nonright-envelope}, we obtain
	\begin{equation*}
		x_{v}+y_i x_{u}\ge
		\left(\frac{\alpha+\beta}{1+\beta}-12\varepsilon^3\right)
		+\left(\frac{\alpha+2\beta}{1+\beta}-10\varepsilon^3\right)
		\left(\frac{\alpha+\beta}{\alpha+2\beta}-12\varepsilon^3\right)\ge 2\frac{\alpha+\beta}{1+\beta}-34\varepsilon^3.
	\end{equation*}
	Now, $\frac{2(\alpha+\beta)}{1+\beta}-1=\frac{2\alpha-1+\beta}{1+\beta}$. If $p=2$, then $\beta\le1/2$ and
	\eqref{eq:estimate-beta-non-right} gives $\frac{\beta}{1+\beta}>\frac{\varepsilon^2/200}{3/2}=\frac{\varepsilon^2}{300}$.
	Thus
	\begin{equation*}
		x_{v}+y_ix_{u}>1+\frac{\varepsilon^2}{300}-34\varepsilon^3\ge1+\frac{\varepsilon^2}{1000},
	\end{equation*}
	where the last inequality is exactly
	$34\varepsilon\le 7/3000$, one of the local smallness requirements included
	when $\varepsilon$ was fixed before \eqref{eq:small-constant-choice}.  If
	$p\ge3$, then $2\alpha-1=(p-2)/p$ and $\beta\le1/p$, so
	\begin{equation*}
		\frac{2\alpha-1+\beta}{1+\beta}\ge\frac{p-2}{p+1}\ge\frac{1}{4}.
	\end{equation*}
	For the same initial choice of $\varepsilon$,
	$1/4-34\varepsilon^3\ge\varepsilon^2/1000$.  This proves (i) in all cases.
	
	\textnormal{(ii)}: 
%Set
%	\[
%	a_0:=\frac{\alpha+2\beta}{1+\beta},\quad
%	b_0:=\frac{\alpha+\beta}{1+\beta}.
%	\]
	Because $\alpha\ge1/2$ and $0\le\beta\le1/p$, we have
	$1/2\leq \frac{\alpha+\beta}{1+\beta} \leq\frac{\alpha+2\beta}{1+\beta}\leq 1$.  Consequently $2\frac{(\alpha+2\beta)(\alpha+\beta)}{(1+\beta)^2}\geq 1/2$, and
	\begin{align*}
		y_i x_{w}+y_jx_{v}\ge2(\frac{\alpha+2\beta}{1+\beta}-10\varepsilon^3)(\frac{\alpha+\beta}{1+\beta}-12\varepsilon^3)\ge2\frac{(\alpha+2\beta)(\alpha+\beta)}{(1+\beta)^2}-44\varepsilon^3\ge\frac12-44\varepsilon^3\ge\frac25,
	\end{align*}
	where the last inequality follows from the same fixed smallness choice of $\varepsilon$.  This proves (ii).
	
	\textnormal{(iii)}:	Finally, $x_{v},x_{w}\le1$ and $y_i\le1+5\varepsilon^3$ give $y_i(x_v+x_w)\le2(1+5\varepsilon^3)\le3,$
	again for the globally fixed sufficiently small $\varepsilon$.  This proves (iii).
\end{proof}

\begin{lemma}\label{lem:nonright-cleaning}
We have $B_2=L=\emptyset$. And in fact, $ G^{\star}=|A|K_1\vee K_{|V_1|,\ldots,|V_p|} $.
Consequently $G^{\star}\cong S_{n,D,p}$.
\end{lemma}
\begin{proof}
Use the notation of \Cref{lem:nonright-envelope}. A direct calculation gives
\begin{equation}\label{eq:nonright-bilinear-expansion}
	\begin{split}
		\boldsymbol y^\top(A(\widehat G)-A(G^{\star}))\boldsymbol x&=\sum_{i=1}^p\sum_{\substack{u\in A,\ v\in V_i\\uv\notin E(G^{\star})}}
		(x_{v}y_{0}+y_i x_{u})
		-\sum_{uu'\in E(G^{\star}[A])}(x_u+x_{u'})y_0\\
		&\quad+
		\sum_{1\le i<j\le p}
		\sum_{\substack{v\in V_i,\ w\in V_j\\vw\notin E(G^{\star})}}
		(y_i x_w+y_jx_v)
		-\sum_{i=1}^p\sum_{vw\in E(G^{\star}[V_i])}y_i(x_v+x_w).
	\end{split}
\end{equation}
There are exactly $M$ missing $A$--$B$ pairs and $I$ edges inside $A$, and the support decomposition gives
\[
M=\sum_{u\in A^+}(D-d_B(u))+D|A^0|.
\]
For $u\in A^+$, \cref{lem:esitimate-x-y} (i) applies.  For $u\in A^0$, we have
$x_u=0$ and $e_{G^{\star}}(A^0,B)=e_{G^{\star}}(A^0,A^+)=0$; hence the internal $A^0$-edges have zero weight, while every missing $A^0$--$B$ pair contributes $x_v\ge(\alpha+\beta)/(1+\beta)-12\varepsilon^3>\varepsilon^2/1000$.  Applying the rooted inequality separately on $A^+$ gives
\begin{equation}\label{eq:nonright-A-payment}
 \begin{aligned}
 &\sum_{i=1}^p\sum_{\substack{u\in A,\ v\in V_i\\uv\notin E(G^{\star})}}
 (x_{v}y_{0}+y_i x_{u})
 -\sum_{uu'\in E(G^{\star}[A])}(x_u+x_{u'})y_0\\
 &\ge \left(1+\frac{\varepsilon^2}{1000}\right)\sum_{u\in A^+}(D-d_B(u))
      -2e(G^{\star}[A^+])
      +\frac{\varepsilon^2}{1000}D|A^0|\\
 &\ge \frac{\varepsilon^2}{1000}\left(\sum_{u\in A^+}(D-d_B(u))+D|A^0|\right)
 =\frac{\varepsilon^2}{1000}M.
 \end{aligned}
\end{equation}
Furthermore, by \eqref{eq:nonright-B12}, every $v\in B_2$ has more than $|A|/(10^4p^2f^2)$ non-neighbours in
$A$, and hence
\begin{equation*}
        M\ge\frac{|A|}{(10^4p^2f^2)}|B_2|.
\end{equation*}

Consider next the internal edges of the $V_i$ having both endpoints outside
$L$.  By \cref{lem:esitimate-x-y} (iii) and \Cref{lem:nonright-B2-internal}\textnormal{(iii)}, their total negative contribution is
at most
\begin{equation}\label{eq:nonright-nonL-negative}
        3(\Lambda_FD^{s_F}+f)|B_2\setminus L|.
\end{equation}
Combining \eqref{eq:nonright-A-payment}--\eqref{eq:nonright-nonL-negative},
we obtain that, if $B_2\ne\emptyset$,
\begin{equation}\label{eq:nonright-B2-positive}
 \frac{\varepsilon^2}{1000}M-3(\Lambda_FD^{s_F}+f)|B_2\setminus L|\ge\left(\frac{\varepsilon^2|A|}{1000(10^4p^2f^2)}-3\Lambda_FD^{s_F}-3f\right)|B_2|>0.
\end{equation}
Indeed, $|A|=n-D\ge\kappa_Fn^{s_F}$ by
\Cref{lem:endpoint-buffer-exclusion}.  From \eqref{eq:small-constant-choice},
\[
 \kappa_F\ge2\zeta_F
 =2\cdot10^6p(10^4p^2f^2)
   (\Lambda_F+C_\partial+f+2)\varepsilon^{-2},
\]
and therefore
\begin{align*}
 \frac{\varepsilon^2|A|}{1000(10^4p^2f^2)}\ge\frac{\varepsilon^2\kappa_F}{1000(10^4p^2f^2)}n^{s_F}\ge2000p(\Lambda_F+C_\partial+f+2)n^{s_F}>10(\Lambda_F+f+1)n^{s_F}.
\end{align*}
Since $D^{s_F}\le n^{s_F}$ and $n^{s_F}\ge1$, the coefficient in
\eqref{eq:nonright-B2-positive} is bounded below by
\begin{equation*}
	\bigl(10(\Lambda_F+f+1)-3\Lambda_F-3f\bigr)n^{s_F}	=(7\Lambda_F+7f+10)n^{s_F}>0.
\end{equation*}
This verifies the strict positivity in \eqref{eq:nonright-B2-positive}
without any further choice of constants.

It remains to pay for the internal edges that touch $L$.  Fix
$v\in L\cap V_i$.  By \Cref{lem:nonright-K-empty},
$d_{V_i}(v)<(10^4p^2f^2)\varepsilon^3D$.  Also,
$|V_i|\le D/p+2\varepsilon^3D$ by
\eqref{eq:nonright-basic-defect}, while the definition of $L$ gives
$d_B(v)\le(1-1/p-40(10^4p^2f^2)\varepsilon^3)D$.  The exact missing-crossing count satisfies
\begin{equation}\label{eq:nonright-L-missing-cross}
 \sum_{j\in [p]\setminus\{i\}}(|V_j|-d_{V_j}(v))=D-|V_i|-d_B(v)+d_{V_i}(v)\ge d_{V_i}(v)+(40(10^4p^2f^2)-2)\varepsilon^3D.
\end{equation}
Let $P_L$ denote the positive contribution in
\eqref{eq:nonright-bilinear-expansion} from missing crossing pairs having at
least one endpoint in $L$, and let $N_L$ denote the negative contribution from
internal $V_i$-edges having at least one endpoint in $L$.  Each missing crossing edge is counted at most twice when
\eqref{eq:nonright-L-missing-cross} is summed over $v\in L$, and each such edge
has weight at least $2/5$ by \cref{lem:esitimate-x-y} (ii).  If $Q_L$ denotes
the number of these missing crossing pairs and $T_L$ the triple sum below,
then $T_L\le2Q_L$, so their weighted positive contribution is at least
$(2/5)Q_L\ge T_L/5$.  Hence
\begin{equation*}
 P_L\ge\frac15
 \sum_{i=1}^p\sum_{v\in L\cap V_i}
 \sum_{j\in[p]\setminus\{i\}}(|V_j|-d_{V_j}(v)).
\end{equation*}
On the other hand, every internal edge has weight at most $3$ by
\cref{lem:esitimate-x-y} (iii), and an internal edge touching $L$ is
counted at least once by the sum of the internal degrees of vertices in $L$.
Thus
\begin{equation}\label{eq:nonright-NL-upper}
        N_L\le3\sum_{i=1}^p\sum_{v\in L\cap V_i}d_{V_i}(v).
\end{equation}
Writing
$R_L:=\sum_{i=1}^{p}\sum_{v\in L\cap V_i}d_{V_i}(v)$, equations
\eqref{eq:nonright-L-missing-cross}--\eqref{eq:nonright-NL-upper} give
\begin{align}
 P_L-N_L
 &\ge \frac15R_L
   +\frac{40(10^4p^2f^2)-2}{5}\varepsilon^3D|L|-3R_L\nonumber\\
 &=\frac{40(10^4p^2f^2)-2}{5}\varepsilon^3D|L|
   -\frac{14}{5}R_L\nonumber\\
 &>\frac{26(10^4p^2f^2)-2}{5}\varepsilon^3D|L|
 >5(10^4p^2f^2)\varepsilon^3D|L|,
 \label{eq:nonright-L-positive}
\end{align}
where $R_L<(10^4p^2f^2)\varepsilon^3D|L|$ and the last inequality uses
$(10^4p^2f^2)>2$.

The negative edges paid for in \eqref{eq:nonright-nonL-negative} have both
endpoints outside $L$, whereas those counted by $N_L$ have at least one
endpoint in $L$; hence these two negative contributions are disjoint.  Likewise,
the payment in \eqref{eq:nonright-A-payment} uses missing $A$--$B$ pairs,
whereas $P_L$ uses missing crossing pairs between distinct $V_i$, so no
positive term is used twice.

The positive crossing terms not incident with $L$ have not been used.  Hence
\eqref{eq:nonright-bilinear-expansion},
\eqref{eq:nonright-B2-positive}, and
\eqref{eq:nonright-L-positive} imply that the left side of
\eqref{eq:nonright-envelope-bilinear} is strictly positive whenever
$B_2\cup L\ne\emptyset$.  This contradicts
\eqref{eq:nonright-envelope-bilinear}.  Therefore
\begin{equation}\label{eq:nonright-B2-L-empty}
        B_2=L=\emptyset.
\end{equation}

Now, \Cref{lem:nonright-K-empty} gives $K=\emptyset$, so
$V_i'=V_i$; and \eqref{eq:nonright-B2-L-empty} gives $B_1=B$, hence
$B'_{1,i}=V_i$.  By \cref{lem:non-right-structural} (iii), every $V_i$ is
independent.  Returning to \eqref{eq:nonright-bilinear-expansion}, all negative
terms inside the $V_i$ vanish.  Therefore
\begin{equation*}
 0\ge \boldsymbol y^\top(A(\widehat G)-A(G^{\star}))\boldsymbol x\ge \frac{\varepsilon^2}{1000}M+\frac25\sum_{1\le i<j\le p}\bigl(|V_i||V_j|-e_{G^{\star}}(V_i,V_j)\bigr).
\end{equation*}
Both terms on the last line are non-negative.  Hence $ M=0,$ $e_{G^{\star}}(V_i,V_j)=|V_i||V_j|$ for $i\ne j$.  In particular, $A^0=\emptyset$, since every vertex of $A^0$ misses all $D$ vertices of $B$.  By \eqref{eq:nonright-root-defect}, $M=0$ also gives $I=0$.  Thus, $A$ and all
$V_i$ are independent, all $A$--$B$ pairs are edges, and all pairs between
distinct $V_i,V_j$ are edges.  Consequently
\begin{equation*}
	G^{\star}=\widehat{G}=|A|K_1\vee K_{|V_1|,\ldots,|V_p|}.
\end{equation*}
Finally, upper-tail extremality gives
$\rho(G^{\star})\ge\rho(S_{n,D,p})$, while
\Cref{lem:model-spectral-comparison} (ii) gives
$\rho(\widehat G)\le\rho(S_{n,D,p})$, with equality only for the balanced
choice of $|V_1|,\ldots,|V_p|$.  Hence, equality holds and
$G^{\star}\cong S_{n,D,p}$.
\end{proof}

The preceding lemma yields the following proposition.
\begin{proposition}\label{prop:nonright-range-exact}
Let $G^{\star}\in\mathfrak U_n(\Delta,\zeta_F;F)$ be an upper-tail extremal graph and put $D=\Delta(G^{\star})$.  If $\left\lceil\frac{p}{p+1}n\right\rceil\le D<\left(1-\frac{\varepsilon^2}{200}\right)n$ and $n-D\ge\kappa_Fn^{s_F}$, then $G^{\star}\cong S_{n,D,p}$ for all sufficiently large $n$.
\end{proposition}
\begin{proof}
The hypotheses are precisely those of \Cref{lem:nonright-cleaning}, which gives the asserted isomorphism.
\end{proof}

\subsection{\texorpdfstring{Proof of \cref{thm:exact-color-critical-main}}{Proof of the main theorem}}
Together with \Cref{prop:right-range-exact} and \Cref{prop:nonright-range-exact}, we have the following proposition.
\begin{proposition}\label{prop:upper-tail-exact}
Let $G^{\star}\in\mathfrak U_n(\Delta,\zeta_F;F)$ be an upper-tail extremal graph, put $D=\Delta(G^{\star})$, and assume $n-D\ge\kappa_Fn^{s_F}$. Then $G^{\star}\cong S_{n,D,p}$.
\end{proposition}

\begin{proof}
If $D\ge(1-\varepsilon^2/200)n$, apply \Cref{prop:right-range-exact}; otherwise apply \Cref{prop:nonright-range-exact}.
\end{proof}

\begin{proof}[Proof of \cref{thm:exact-color-critical-main}]
	Let $G^{\star}$ be an arbitrary upper-tail extremal graph in $\mathfrak U_n(\Delta,\zeta_F;F)$ and put $D=\Delta(G^{\star})$.  The endpoint buffer gives $n-D\geq\kappa_Fn^{s_F}$.  By \Cref{prop:upper-tail-exact}, $G^{\star}\cong S_{n,D,p}$. Since $D\geq\Delta$ and $D\mapsto\rho(S_{n,D,p})$ is strictly decreasing by \Cref{lem:model-spectral-comparison}(iii), $\rho(G^{\star})=\rho(S_{n,D,p})\leq\rho(S_{n,\Delta,p})$. The graph $S_{n,\Delta,p}$ belongs to the upper-tail class, so the reverse inequality follows from upper-tail extremality.  Thus the upper-tail maximum equals $\rho(S_{n,\Delta,p})$, and equality forces $D=\Delta$.

	Every $n$-vertex $F$-free graph $G$ with $\Delta(G)=\Delta$ belongs to the upper-tail class.  Hence $\rho(G)\le \rho(S_{n,\Delta,p})$. If equality holds, then $G$ itself attains the upper-tail maximum and is therefore upper-tail extremal.  Applying \Cref{prop:upper-tail-exact} with $D=\Delta$ gives $G\cong S_{n,\Delta,p}$.
\end{proof}

\section{Concluding remarks}

The theorem gives the exact spectral form of the rooted high-maximum-degree phenomenon for color-critical forbidden graphs with $\chi(F)\ge4$. The proof combines the stability and decomposition-family estimates from \cite{paper1} with an upper-tail argument. The upper-tail formulation permits Perron-vector switchings while the endpoint buffer controls the maximum degree. The reduction to a component attaining the spectral radius, followed by the multipartite envelope comparison, forces the rooted model. Thus, the edge theorem of Huo and Yuan \cite{Huo-Yuan} and the stability result of \cite{paper1} are upgraded to an exact spectral statement.

The case $\chi(F)=3$ is not covered here. Then $S_{n,\Delta}^{(2)}=K_{n-\Delta,\Delta}$, and the right-endpoint comparison has a square-root character rather than the linear expansion used for $p\ge2$. Moreover, the relevant boundary scale in this case need not be governed by the exponent $s_F$ used above; a different exponent may arise from the corresponding boundary family. Thus, the three-chromatic case differs from the case $\chi(F)\ge4$ both in the endpoint analysis and, potentially, in the boundary scale.

\begin{conjecture}\label{conj:three-chromatic-exact}
Let $F$ be a connected color-critical graph with $\chi(F)=3$. Then there is an exponent $\tau_F\in[0,1)$, not necessarily equal to $s_F$, such that, for all sufficiently large $n$, the unique spectral maximizer among $n$-vertex $F$-free graphs with maximum degree $\Delta$ in the range $\left\lceil\frac n2\right\rceil\le \Delta\le n-\Theta(n^{\tau_F})$ is $K_{n-\Delta,\Delta}$.
\end{conjecture}

\section*{Declaration of competing interest}
The author declares no competing interests.

%\section*{Declaration of generative AI and AI-assisted technologies in the manuscript preparation process}
%The author used AI-assisted tools only for language editing, editorial organization, and LaTeX consistency checks.  The author reviewed the manuscript and is responsible for all mathematical content.

\section*{Acknowledgments}
The author would like to thank Professor Yongtang Shi and Professor Shuchao Li for their valuable comments and suggestions on this manuscript.

\end{document}